\documentclass[12pt]{amsart}
\usepackage[margin=1in]{geometry}  % set the margins to 1in on all sides

\usepackage{graphicx}              % to include figures
\usepackage{amsmath}               % great math stuff
\usepackage{amsfonts}              % for blackboard bold, etc
\usepackage{amsthm}                % better theorem environments
\usepackage{color}
\usepackage{tabulary}
\usepackage{caption}
\usepackage{subcaption}
\usepackage{amssymb}
\usepackage{todonotes}
\theoremstyle{plain}
\usepackage{mathptmx}
\usepackage{tikz-cd} 
\usepackage{mathrsfs}
\usepackage{hyperref}
\usepackage{MnSymbol}

\newenvironment{theorem}[2][Theorem]{\begin{trivlist}
\item[\hskip \labelsep {\bfseries #1}\hskip \labelsep {\bfseries #2}]}{\end{trivlist}}

\newtheorem{thm}{Theorem}[section]

\newtheorem{lem}[thm]{Lemma}

\newtheorem{prop}[thm]{Proposition}
\newtheorem{cor}[thm]{Corollary}

\newtheorem{defn}[thm]{Definition}
\newtheorem{rmk}[thm]{Remark}

\newtheorem{ques}[thm]{Question} 

\theoremstyle{definition}
\newtheorem{thmx}{Theorem}

\newcommand{\fakeenv}{} %%% prints the emptystring

\newenvironment{restate}[2]  %%% restate takes two arguments 
{ 
 \renewcommand{\fakeenv}{#2} %%% So now \fakeenv prints #2
 \theoremstyle{plain} 
 \newtheorem*{\fakeenv}{#1~\ref{#2}} %%% so now #2 is the name of a
 \begin{\fakeenv}
}
{
 \end{\fakeenv}
}

\DeclareMathOperator{\id}{id}

\DeclareMathOperator{\PR}{PSL_2(\RR)}
\DeclareMathOperator{\PC}{PSL_2(\CC)}

\DeclareMathOperator{\LL}{\mathcal{L}}
\DeclareMathOperator{\GG}{\mathscr{G}}
\DeclareMathOperator{\Aut}{Aut}

\DeclareMathOperator{\Homeop}{Homeo^+}

\DeclareMathOperator{\COL}{COL}

\DeclareMathOperator{\Stab}{Stab}
\DeclareMathOperator{\Fix}{Fix}

\newcommand{\NN}{\mathbb{N}}      %
\newcommand{\ZZ}{\mathbb{Z}}      % for Integers
\newcommand{\QQ}{\mathbb{Q}}  
\newcommand{\RR}{\mathbb{R}}      % for Real numbers
\newcommand{\CC}{\mathbb{C}}      % for Real numbers
\newcommand{\DD}{\mathbb{D}}      % for Integers
\newcommand{\HH}{\mathbb{H}}      % for Integers
\newcommand{\hol}{\operatorname{hol}}     
\newcommand{\intr}{\operatorname{int}} 
\newcommand{\extr}{\operatorname{ext}} 
\newcommand{\stab}{\operatorname{Stab}}
\newcommand{\x}{\overline{x}}

\begin{document}

\title[Laminar groups]{
Characterization of Fuchsian groups as laminar groups $\&$ the structure theorem of hyperbolic $2$-orbifolds} % \\
%\small Caractérisation des groupes fuchsiens en termes de laminage invariantes}

\author{Hyungryul Baik \& KyeongRo Kim}
\address{Department of Mathematical Sciences, KAIST,  
291 Daehak-ro, Yuseong-gu, Daejeon 34141, South Korea }
\email{hrbaik@kaist.ac.kr \& cantor14@kaist.ac.kr}

%%%%%%%%%%%%%%%
\begin{abstract}
Thurston and Calegari-Dunfield showed that the fundamental group of some tautly foliated hyperbolic $3$-manifold acts on the circle in a distinctive way that the action preserves some structure of $S^1$, so-called a circle lamination. Indeed, a large class of Kleinian groups acts on the circle preserving a circle lamination. In this paper, we are concerned with the converse problem that a group acting  on the circle with at least two invariant circle laminations is Kleinian. We prove that a subgroup of the orientation preserving circle homeomorphism group is a Fuchsian group whose quotient orbifold is not a geometric pair of pants (or turnover) if and only if it preserves three circle laminations with a certain transversality. This is the complete generalization of the previous result of the first author which is proven under the assumption that the subgroup is discrete and torsion-free. On the way to our main result, we also show a structure theorem (generalized pants decomposition) for complete 2-dimensional hyperbolic orbifolds, including the case of infinite type, which is of independent interest. 
\end{abstract}
\maketitle
%\begin{abstract}[R\'{E}SUM\'{E}]
%Dans \cite{BaikFuchsian}, le premier auteur a proposé un programme pour étudier des groupes agissant fidèlement sur $S^1$ en termes du nombre de laminations invariantes.
% Dans cet article,  il a également montré qu'un sous-groupe discret sans torsion du groupe d'homéomorphismes du cercle préservant l'orientation $\Homeop(S^1)$ est
 %un groupe fuchsien si et seulement s'il admet trois laminations invariantes très complètes différentes qui satisfont une sorte de condition de transversalité. 
%Le but de cet article est de généraliser le résultat de \cite{BaikFuchsian} au cas des groupes fuchsiens généraux sans supposer que ces groupes sont sans torsion, tout en fournissant
 %des arguments plus complets et détaillés. De plus, nous donnons un théorème de structure pour des orbifolds hyperboliques de dimension $2$, y compris le cas de type infini, 
%pour construire des laminations sur les orbifolds.
%\end{abstract}

%\tableofcontents

\section{Introduction}
\subsection{Background and Motivation}

The classification of a group action on a manifold is a fundamental question in mathematics. Along this line, one of the outstanding studies is the Zimmer's program. One part of Zimmer's program is to understand the actions of a lattice $G$ of a semi-simple Lie group $\Gamma$ on a compact manifold $M$.
See \cite{Fisher11} for the survey of Zimmer's program.

The most accessible case is that $M$ is dimension one, namely, $M=S^1$.  When $\Gamma$ is of real rank at least two, Ghys~\cite{Ghys99} gave a complete characterization for homomorphisms from $G$ to the orientation preserving circle homeomorphism group $\Homeop(S^1)$. This result says that the circle action of a higher rank lattice is rigid. 

On the contrary, in the rank one case, the circle actions of $G$ are more flexible. For instance, when  $\Gamma=\PR$ and $G$ is a closed hyperbolic surface group, the space of $G$-actions is studied by Mann~\cite{Mann15} and Mann-Wolff~\cite{MannWolff19} in terms of the Euler number. They showed the equivalency of the rigidity and the geometricity of the $G$-actions. This implies that non-geometric actions have no rigidity in some reasonable sense. Hence, in general, given any subgroup $H$ of $\Homeop(S^1),$ it is difficult to determine whether $H$ comes from a lattice of a semi-simple Lie group of real rank one.

Nonetheless, in the case of $\PC$, we can observe that a large class of lattices act on the circle in a distinctive way that the actions preserve some structure of $S^1$, so-called a \emph{circle lamination}. This feature is first observed by Thurston along the proof of the hyperbolization theorem for surface bundles over $S^1$. More generally, in \cite{Thurston97}, Thurston showed that the fundamental group of a tautly foliated hyperbolic $3$-manifold faithfully acts on $S^1$, preserving a pair of circle laminations. In various circumstance,  Thurston's theorem has been generalized by many authors, e.g. \cite{Calegari00}, \cite{CalDun03}, \cite{Calegari06}, \cite{fenley2012ideal},  \cite{FrankelSchleimerSegerman19} and so on.  Also,  we can see that almost all torsion-free lattices in $\PR$ which are the fundamental groups of hyperbolic surfaces of finite type act on $S^1$ preserving infinitely many distinct circle laminations.

Hence, it is natural to ask whether a subgroup of $\Homeop(S^1)$ preserving several circle laminations, so-called a \emph{laminar group}, is a discrete subgroup of $\PC$, so-called a \emph{Kleinian group}. Also, we may ask that a laminar group with infinite many invariant circle laminations is a \emph{Fuchsian group} which is a discrete subgroup of $\PR$. 
In this paper, we are mainly concerned with the second question. We will show that three is the enough number of invariant circle laminations to characterize Fuchsian groups.

\subsection{Main results}
To present the precise statements for the questions and our main theorems, we roughly introduce basic notions about laminations.
Let $\HH^2$ be the hyperbolic plane. For now, we use the Poincar\'{e} disk for $\HH^2$ which is the open unit disk in the complex plane $\CC$. The Gromov boundary of $\HH^2$ is the unit circle $S^1$ in $\CC$. Hence, we can think of the orientation preserving isometry group $\PR$ for $\HH^2$ as a subgroup of $\Homeop(S^1)$ since each element of $\PR$ is continuously and uniquely extended to $S^1$.  

A given geodesic in $\HH^2$ joins two points in $S^1$. Conversely, given two distinct points in $S^1$, there is a unique geodesic  in $\HH^2$ connecting them. Hence, the set $g(\HH^2)$ of geodesics in $\HH^2$ is parametrized by the set $\mathcal{M}$ of subsets of $S^1$ with caldinality of $2$. We denote the parametrization from $g(\HH^2)$ to $\mathcal{M}$ by $\epsilon$. 
 
A \emph{geodesic lamination} $\Lambda$ in $\HH^2$ is a non-empty subcollection of $g(\HH^2)$ such that $\bigcup \Lambda$ is a closed subset in $\HH^2$ and each pair of elements of $\Lambda$ are disjoint. We say that for each component of $\HH^2 \setminus \Lambda$, the closure of the component in $\overline{\HH^2}$  is a \emph{non-leaf gap} of $\Lambda$. A geodesic lamination is \emph{quite full} if every non-leaf gap is either an ideal polygon or a crown. Here, a \emph{crown} $\GG$ is a non-leaf gap having a sequence $\{t_i\}_{i\in \NN}$ in $S^1$ such that $\epsilon^{-1}(\{t_i,t_{i+1}\})$ is a boundary geodesic of $\GG$ and $\{t_i\}_{i\in \NN}$ accumulates to a unique point $p$. The point $p$ is called the \emph{pivot} of the crown. In particular, a quite full geodesic lamination without crowns is said to be \emph{very full}.  

A \emph{circle lamination} $\LL$ is a subcollection of $\mathcal{M}$ if $\LL=\epsilon(\Lambda)$ for some geodesic lamination $\Lambda$ in $\HH^2$. For convenience, we think of  circle laminations as geodesic laminations in $\HH^2$. Given a collection $\mathcal{C}=\{\LL_\alpha\}_{\alpha\in \Gamma}$ of circle laminations, we denote by $\Aut(\mathcal{C})$ the set of all elements in $\Homeop(S^1)$ preserving $\mathcal{C}$, namely, $\Aut(\mathcal{C}):=\{g\in \Homeop(S^1): g(\LL_\alpha)=\LL_\alpha \text{for all $\alpha\in \Gamma$} \}$. We also call $\Aut(\mathcal{C})$ the \emph{laminar automorphism group} of $\mathcal{C}$. A subgroup of $\Homeop(S^1)$ is \emph{laminar} if it is a subgroup of the laminar automorphism group of some circle lamination. The  laminar groups were first introduced in  \cite{calegari2001foliations}.

Now, we are ready to give the precise statements for the questions explained in the previous section.

\begin{ques}\label{firstQ}
Let $\mathcal{C}=\{\LL_\alpha \}_{\alpha \in \Gamma}$ be a collection of quite full circle laminations whose endpoint sets are pairwise disjoints, namely, $(\bigcup \LL_\alpha) \cap (\bigcup \LL_\beta)=\emptyset $ for any $\alpha\neq \beta$. Is every subgroup of $\Aut(\mathcal{C})$  isomorphic to a discrete subgroup of $\PC$ ? 
\end{ques}

\begin{ques}\label{secondQ}
Are the Fuchsian groups the only case to act faithfully on $S^1$ preserving an infinite collection of quite full circle laminations whose endpoint sets are pairwise disjoint? In other words, if  $\mathcal{C}$ is an infinite collection of quite full circle laminations whose endpoint sets are pairwise disjoint, is every subgroup of $\Aut(\mathcal{C})$ conjugated  to a Fuchsian group as subgroups in $\Homeop(S^1)$?
\end{ques}

Indeed, Question~\ref{firstQ} was proposed in \cite{BaikFuchsian} and \cite{alonso2019laminar} in terms of pants-like $\COL_n$ groups. See Section~\ref{sec:laminar groups with two lami} for the definition of a pants-like $\COL_n$ group. According to \cite{BaikJungKim22}, given a finite collection $\mathcal{C}=\{\LL_\alpha \}_{\alpha \in \Gamma}$ of quite full circle laminations whose endpoint sets are pairwise disjoint, for any subgroup $G$ of $\Aut(\mathcal{C})$, we can construct a pants-like collection $\mathcal{C}'=\{\LL_\alpha' \}_{\alpha \in \Gamma}$ of circle laminations preserved by $G$. Hence, even thought the definition of a pants-like $\COL_n$ group looks artificial, this is a weaker condition than preserving $n$ quite full circle laminations with the disjoint endpoint set condition. Therefore, Question~\ref{firstQ} assumes conditions which are stronger than in \cite{BaikFuchsian} and weaker than in \cite{alonso2019laminar}.

In this paper,  we show the following corollary which is the affirmative answer for Question~\ref{secondQ}. Therefore, three is the enough number of invariant laminations to characterize the circle actions of Fuchsian groups.
\begin{cor}\label{quiteFullVer}
If  $\mathcal{C}=\{\LL_1,\LL_2,\LL_3\}$ is a collection of quite full circle laminations whose endpoint sets are pairwise disjoint, every subgroup of $\Aut(\mathcal{C})$ is conjugated  to a Fuchsian group as subgroups in $\Homeop(S^1)$. Conversely, given a Fuchsian group $G$, we can construct three quite full circle laminations, with the disjoint endpoint condition, preserved by $G$ if $\HH^2/G$ is not a geometric pair of pants.
\end{cor}

In fact, we show a stronger statement than the first statement of Corollary~\ref{quiteFullVer} as follows.
\begin{thmx}\label{A}
  Every pants-like $\COL_3$ group is a convergence group. 
  \end{thmx}
Thanks to the following famous theorem,  Theorem~\ref{A} is equivalent to Corollary~\ref{charFuchsian} which we desired. 
\begin{thm}[ Convergence group theorem (\cite{Tukia88},\cite{Gabai91},\cite{Casson1994}, cf. \cite{Hinkkanen1990} for the indiscret case)] \label{thm:convergenceGroupThm}~
A subgroup of $\Homeop(S^1)$ is a convergence group if and only if it is a Fuchsian group. 

\end{thm}

\begin{cor}\label{charFuchsian}
   Every pants-like $\COL_3$ group $G$ is a discrete M\"obius group, namely,  there is an element $\phi$ in $\Homeop(S^1)$ such that $\phi G \phi^{-1}$ is a Fuchsian group. 
\end{cor}

Conversely, we also show the following theorem which is weaker than the second statement of Corollary~\ref{quiteFullVer}.

\begin{thmx}\label{B}
  Let $G$ be a Fuchsian group. Suppose that $\HH/G$ is not a geometric pair of pants. Then $G$ is a pants-like $\COL_3$ group.  
\end{thmx}

In the proof of Theorem~\ref{B}, by removing the gray geodesics in the pieces of  Figure~\ref{fig:monogonhole} type, we can modify Lemma~\ref{lem : a lamination in a geometric pair of pants} for quite full laminations. Eventually, we can strengthen Theorem~\ref{B} in order to get the second statement of Corollary~\ref{quiteFullVer}.

The combination of Theorem~\ref{A} and Theorem~\ref{B} is the complete generalization of the following theorem shown in \cite{BaikFuchsian}.
\begin{thm}\label{BaikThm}
Let $G$ be a torsion-free discrete subgroup of $\Homeop(S^1)$. Then $G$ is a pants-like $\COL_3$ group if and only if G is a Fuchsian group whose quotient is not the thrice-punctured sphere.
\end{thm}

Note that Theorem~\ref{A} and Theorem~\ref{B} are not immediately implied by Theorem~\ref{BaikThm}. First, in \cite{BaikFuchsian}, the proof of the only-if-direction depends on the fact that every non-trivial element has a fixed point which is implied by the torsion-free condition. Also, along the proof, discreteness was used in a lemma which is corresponding to Lemma~\ref{lem: having the convergence property} in this paper. Hence, we have to overcome these difficulties to prove Theorem~\ref{A}. For the if-direction, the construction of three different geodesic laminations contained some technical issues. Also, the Selberg's Lemma does not make it easy to construct the geodesic laminations on hyperbolic $2$-orbifolds since we also consider the infinite type orbifolds. Instead, we have to use the structure theorem (generalized pants decomposition) for hyperbolic $2$-orbifolds including the case of infinite type as follows.  

\begin{restate}{Theorem}{thm : straightening a pants decomposition}[The structure theorem for complete hyperbolic 2-dimensional orbifolds]
  Let $G$ be a non-elementary Fuchsian group and $S$ be the complete $2$-dimensional hyperbolic orbifold $\HH/G.$ Suppose that there is a pants decomposition $\mathcal{C}$ of $S.$ Let $\Lambda$ be the union of all geodesic realizations in $S$ of elements of $\mathcal{C}.$ Then $\Lambda \cup (\partial CH(L(G))/G)$ is a geometric pants decomposition of $CC(S).$
\end{restate}

This theorem is analogous of the structure theorem for complete hyperbolic surfaces of \'{A}lvarez and Rodr\'{\i}guez \cite{structuretheorem} and Basmajian and Šarić~\cite{BASMAJIAN_2017}. The pants decomposition theorem for hyperbolic orbifolds of finite type is well known. For instance, see Thurston~\cite[Chapter~13, ~Orbifolds]{Thurston80}. However, we could not find the structure theorem of hyperbolic $2$-orbifold including the case of infinite type in the literature. Hence, we collect and reorganize basic facts about the hyperbolic geometry of orbifolds, and prove the structure theorem for hyperbolic orbifolds for the sake of completeness. This is of the independent interest and could serve as a useful reference to the community (for instance, see \cite{biringer2021chabauty}).

Now, by Corollary~\ref{quiteFullVer}, the remaining case of Question~\ref{firstQ} is that $\mathcal{C}$ consists of at most two circle laminations. The case where $\mathcal{C}$ consists of two circle laminations is studied in \cite{BaikJungKim22}. In \cite{BaikJungKim22}, we  and Jung show that if $\mathcal{C}=\{\LL_1,\LL_2\}$ is a \emph{veering pair} which satisfies the disjoint endpoints condition and some extra conditions, then every subgroup in $\Aut(\mathcal{C})$ is a fundamental group of an irreducible $3$-orbifold. Moreover, if the $3$-orbifold is finite in some sense, then the orbifold admits a complete hyperbolic metric and so the fundamental group is Kleinian. Along the proof of this theorem, we use the following lemma to prove the classification theorem of the gap stabilizers. It is also the key lemma to prove Theorem~\ref{A}.

\begin{restate}{Lemma}{lem: taking a pre-approximation sequence}
  Let $G$ be a pants-like $\COL_2$ group and $\mathcal{C}=\{\LL_1, \LL_2\}$ be a pants-like collection for $G$. Suppose that we have  a sequence $\{x_n\}_{n=1}^{\infty}$ of elements of $S^1$ converging to $x\in S^1$ and a sequence $\{g_n\}_{n=1}^{\infty}$ of distinct elements of $G$ such that $\{g_n(x_n)\}_{n=1}^\infty$ converges to $x'\in S^1$. Then we can have one of the following cases:
\begin{enumerate}
\item there is a subsequence $\{g_{n_k}\}_{k=1}^{\infty}$ such that $g_{n_k}(x)=x'$ for all $k\in\NN$; 
\item there is a subsequence $\{g_{n_k}\}_{k=1}^{\infty}$ and a quasi-rainbow $\{I_k\}_{k=1}^\infty$ at $x'$ in $\LL_i$ for some $i\in \ZZ_2$ such that the sequence $\{(g_{n_k},I_k)\}_{k=1}^{\infty}$ is a pre-approximation sequence at $x'$;
\item  there is a subsequence $\{g_{n_k}\}_{k=1}^{\infty}$ and  a quasi-rainbow $\{I_k\}_{k=1}^\infty$ at $x$ in $\LL_i$ for some $i\in \ZZ_2$  such that the sequence $\{(g_{n_k}^{-1},I_k)\}_{k=1}^{\infty}$ is a pre-approximation  sequence at $x$. 
\end{enumerate} 
\end{restate}

Throughout the paper, we use basic notations and statements of \cite{Baik2019LaminarGA} and \cite{BaikFuchsian}. Especially, instead of the circle laminations, we use the notion of laminations systems  which are more proper to keep track of the dynamic of the circle actions. 

\subsection{Organization}

We give a brief review of necessary preliminaries in Section \ref{sec:prelim}. In Section \ref{sec:laminar groups with two lami}, we show that Möbius-like laminar groups with two invariant laminations satisfying some special conditions are discrete in $\Homeop(S^1).$ This implies that every pants-like $\COL_3$ group is discrete. In Section \ref{sec:main1}, we prove Theorem \ref{A}.  In Section \ref{sec:elementary facts of orbifolds}, we give a summary of elementary facts about the geometry and topology of $2$-dimensional hyperbolic orbifolds and for completeness, we prove some basic facts which are well-known among experts. In Section \ref{subsec:the structure theorem}, we show a structure theorem of complete $2$-dimensional hyperbolic orbifolds, Theorem \ref{thm : straightening a pants decomposition}, following the proof for hyperbolic surfaces in \cite{BASMAJIAN_2017}. Note that the structure theorem for complete hyperbolic surfaces was firstly proven by \'{A}lvarez and Rodr\'{\i}guez \cite{structuretheorem}, and after that, a geometric proof was suggested by Basmajian and Šarić\cite{BASMAJIAN_2017}. Using Theorem~\ref{thm : straightening a pants decomposition}, in Section~\ref{subsec:main2b} and Section~\ref{subsec:main2c}, we construct three different laminations in $S^1$ and prove Theorem~\ref{B}.

\vspace{.5cm} 
\noindent \textbf{Acknowledgement:} A large part of this paper is a part of the Ph.D. thesis of the second author. We thank Inhyeok Choi, Suhyoung Choi, Sang-hyun Kim, Seonhee Lim, Junghwan Park, Philippe Tranchida  for helpful discussions. 
Both authors were partially supported by Samsung Science \& Technology Foundation grant No. SSTF-BA1702-01.

%%%%%%%%%%%%%%%%%%%%%%%%%%%%%%%%%%%%%%%%%%%%%%%%%%%%%%%%%%%%%%%%%%%
\section{Basic notions and notations to study laminar groups}\label{sec:prelim} 
In first three sections, we review basic notions about Riemann surfaces and the hyperbolic geometry on orbifolds. In this part, to prove Theorem~\ref{thm : straightening a pants decomposition}, we adopt a proper collection of definitions. This is based on  \cite{ratcliffe2006foundations}, \cite{beardon1983geometry}, and \cite{kra1972automorphic}. Then, we recall circle laminations and lamination systems. Also, we summerize related notions and technical lemmas. This part is a summary of \cite{Baik2019LaminarGA}.

\subsection{The Riemann sphere and the hyperbolic plane}
Let $\hat{\CC}=\CC \cup \{ \infty \}$ be the Riemann sphere and $\hat{\RR}= \RR \cup \{\infty\}$ be the extended real line. We denote the upper half plane by  $\HH$ and the lower half plane by $\HH^*.$ The upper half plane $\HH$ has  a standard hyperbolic metric $\rho$, so-called the Poincaré metric. With this metric, $\hat{\RR}$ is the Gromov boundary $\partial \HH.$ On the other hand, a Cayley transformation $p$ expressed as
$$p(z)=i\frac{1+z}{1-z}$$ gives an identification between the Poincaré disk  $\DD$ and the upper half plan $\HH$ 
where $\DD$ is the unit open disk in $\CC$ with the standard hyperbolic metric.  

Recall that the automorphism groups of $\hat{\CC}$ and $\HH$ are $\PC$ and $\PR,$ respectively. The group $\PC$ acts on $\hat{\CC}$ by linear fractional transformations and especially the subgroup $\PR$ acts on $\hat{\CC}$ preserving $\hat{\RR}.$ Recall that the orientation preserving isometry group of $\HH$ is exactly $\PR.$ Note that the automorphism group of $\DD$ can be identified  with $\PR$, via $p,$ and so we think of the automorphism group of $\DD$ as $\PR.$

A discrete subgroup of $\PC$ is a \emph{Kleinian group}. Especially,  a Kleinian group in $\PR$ is a \emph{Fuchsian group}.  Let $G$ be a Fuchsian group. We denote the domain of discontinuity of $G$ on $\hat{\CC}$ by $\Omega(G)$ and the limit set of $G$ by $L(G). $ Note that $L(G)$ is a closed subset of $\hat{\RR}.$ Also, $\HH\subset \Omega(G).$ The Fuchsian group $G$ is \emph{of the first kind} if $L(G)=\hat{\RR},$ and \emph{of the second kind}, otherwise. In the first kind case,  $\Omega(G)$ has exactly two connected components $\HH$ and $\HH^*.$ If $L(G)$ is not $\hat{\RR}$ and is finite, then the cardinality of $L(G)$ is at most two and in this case we say that $G$ is  \emph{elementary}. Recall that every elementary Fuchsian group is virtually abelian and the converse is also true.  When $L(G)$ is not $\hat{\RR}$ and is infinite, the limit set $L(G)$ is a closed, perfect, nowhere dense subset of $\hat{\RR}.$ Hence, in the second kind case, $\Omega(G)$ has the only one connected component which is conformally equivalent to the Riemann sphere without a cantor set.
 
 \subsection{Riemann surfaces with branched coverings}
Let $G$ be a subgroup of $\PC$ and $\Omega$ be a non-empty $G$-invariant open set in $\hat{\CC}.$ Suppose that $G$ acts properly discontinuously on $\Omega.$ If $\Omega$ has more than two boundary points and $z_0$ is a point in $\Omega,$ then the \emph{stabilizer} $G_{z_0}$ defined by $$G_{z_0}=\{g\in G: g(z_0)=z_0\}$$ is a finite cyclic group which is generated by an elliptic element fixing $z_0$.  Moreover, we can take  a complex structure on the orbit space $\Omega/G$ so that the quotient map $q$ from $\Omega$ to $\Omega/G$ is holomorphic and a branched covering. In other words, for each $z_0\in \Omega$, there are charts $\varphi_1:U_1\to V_1$ on $\Omega$  and $\varphi_2: U_2 \to V_2$ on $\Omega/G$ with $q(U_1)\subset U_2$ such that $\varphi_1^{-1}(0)=z_0$ and $\varphi_2 \circ q \circ \varphi_1^{-1}(z)=z^m$ where $m$ is the cardinality of the stabilizer $G_{z_0}.$ The cardinality of the stabilizer $G_{z_0}$ is called the \emph{ramification index} of $q$ at the point $z_0,$ and  the point $z_0$ is a \emph{ramification point} if the ramification index is greater than $1.$ See \cite{kra1972automorphic} for more detaied discussion.

Let  $G$ be a non-elementary Fuchsian group. We write $\pi_G^d$ for the quotient map from $\Omega(G)$ to $\Omega(G)/G,$ and $\pi_G$ for the quotient map from $\HH$ to $\HH/G.$ The quotient spaces $\Omega(G)/G$ and $\HH/G$ with conformal structures defined as above are Riemann surfaces. Note that if $G$ is of the first kind, $\Omega(G)/G$ has exactly two connected components, and is connected, otherwise. If $G$ is of the second kind, there is a natural anti-conformal involution $J$ on $\Omega(G)/G$ induced by the complex conjugation on $\hat{\CC}.$ Then the map $J$ switches between $\HH/G$ and  $\HH^*/G,$ and fixes every point of $(\hat{\RR}-L(G))/G.$

\subsection{$2$-dimensional hyperbolic orbifolds.}
Let $S$ be a connected metric space. The space $S$ is called a \emph{$2$-dimensional hyperbolic orbifold} if there is a  subgroup $G(S)$ of $\PR$ and a non-empty $G(S)$-invariant simply connected open set $C(S)$ of $\HH$ such that $G(S)$ acts properly discontinuously on $C(S)$ and $C(S)/G(S)$ with the metric induced by $\rho$ is isometric to $S.$ We say that $S$ is \emph{complete} if $C(S)=\HH.$ 

Similarly, the space $S$ is  a \emph{$2$-dimensional hyperbolic orbifold with geodesic boundary} if there is a subgroup $G(S)$ of $\PR$ and a $G(S)$-invariant closed convex subset $C(S)$ of $\HH$ with non-empty interior whose boundary is a disjoint union of  bi-infinite geodesics such that $G(S)$ acts properly discontinuously on $C(S)$ and $C(S)/G(S)$ with the metric induced by $\rho$ is isometric to $S.$ We allow $C(S)$ to be $\HH.$ By definition, every $2$-dimensional hyperbolic orbifold with geodesic boundary is complete as a metric space. Note that $G(S)$  is uniquely determined up to conjugation by $\PR.$

Let $\pi_S$ be the quotient map from $C(S)$ to $C(S)/G(S).$ As discussed previously, when $S$ is a $2$-dimensional hyperbolic orbifold,  $\pi_S$ is a branched covering. Even though $S$ is a $2$-dimensional hyperbolic orbifold with geodesic boundary, the restriction $\pi_S|int(C(S))$ to the interior $int(C(S))$ of $C(S)$ is a branched covering onto $int(C(S))/G(S)$ of $S$ and this map can be also extended continuously onto the geodesic boundary $\partial C(S).$ Therefore, $\pi_S$ is a branched covering.    

Let $S$ be a $2$-dimensional hyperbolic orbifold with or without geodesic boundary. The set $C(S)$ is a universal orbifold covering space and $\pi_S$ is the orbifold covering map. We say that a point $x\in S$ is a \emph{cone point of order $n$} if $x$ is an image of a ramification point with ramification index $n$ and we denote the set of all cone points in $S$ by $\Sigma(S).$ Note that by definition, a ramification point is not in $\partial C(S)$ when $S$ is a $2$-dimensional hyperbolic orbifold with geodesic boundary.  

We define the \emph{geometric interior} $S^o$of $S$ to be the set $int(C(S))/G(S)-\Sigma(S).$ We call an element of the geometric interior a \emph{regular point.}
A point  in $S-int(C(S))/G(S)$ is called a \emph{boundary point} and we denote the set of  boundary points of $S$ by $\partial S.$ Note that when $S$ is a $2$-dimensional hyperbolic orbifold with geodesic boundary, $\partial S$ is a disjoint union of geodesics. 

Recall that if a bordered Riemann surface $X$ has an open subset which is conformally equivalent to a unit disk without center, then we say that $X$ has a \emph{puncture}. Moreover, when $X$ has a puncture, we can uniquely extend $S$ to a bigger Riemann surface $\overline{X}$ by adding the center. In this sense, we think of a puncture as the added center of the disk. If  a point $x$ is in $\partial S \cup \Sigma(S)$ or a puncture of $S$, $x$ is called a \emph{geometric boundary point} and $\Delta S$ denotes the set of all geometric boundary points of $S.$

Let $S$ be a hyperbolic surface with geodesic boundary. In our definition, $S$ is a $2$-dimensional hyperbolic orbifold with geodesic boundary and without cone points. As usual, a \emph{geodesic lamination} on $S$ is a nonempty closed subset of $S$ which is a disjoint union of  simple geodesics, and the geodesics composing the lamination are called \emph{leaves}. We can also define geodesic laminations on $2$-dimensional hyperbolic orbifolds with geodesic boundary as follows. 

\begin{defn}
Let $S$ be a $2$-dimensional hyperbolic orbifold with geodesic boundary. Fix the corresponding group $G(S).$  Then a closed subset  $\Lambda$ of $S$ is called a \emph{geodesic lamination} on $S$ if $\pi_{S}^{-1}(\Lambda)$ is a geodesic lamination of $C(S),$ namely $\pi_S^{-1}(\Lambda)$  is a closed subset of $C(S)$ which consists of pairwise disjoint bi-infinite geodesics. In particular, if $\pi_{S}^{-1}(\Lambda)=G(S)\cdot \ell$ for some bi-infinite geodesic $\ell$ in $\HH,$ then we call $\Lambda$ a \emph{simple geodesic} in $S$.
\end{defn}

When $S$ is a complete $2$-dimensional hyperbolic orbifold and $\Lambda$ is a geodesic lamination of $S,$ we think of $\Lambda$ as a closed subset of $S$ which is a disjoint union of simple geodesics. Hence, the geodesics composing the lamination are called \emph{leaves}. As $\pi_{S}^{-1}(\Lambda)$ is a geodesic lamination of $C(S)$, $\pi_{S}^{-1}(\Lambda)$ is a geodesic lamination on $\HH.$ Since given two points in $\partial \HH,$ there is a unique bi-infinite geodesic which has the points as end points, and vice versa, we may think of $\pi_{S}^{-1}(\Lambda)$ as a collection of two-point sets of $S^1$ corresponding to leaves of $\pi_{S}^{-1}(\Lambda).$ Moreover, since $G(S)$ acts on $\hat{\RR}$ and $G(S)$ permutes leaves of $\pi_{S}^{-1}(\Lambda),$ $G(S)$ acts on the collection of two-point sets of $S^1$ and permutes two-point sets. In the following sections, we will formulate precisely the collection of two-point sets corresponding to $\pi_{S}^{-1}(\Lambda)$ in terms of the topology of $S^1$ and group actions on such collections.

\subsection{The topology of $S^1$}
Let $S^1$ be the unit circle on $\CC$ which is the boundary of the Poincare disk $\DD.$  Now we introduce a circular order to $S^1$ which is helpful to discuss the topology of $S^1.$ 

We start with a set $X.$ For each $n\in \NN$ with $n\geq2$, we write $\Delta_n(X)$ for the set $$\{ (x_1, \cdots, x_n)\in X^n : x_i=x_j \ \text{for some} \ i \neq j \}.$$ A \emph{ circular order} of $X$  means a map $\varphi$ from $X^3$ to $\{-1,0,1\}$ satisfying following properties :
 \begin{enumerate}
\item[(DV)] $\varphi^{-1}(0)=\Delta_3(X)$, and 
\item[(C)] $\varphi(x_1, x_2, x_3)-\varphi (x_0,x_2, x_3)+\varphi(x_0,x_1,x_3)-\varphi(x_0, x_1, x_2)=0$ for all $x_0,x_1,x_2,x_3 \in X.$  
\end{enumerate}

Let $n\in \NN$ with $n\geq 3$. A $n$-tuple $(x_1, \cdots, x_n)$ in $(S^1)^n-\Delta_n(S^1)$ is called \emph{positively oriented} in $S^1$ if for each $i \in \{2, \cdots, n-1\}$, $p(x_1^{-1}x_i)<p(x_1^{-1}x_{i+1})$ where $p$ is the Cayley transformation. Similarly, a $n$-tuple $(x_1, \cdots, x_n)$ in $(S^1)^n-\Delta_n(S^1)$ is called \emph{negatively oriented} in $S^1$ if for each $i \in \{2, \cdots, n-1\}$, $p(x_1^{-1}x_{i+1})<p(x_1^{-1}x_{i}).$ 
Let $\varphi$ be a map from $(S^1)^3$ to $\{-1,0,1\}$ defined by
\begin{equation*} \varphi(x)=
\begin{cases}
-1 & \text{if } \  x \ \text{ is negatively oriented,}\\
 0 & \text{if } \  x \in\Delta_3(S^1) ,\\ 
 1 & \text{if } \ x \ \text{is positively oriented.}
\end{cases} \end{equation*}
We can see that  $\varphi$ is a circular order of $S^1.$ From now on, the circular order of $S^1$ means the map $\varphi.$

We call a nonempty proper connected open subset of $S^1$ an
\emph{open interval} in $S^1$.  Let $u, v$ be elements of $S^1$. Then, we define the subset $(u,v)_{S^1}$ of $S^1$ as follows:
\begin{enumerate}
\item If $u\neq v$, $(u,v)_{S^1}$ is the set $(u,v)_{S^1}=\{ p\in S^1 : \varphi(u,p,v)=1 \}$ and is called a \emph{nondegenerate open interval} in $S^1$.
\item If $u= v$,  $(u,v)_{S^1}$ is the set $S^1-\{ u \}$ and is called a \emph{degenerate open interval} in $S^1$.
\end{enumerate} 
In particular, if $(u,v)_{S^1}$ is a nondegenerate open interval, then we denote $(v,u)_{S^1}$ by $(u,v)_{S^1}^*$, and call it the \emph{dual interval} of $(u,v)_{S^1}$. We can see that the set of all nondegenerate open intervals of $S^1$ is a basis for the topology of $S^1$ which is induced from the standard topology of $\CC$.  
For convenience,  when $(u,v)_{S^1}$ is a nondegenerate open interval, we use the following notations: $[ u, v )_{S_1}=\{u\} \cup (u,v)_{S^1}$; $(u,v]_{S^1}= (u,v)_{S^1} \cup \{ v \}$; and
 $[u,v]_{S^1}= (u,v)_{S^1} \cup \{u, v \}$.
 
 \begin{rmk}
 Indeed, in \cite{Baik2019LaminarGA}, $p$ was the stereographic projection 
 $$p(z)=\frac{\operatorname{Im}(z)}{\operatorname{Re} (z)-1}.$$ Even though we use the Cayley transformation to define a circular order $\varphi$, the induced circular orders are same and it does not affect the results of \cite{Baik2019LaminarGA}.
 \end{rmk}

%%%%%%%%%%%%%%%%%%%%%%%%%%%%%%%%%%%%%%%%%%%%%%%%%%%%%%%%%%%%%%%%%%%
\subsection{Laminations of $S^1$}
Let $\mathcal{M}$ be the set of all two-point subsets of $S^1$. More precisely, 
 $$\mathcal{M}=((S^1)^2-\Delta_2(S^1))/(x,y)\sim (y,x).$$ 
 For each $\{a,b\}$ and $\{c,d\}$ of $\mathcal{M}$, they are said to be \emph{linked} if each connected component of $S^1-\{a,b\}$  contains precisely one of $c$ or $d.$ They are called \emph{unlinked} if they are not linked. 
 We can see easily that $\mathcal{M}$ has a hausdorff distance induced from the Euclidean metric of $\CC$. 
 
 \begin{defn}
 A \emph{lamination} of $S^1$ is a nonempty closed subset $\Lambda$ of $\mathcal{M}$ whose elements are pairwise unlinked. We call each element $\{a,b\}$ of $\Lambda$ a \emph{leaf} of $\Lambda$ and the points $a$ and $b$ are called the \emph{end points} of the leaf. 
 \end{defn}
 
\begin{rmk}
Let $A$ be a geodesic laminations of $\HH.$ For each leaf in $\HH,$ the end point set of the leaf is a two elements subset of $S^1.$ In this sense, there is a subset of $\mathcal{M}$ corresponding to $A$. Then, this subset is obviously a lamination of $S^1.$ For more detailed discussion, consult \cite[Chapter~$2$]{Calebook}.
\end{rmk}

%%%%%%%%%%%%%%%%%%%%%%%%%%%%%%%%%%%%%%%%%%%%%%%%%%%%%%%%%%%%%%%%%%%

\subsection{Lamination systems on $S^1$}

From now on, we recall the notion of lamination systems on $S^1$ which is first introduced in \cite{Baik2019LaminarGA}.  
This makes it easier to do general topological arguments for laminations of $S^1$. It turns out that a lamination system on $S^1$ is  a subset of the topology of $S^1$.

The idea of making a lamination system on $S^1$ from a lamination of $S^1$ is to take all connected components of complements of each leaf of the lamination of $S^1$. So, we will define lamination systems as a collection of nondegenerate open intervals in $S^1$ with some conditions corresponding to unlinkedness and closedness. 

Let $I$ and $J$ be two nondegenerate open intervals. If  $I\subseteq J$ or $I^* \subseteq J$, then we say that the pair of points $\{ I, I^*\}$ \emph{lies on} $J$ (See Figure \ref{fig:lie}). If  $\bar{I}\subseteq J$ or $\overline{I^*} \subseteq J$, then 
 we say that the pair of points $\{ I, I^*\}$ \emph{properly lies on} $J$. Now, we can define lamination systems as follows. 
 
 \begin{figure}
  \includegraphics[width=10cm]{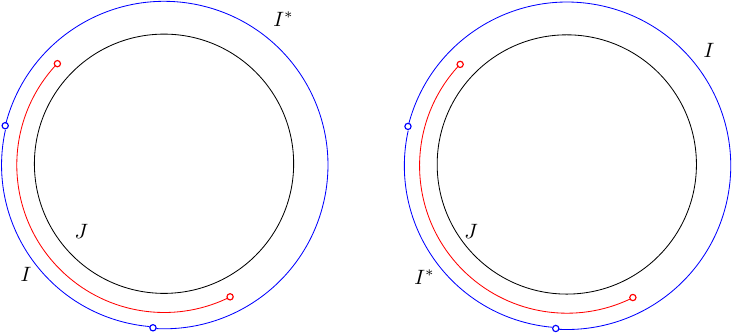}
  \caption{The red segment represents the nondegenerate open interval $J$ and the blue parts represent $I$ and $I^*$. Two figures show  all possible cases where $\{I,I^*\}$ lies on $J$.}
  \label{fig:lie}
\end{figure}
 
\begin{defn}[\cite{Baik2019LaminarGA}]
Let $\mathcal{L}$ be a nonempty family of nondegenerate open intervals of $S^1$. $\mathcal{L}$ is called 
a \emph{lamination system} \index{lamination system} on $S^1$ if it satisfies the following properties :
\begin{enumerate}
\item If $I\in \mathcal{L}$, then $I^* \in \mathcal{L}$. 
\item For any $I, J\in \mathcal{L}$, $\{I,I^*\}$ lies on $J$ or $J^*$.
\item If there is a sequence $\{I_n\}_{n=1}^{\infty}$ on $\mathcal{L}$ such that for any $n\in \NN$, $I_n\subseteq I_{n+1}$, and $\displaystyle \bigcup_{n=1}^{\infty} I_n$ is a nondegenerate open interval in $S^1$, then $\displaystyle \bigcup_{n=1}^{\infty} I_n\in \mathcal{L}$.
\end{enumerate}  
\end{defn}
The third condition will imply the closedness of lamination systems on $S^1.$ 

We define leaves and gaps on a lamination system $\LL$ as follows. A subset $\GG$ of $\LL$ is  a \emph{leaf} of $\LL$ if  $\GG=\{ I,I^*\}$ for some $I\in \LL$. We denote such  a leaf  $\GG$ by $\ell(I)$. Likewise, a subset $\GG$ of $\LL$ is a \emph{gap} of $\LL$ if $\GG$ satisfies  following conditions: 
\begin{enumerate}
\item Elements of $\GG$ are pairwise disjoint , and
\item for each $I\in\LL$, there is a $J$ in $\GG$ on which $\ell(I)$ lies. 
\end{enumerate}
By the second condition on gaps, every gap is nonempty.   Obviously,  a leaf is also a gap with two elements. Hence, we say that a gap is a \emph{non-leaf gap} if it is not a leaf. We denote $S^1- \displaystyle \bigcup_{I \in \GG} I$ by $v(\GG)$ and call it a \emph{vertex set} of $\GG$ or an  \emph{end points set} of $\GG$. Each element of a vertex set is called a \emph{vertex} or an \emph{end point}. Note that in general, a vertex set  need not be  a discrete subset of $S^1$.

Let us talk about the closedness of lamination systems. First, we must have the concept of  convergence in lamination systems. The third condition of lamination systems allows us define the notion of the limit of a sequence of leaves.

\begin{defn}[\cite{Baik2019LaminarGA}]
Let $\LL$ be a lamination system, and $\{\ell_n\}_{n=1}^{\infty}$ be a sequence of leaves in 
$\LL$. Let $J$ be a nondegenerate open interval. We say that $\{\ell_n\}_{n=1}^{\infty}$ \emph{converges to} $J$ if there is a sequence $\{I_n\}_{n=1}^{\infty}$ in $\LL$ such that 
for each $n\in \NN$, $\ell_n=\ell(I_n)$, and 
$$J\subseteq \liminf I_n \subseteq \limsup I_n \subseteq \overline{J}.$$ We denote this by 
$\ell_n\rightarrow J$.
\end{defn}

This definition is symmetric in the following sense. 

\begin{prop}[\cite{Baik2019LaminarGA}]\label{prop:symmetry of sequences}
Let $\LL$ be a lamination system and $\{\ell_n\}_{n=1}$ be a sequence of leaves in $\LL$. Let $J$ be a nondegenerate open interval. Suppose that  there is a sequence $\{I_n\}_{n=1}^{\infty}$ in $\LL$ such that for each $n\in \NN$, $\ell_n=\ell(I_n)$ and $$J\subseteq \liminf I_n \subseteq \limsup I_n \subseteq \overline{J}.$$ Then $$J^*\subseteq \liminf I_n^* \subseteq \limsup I_n^* \subseteq \overline{J^*}.$$
\end{prop}

Since the third condition on lamination systems guarantees  that the limit of an ascending sequence on a lamination system is in the lamination system,  we need to consider descending sequences to say about  limits of arbitrary sequences on lamination systems. 
The following lemma implies the  closedness of  descending sequences in a lamination system $\LL$. 
\begin{lem}[\cite{Baik2019LaminarGA}]\label{lods}
Let $\{ I_n\}_{n=1}^{\infty}$ be a sequence on a lamination system $\LL$ such that $I_{n+1}\subseteq I_n$ for all $n\in \NN$, and $\displaystyle \bigcup_{n=1}^{\infty} I_n^* =J \in \LL$. Then $\operatorname{Int}\Big( \displaystyle\bigcap_{n=1}^{\infty} I_n \Big)=J^*\in \LL$. 
\end{lem}

With this lemma, the following proposition shows the closedness of  lamination systems.

\begin{prop}[\cite{Baik2019LaminarGA}]\label{prop:closedness of sequence of leaves}
If a sequence $\{\ell_n\}_{n=1}^{\infty}$ of leaves of  a lamination system $\LL$ converges to a nondegenerate open interval $J$, then $J\in \LL$.
\end{prop}

Moreover, by Proposition \ref{prop:symmetry of sequences} we can prove that if $\ell_n \rightarrow J$, then $\ell_n \rightarrow J^*$, and so $J^*\in \LL$. 
Due to that, we can make a concept of the limit of a sequence of leaves.  

\begin{defn}[\cite{Baik2019LaminarGA}]
Let $\LL$ be a lamination system, and $\{\ell_n\}_{n=1}^{\infty}$ be a sequence of leaves on $\LL$. Let $\ell$ be a leaf of $\LL$. Then, we say that $\{\ell_n\}_{n=1}^{\infty}$ \emph{converges to} $\ell$ if $\ell_n \rightarrow I$ for some $I\in \ell$ and we denote it by 
$\ell_n \rightarrow \ell$. 
\end{defn}

\begin{rmk}
Let $\Lambda$ is a lamination of $S^1.$ Then the set $$\LL= \{(a,b)_{S^1}: \{ a,b\} \in \Lambda\}$$
is a lamination system. Conversely, for a given lamination systme $\LL$, the set 
$$\Lambda=\{v(\ell(I)): I \in \LL\}.$$ Note that if a sequence $\{\ell_n\}_{n=1}^\infty$ of leaves of $\LL$ converges to $\ell,$ then $\{v(\ell_n)\}_{n=1}^\infty$ is also a sequence of leaves of $\Lambda$ and converges to $v(\ell).$
\end{rmk}

%%%%%%%%%%%%%%%%%%%%%%%%%%%%%%%%%%%%%%%%%%%%%%%%%%%%%%%%%%%%%%%%%%%%
\subsection{Lamination systems and laminar groups}\label{Lamination systems and laminar groups}

A homeomorphism $f$ on $S^1$ is \emph{orientation preserving} if for each positively oriented triple $(x,y,z)\in (S^1)^3-\Delta_3(S^1)$, $(f(x),f(y),f(z))$ is positively oriented. $\Homeop(S^1)$ is the set of orientation preserving homeomorphisms on $S^1$. For each $f \in \Homeop(S^1)$, we denote the set of fixed points of $f$ by $\Fix_f.$

Let $G$ be  a subgroup of $\Homeop(S^1)$. The group $G$ is called a \emph{laminar group} if there is a lamination system $\LL$  such that  for any $g\in G$ and $I\in \LL$, $g(I)\in \LL.$ We say a lamination system $\LL$ to be \emph{$G$-invariant} if for each $g\in G$ and $I\in \LL$, $g(I)\in \LL.$ Note that on a $G$-invariant lamination system, every gap is mapped to a gap and the converging property is preserved under the given $G$-action. 

Let $\LL$ be a lamination system on $S^1.$ On a lamination system $\LL$ on $S^1$,   a gap $\GG$ with $|v(\GG)| < \infty$ is called 
 an \emph{ideal polygon}. In particular, an ideal polygon is called a \emph{non-leaf ideal polygon} if it is not a non-leaf gap.  We say that a lamination system  $\LL$ is \emph{very full} if every gap of $\LL$ is an ideal polygon.  Note that for an ideal polygon $\GG$, since $v(\GG)$ is a finite set, we can write $v(\GG)=\{x_1, x_2, \cdots ,x_n\}$ where $|v(\GG)|=n$, and $(x_1, x_2, \cdots, x_n)$ is a positively oriented $n$-tuple. Moreover, we can represent $\GG=\{ (x_1,x_2)_{S^1}, (x_2, x_3)_{S^1}, \cdots, (x_{n-1},x_n)_{S^1}, (x_n,x_1)_{S^1} \}$.
 
 Let $\{I, J\}$ be a subset of $\LL$. Then, $\{I, J\}$ is called a \emph{distinct pair} if $I\cap J =\emptyset$, and $\{I,J\}$ is not a leaf. 
A distinct pair $\{I,J\}$ is \emph{separated} if there is a non-leaf gap $\GG$ such that $I\subseteq K$ and $J\subseteq L$ for some $K, L\in \GG$, not necessarily $K\neq L.$ The lamination system  $\LL$ is \emph{totally disconnected} if every distinct pair is separated.

Let $I\in \LL$ and  $\{ \ell_n\}_{n=1}^{\infty}$ be a sequence of leaves of $\LL$. Then we call $\{\ell_n \}_{n=1}^{\infty}$ an \emph{$I$-side sequence} if for all $n\in \NN$, $I\notin \ell_n$, and $\ell_n$ lies on $I$, and $\ell_n\rightarrow I$. And we say that $I$ is \emph{isolated} if there is no $I$-side sequence on $\LL$. Moreover, a leaf $\ell$ is \emph{isolated} if each element of $\ell$ is isolated.

For a lamination system $\LL,$ let $E(\LL)=\displaystyle \bigcup_{I\in \LL}v(\ell(I))$ and we call it the \emph{end points set} of $\LL$. When we consider several lamination systems at the same time, considering transversality is useful.  

\begin{defn}
Two lamination systems  $\LL_1$ and $\LL_2$ on $S^1$ are \emph{transverse} if $\LL_1 \cap \LL_2=\emptyset.$ They are said to be \emph{strongly transverse} if $E(\LL_1)\cap E(\LL_2)=\emptyset.$
\end{defn}

\subsection{Some technical lemmas}
The following structure is useful to deal with the  configuration of leaves and to argue the existence of a gap. 
\begin{defn}[\cite{Baik2019LaminarGA}]
Let $\LL$ be a lamination system on $S^1$ and $I$ be a nondegenerate open interval. Then, for $p\in I$, we define $C_p^I$ as the set 
$C_p^I=\{ J\in \LL : p\in J \subseteq I \}.$
\end{defn}

We can see that $C_p^{I}$ is totally ordered by the set inclusion \cite{Baik2019LaminarGA}. 
The following lemma gives good descending or ascending sequences of $C_p^I$ to make sure the existence of the maximal or minimal interval in $C_p^{I}$. 
\begin{lem}[\cite{Baik2019LaminarGA}]\label{lem: existence of the extremal leaf}
Let $\LL$ be a lamination system on $S^1$, and $I$ be a nondegenerate open interval. Let $x$ be an element of $I$. Assume that $C_x^{I}$ is nonempty. Then there is a sequence $\{J_n\}_{n=1}^{\infty}$ on $C_x^{I}$ such that for all $n\in \NN$, $J_n\subseteq J_{n+1}$, and $\displaystyle \bigcup_{n=1}^{\infty} J_n= \bigcup_{K\in C_x^{I}}K $. Also, there is a sequence 
$\{K_n\}_{n=1}^{\infty}$  on $C_x^{I}$  such that for all $n\in \NN$,  $K_{n+1}\subseteq K_n$, and $\displaystyle\bigcap_{n=1}^{\infty} K_n= \bigcap_{K\in C_x^{I}}K $.
\end{lem}

The following lemma shows that the one-sided isolation condition is a sufficient condition for a gap to exist. 

 \begin{lem}[\cite{Baik2019LaminarGA}]\label{lem: existence of a real gap}
Let $\LL$ be a lamination system on $S^1$ and $I\in \LL$. Suppose that $I^*$ is isolated. Then there is a non-leaf gap $\GG$ such that $I \in \GG$.
\end{lem}

Like leaves, there is a kind of unlinkedness for gaps. 
\begin{lem}[\cite{Baik2019LaminarGA}]\label{lem: the configuration of two gaps}
Let $\LL$ be a lamination system on $S^1$, and $\GG$, $\GG'$ be two gaps with $|\GG|,|\GG'|\ge2$. Then, 
$\GG=\GG'$ or there are $I$ in $\GG$, and $I'$ in $\GG'$ such that $I^*\subseteq I'$, and for all $J\in \GG$, $\ell(J)$ lies on $I'$, and for all $J'\in \GG'$, $\ell(J')$ lies on $I$.   
\end{lem}

Let $\LL$ be a lamination system on $S^1.$  $\LL$ is called \emph{dense} if $E(\LL)$ is a dense subset of $S^1$. Let $p\in S^1$ and $\LL$ be a dense lamination system. Suppose that there is a sequence $\{I_n\}_{n=1}^{\infty}$ on $\LL$ such that for
all $n\in \NN$, $I_{n+1}\subseteq I_n$, and $\displaystyle
\bigcap_{n\in \NN} I_n=\{p\}$. We call such a sequence a
\emph{rainbow} at $p$. In \cite{BaikFuchsian}, it is observed that
very full laminations have abundant rainbows. 

\begin{thm}[\cite{BaikFuchsian},\cite{Baik2019LaminarGA}]\label{er}%enough rainbow
Let $\LL$ be a very full lamination system. For $p\in S^1$, either $p$ is in $E(\LL)$ or $p$ has a rainbow. These two possibilities are mutually exclusive. 
\end{thm}

\begin{cor}[\cite{BaikFuchsian},\cite{Baik2019LaminarGA}]\label{ed}%endpoints set are dense in S^1
Let $\LL$ be a very full lamination system of $S^1$. Then $E(\LL)$ is dense in $S^1$.
\end{cor}
Hence, very full lamination systems are dense. 

Totally disconnectedness is an important condition since it guarantees the existence of a gap. The transversality condition makes lamination systems be totally disconnected. 

\begin{lem}[\cite{alonso2019laminar},\cite{Baik2019LaminarGA}]\label{tot dis}
If two dense lamination systems are strongly transverse , then each of the lamination systems is totally disconnected. 
\end{lem}

\section{Laminar groups with two lamination systems and fixed point sets}\label{sec:laminar groups with two lami}

Analyzing fixed point sets of elements of laminar groups is useful to study laminar groups.  Let $g$ be an element of $\Homeop(S^1).$ If $g$ is not trivial, the fixed point set $\Fix_g$ is a proper closed subset of $S^1.$ Hence, the complement of the fixed set $\Fix_g$ is an at most countable disjoint union of open intervals. Therefore, we can see that the action of $g$ is just translations on the intervals of the complement of $\Fix_g.$ Fortunately, unlike laminar groups with one invariant lamination system, in the case where laminar groups with more than two lamination systems, the fixed point set of each element is a finite set \cite{Baik2019LaminarGA}.  
\begin{defn}[\cite{BaikFuchsian}]
A nontrivial element $g$ in $\Homeop(S^1)$ is said to be :
\begin{enumerate}
\item \emph{elliptic} if $|\Fix_g|=0$,
\item \emph{parabolic} if $|\Fix_g|=1$, and 
\item \emph{hyperbolic} if $|\Fix_g|=2$
\end{enumerate}
\end{defn}

A point $p$ of $S^1$ is called a \emph{cusp point} of $G$ if there is a parabolic element $g$ in $G$ fixing $p.$
\begin{rmk}[\cite{BaikFuchsian}]
Let $G$ be a laminar group and $\LL$ a $G$-invariant lamination system. Suppose that $\LL$ is very full. Given a cusp point, there are infinitely many leaves of $\LL$ having the cusp point as an end point.
\end{rmk}

Now we introduce some classes of laminar groups to study laminar groups having more than two lamination systems. 

\begin{defn}[\cite{BaikFuchsian}]\label{Defn:COL}
Let $G$ be a subgroup of $\Homeop(S^1)$ and let $\mathcal{C}=\{ \LL_\alpha \}_{\alpha \in J}$ be a collection of $G$-invariant  very full lamination systems, where $J$ is an index set. Then $\mathcal{C}$ is called \emph{pants-like} if the lamination systems in $\mathcal{C}$ are pairwise transverse and for any $\alpha \neq \beta \in J$, every element of $E(\LL_\alpha)\cap E(\LL_\beta)$ is a cusp point of $G$. A subgroup $G$ of $\Homeop(S^1)$ is said to be \emph{pants-like $\COL_n$} for some $n\in \NN$ if it admits a pants-like collection with $n$ elements. 
\end{defn}

 A subgroup $G$ of $\Homeop(S^1)$ is said to be \emph{Möbius-like} if for each element $g \in G$ , there is an element $h$ in $\Homeop(S^1)$ such that $hgh^{-1}\in \PR$. The group $G$ is said to be \emph{Möbius} if there is an element $h$ in $\Homeop(S^1)$ such that $\{hgh^{-1}: g\in G\}$ is a subgroup of $\PR$.
We recall the following theorem.
\begin{thm}[\cite{BaikFuchsian}]\label{thm: Mobius-like}
Let $G$ be a pants-like $\COL_3$ group. Then for any nontrivial $g\in G$, there is an element $h$ in $\Homeop(S^1)$ such that $hgh^{-1}\in \PR.$ Moreover, 
\begin{enumerate}
\item $g$ is elliptic if and only if $hgh^{-1}$ is an elliptic element of finite order in $\PR$,
\item $g$ is parabolic if and only if $hgh^{-1}$ is parabolic in $\PR$, and 
\item $g$ is hyperbolic if and only if $hgh^{-1}$ is hyperbolic in $\PR$.
\end{enumerate}
\end{thm}
Therefore, every pants-like $\COL_3$ group is Möbius-like. So one of our goals is to show that every pants-like $\COL_3$ group is Möbius.
%The main strategy is similar to that of \cite{BaikFuchsian}. 
Indeed every pants-like $\COL_3$ group is discrete in $\Homeop(S^1)$ with the compact open topology.

As we can see in Lemma \ref{tot dis}, the transversality condition implies totally disconnectedness. Likewise, we can prove also a similar result. Then this implies the discreteness of Möbius-like laminar groups. 
 
\begin{prop}\label{prop: COL_2 is totally disconnected}
Let $G$ be a pants-like $\COL_2$ group and $\mathcal{C}=\{\LL_1, \LL_2\}$ be a pants-like collection for $G$. Then both $\LL_1$ and $\LL_2$ are totally disconnected.
\end{prop}
\begin{proof}
First we show that $\LL_1$ is totally disconnected. Assume that there is a  $\{I,J\}$ be a distinct pair of $\LL_1.$ We claim that 
if there is a point $p$ in $I^* \cap J^*$ such that $C_p^{I^*} \cap C_p^{J^*}$ is nonempty in $\LL_1$, then $\{I,J\}$ is separated.  Let $K$ be the union of elements of $C_p^{I^*} \cap C_p^{J^*}.$ Then $K$ is a nondegenerate open interval such that $C_p^K=C_p^{I^*} \cap C_p^{J^*}.$ Moreover, by Lemma \ref{lem: existence of the extremal leaf} and the closedness of $\LL_1$, we have that $K\in \LL_1.$ Now we show that $K^*$ is isolated. Assume that there is a $K^*$-side sequence $\{\ell_n\}_{n=1}^\infty$ of $\LL_1.$ Then we can take a sequence $\{I_n\}_{n=1}^\infty$ in $\LL_1$ so that $\ell_n=\ell(I_n)$ for all $n\in \NN$  and 
$$K^* \subseteq \liminf I_n \subseteq \limsup I_n \subseteq \overline{K^*}.$$
Choose a point $p_I$ in $I$ and a point  $p_J$ in $J$. Since $\{p_I, p_J\} \subseteq K^*$ and $K^* \subseteq \liminf I_n$, there is a number $N$ in $\NN$ such that $\{p_I,p_J\} \subseteq \displaystyle \bigcap_{k=N}^\infty I_k$, that is for all $n$ with $N\leq n$, $\{p_I,p_J\} \subseteq I_n.$ Similarly choose a point $p_K$ in $K$. Since $\limsup I_n \subseteq \overline{K^*}$ and $p_K \notin K^c= \overline{K^*},$ there is a number $M$ in $\NN$ such that $\displaystyle p_K\notin \bigcup_{k=M}^\infty I_k$, that is for all $n\in \NN$ with $M\leq n$, $p_K\notin I_n.$ Fix $n_0$ with $\max\{N,M\} < n_0.$ Now we consider the unlinkedenss of $I$ and $I_{n_0}.$ Since $p_I\in I\cap I_{n_0}$ and $I\cap I_{n_0}$ is nonempty, there are three possible cases: $I\subseteq I_{n_0}$; $I_{n_0} \subseteq I$; $I_{n_0}^* \subseteq I$. If $I_{n_0} \subseteq I$, then $p_J\notin I_{n_0}$ as $p_J\notin I.$ This is a contradiction. If $I_{n_0}^* \subseteq I$, then $p_K\in \overline{I_{n_0}^*} \subseteq \overline{I}$ but $p_K\in K\subseteq (I^* \cap J^*) \subseteq I^*=(\overline{I})^c.$ This is also a contradiction. Hence $I\subseteq I_{n_0}$ is the only possible case. Likewise $J\subseteq I_{n_0}.$ Moreover since $\ell(I_{n_0})$ lies in $K^*$, $I_{n_0}\subseteq K^*$ or $I_{n_0}^* \subseteq K^*.$ However, if $I_{n_0}^* \subseteq K^*$, then $p_K\in K\subseteq I_{n_0}$ so this is in contradiction to that $p_K\notin I_{n_0}.$ So $I_{n_0}\subseteq K^*$ is the only possible case and then $p\in I_{n_0}^*.$ Therefore $I_{n_0}^*\subseteq C_p^{I^*} \cap C_p^{ J^*}$ and $K\subseteq I_{n_0}^*$. However since $\{\ell_n\}_{n=1}^\infty$ is a $K^*$-side sequence, $\ell(I_{n_0})\neq \ell(K)$ and so $K\subsetneq I_{n_0}^*.$ This contradicts the maximality of $K$ in  $C_p^K.$ Therefore $K^*$ is isolated. By Lemma \ref{lem: existence of a real gap}, there is a non-leaf gap $\GG$ of $\LL_1$ such that $K\in \GG.$ Then since $\GG$ is a gap, there are two elements $U_I$ and $U_J$ in $\GG$ such that $\ell(I)$ lies on $U_I$ and $\ell(J)$ lies on $U_J.$ If $I^*\subseteq U_I$, then $K\subseteq I^*\subseteq U_I$ so $K=I^*=U_I.$ However $p_J\in I^*$ and $p_J\notin K$. Hence this is a contradiction. Therefore $I\subseteq U_I.$ Similarly we can get that $J\subseteq U_J.$ Thus the distinct pair $\{I,J\}$ is separated so the claim is proved. 

 Now observe that the intersection $I^*\cap J^*$ is a nonempty open set as $\{I, J\}$ is not a leaf of $\LL_1.$ Then since $E(\LL_2)$ is dense by Lemma \ref{ed}, there is a point $p$ in $I^*\cap J^*\cap E(\LL_2).$ If $p\notin E(\LL_1)$, then by Lemma \ref{er}, there is a rainbow $\{J_n\}_{n=1}^\infty$ at $p$ in $\LL_1$. Therefore we can show that $C_p^{I^*}\cap C_p^{ J^*}$ is nonempty in this case, proving that $J_n\in C_p^{I^*}\cap C_p^{J^*}$ for some $n\in\NN$. Hence by the claim, $\{I,J\}$ is separated. If not, $p$ is in $E(\LL_1)\cap E(\LL_2)$ and is a cusp point of $G$ since $\{\LL_1, \LL_2\}$ is a pants-like collection. Then there is a parabolic element $g$ of $G$ fixing $p$ and there is a leaf $\ell$ of $\LL_1$ having $p$ as an end point. Now we take a nondegenerate open interval $U$ which is a neighborhood of $p$ so that $U\cap (I\cup J)=\emptyset.$ Since $g$ is parabolic, $g^L(\ell)$ lies on $U$ for some $L\in \NN.$ Hence there is an element $V$ in $g^L(\ell)$ such that $V\subseteq U$ and we pick a point $q$ in $V.$ Then $C_q^{I^*}\cap C_q^{J^*}$ is nonempty as $V\in C_q^{I^*}\cap C_q^{J^*}.$ By the claim, $\{I,J\}$ is separated. Thus $\LL_1$ is totally disconnected. Likewise $\LL_2$ is also totally disconnected. 
\end{proof} 

\begin{cor}\label{cor: finding a real gap}
Let $G$ be a pants-like $\COL_2$ group and $\mathcal{C}=\{\LL_1, \LL_2\}$ be a pants-like collection for $G$. Then each of  $\LL_1$ and $\LL_2$ has a non-leaf ideal polygon.
\end{cor}
\begin{proof}
Let $\LL$ be an element of $\mathcal{C}.$ Since $\LL$ is nonempty, there is an element $I$ of $\LL$ and because $E(\LL)$ is dense in $S^1$, there is a point $p$ in $I\cap E(\LL).$ Hence there is a leaf $\ell$ having $p$ as an end point and then since $\ell$ lies on $I$, there is an element $J$ of $\ell$ such that $J\subseteq I.$ Note that $I\notin \ell$ since $p\notin v(\ell(I))$ Therefore $\{J, I^*\}$ is a distinct pair and so by Proposition \ref{prop: COL_2 is totally disconnected}, there is a non-leaf gap $\GG$ of $\LL$ which makes $\{J,I^*\}$ be separated. Thus $\GG$ is a non-leaf ideal polygon since $\LL$ is very full. 
\end{proof}

\begin{cor}\label{cor: discreteness of laminar groups}
Let $G$ be a pants-like $\COL_2$ group and $\mathcal{C}=\{\LL_1, \LL_2\}$ be a pants-like collection for $G$. If $G$ is Möbius-like, then $G$ is a discrete subgroup of $\Homeop(S^1)$ with the compact open topology. 
\end{cor}
\begin{proof}
Let $\LL$ be an element of $\mathcal{C}.$ Then by Corollary \ref{cor: finding a real gap}, there is a non-leaf ideal polygon $\GG$ of  $\LL.$ Then we write $\GG=\{(x_1,x_2)_{S^1}, (x_2, x_3)_{S^1}, \cdots, (x_m, x_1)_{S^1}\}$ for some $m\in \NN$ with $2<m.$ We choose a number $\epsilon$ in $\RR$ so that $0<\epsilon<d_{\CC}(x_r, x_s)/3$ for all $r,s\in \ZZ_m.$ 

Now we assume that $G$ is not discrete, that is there is a sequence $\{g_n\}_{n=1}^\infty$ of nontrivial distinct elements of $G$ converging to the identity map $\id_{S^1}$ under the compact open topology. Then for each $r\in \ZZ_m$, let $U_r$ be the $\epsilon$-neighborhood of $x_r$ in $S^1$, i.e. $U_r=\{z\in S^1 : d_\CC(x_r, z)<\epsilon \}.$ Fix $r\in \ZZ_m.$ Since $\{x_r\}$ is compact and $\{g_n\}_{n=1}^\infty$ converges to $\id_{S^1},$ there is a number $N_r$ in $\NN$ such that $g_n(x_r)\subseteq U_r$ whenever $N_r<n.$ 

Then we choose a number $N$ in $\NN$ which is greater than $\max \{N_1, N_2, \cdots, N_m\}.$ Then for each $r\in \ZZ_m$, $g_N(x_r)\in U_r.$ On the other hand, by Lemma \ref{lem: the configuration of two gaps}, $g_N(\GG)=\GG$ or  $g_N(\{x_1,x_2, \cdots, x_m\})\subseteq [x_i, x_{i+1}]_{S^1}$ for some $i\in \ZZ_m.$ If $g_N(\{x_1,x_2, \cdots, x_m\})\subseteq [x_i, x_{i+1}]_{S^1}$ for some $i\in \ZZ_m$, then $\epsilon < d_{\CC}(x_{i+2}, g_N(x_{i+2}))$  so this is in contradiction to that $g_N(x_{i+2})\in U_{i+2}.$ If  $g_N(\GG)=\GG$, then by the choice of $N$, $g_N$ must fix the end points of $\GG$ and so $g_N$ has at least three fixed points since $\GG$ is not a leaf. However because $\GG$ is Möbius-like, $g_N$ must be the identity map but this is a contradiction since $g_N$ was taken to be nontrivial. Thus $G$ is a discrete subgroup of $\Homeop(S^1).$    
\end{proof}

\begin{cor}
 Every pants-like $\COL_3$ group is a discrete subgroup of $\Homeop(S^1)$ with the compact open topology.
\end{cor}
\begin{proof}
It follows from Theorem \ref{thm: Mobius-like} and Corollary \ref{cor: discreteness of laminar groups}.
\end{proof}
Hence we cloud expect the following lemma. 
\begin{lem}\label{lem: disjoint fixed point sets}
Let $G$ be a pants-like $\COL_2$ group and $\mathcal{C}=\{\LL_1, \LL_2\}$ be a pants-like collection for $G$. If $G$ is Möbius-like, then hyperbolic elements can not share fixed points with parabolic elements, that is for any hyperbolic element $g$ and parabolic element $h$ of $G$, $\Fix_g\cap \Fix_h=\emptyset.$ 
\end{lem}
\begin{proof}
Let $\mathcal{C}=\{\LL_1,\LL_2\}$ be a pants-like collection for $G.$
Assume that there is a hyperbolic element $g$ and a parabolic element $h$ of $G$ such that $\Fix_{g}\cap \Fix_h =\{p\}.$ We denote the other fixed point of $g$ by $q$.  For each $i\in \ZZ_2$, there is a leaf $\ell^i$ of $\LL_i$ since $p$ is a cusp point.

 Fix $i\in \ZZ_2$. We claim that there is a leaf $\ell_\infty^i$ of $\LL_i$ such that $v(\ell_\infty^i)=\Fix_g.$ If $q\in v(\ell^i)$, then 
 $\ell_\infty^i=\ell^i.$ If not, there is an element $I^i$ in $\ell^i$ containing the point $q$. Consider the sequences $\{g^n(I^i)\}_{n=1}^\infty$ and $\{g^{-n}(I^i)\}_{n=1}^\infty.$ When $g(I^i) \subsetneq I^i$, $\displaystyle I_\infty^i=\operatorname{Int}\big( \bigcap_{n=1}^\infty g^n(I^i) \big)$ is in $\LL_i$ by Lemma \ref{lods} and $v(\ell(I_\infty^i))=\{p, q\}.$ Hence $\ell_\infty^i=\ell(I_\infty^i).$ Similarly if $I^i\subsetneq g(I^i),$ then $\displaystyle I_\infty^i=\operatorname{Int}\big( \bigcup_{n=1}^\infty g^{-n}(I^i) \big)$ is in $\LL_i$ and $v(\ell(I_\infty^i))=\{p, q\}.$ Therefore $\ell_\infty^i=\ell(I_\infty^i).$  However this implies that $\LL_1\cap \LL_2 \neq \emptyset$ and so it is a contradiction since $\LL_1$ and  $\LL_2$ are transverse. Thus we are done.
\end{proof}

\section{Pants-like $\COL_3$ groups are convergence groups}\label{sec:main1}

% In this section, we show that every pants-like $\COL_3$ group is a convergence group. Let $\Phi$ be a subset of $\Homeop(S^1)$. We say a sequence of distinct elements of $\Phi$ to have the \emph{convergence property} if there exist two points $a$ and $b$ in $S^1$ and a subsequence $\{g_n\}_{n=1}^\infty$ such that the sequence $\{g_n |_{S^1-\{b\}}\}_{n=1}^\infty$ of maps restricted to $S^1-\{b\}$ converges uniformly to the constant map from $S^1-\{b\}$ to $\{a\}$ on every compact subset of $S^1-\{b\}$. Then, we say that $\Phi$ to be \emph{convergence} if every sequence of distinct elements of $G$ has the convergence property.

% \begin{rmk}
%   Let $\{g_n\}_{n=1}^\infty$ be a sequence of distinct elements of a subset $\Phi \subset \Homeop(S^1).$ Then 
%   $\{g_n\}_{n=1}^\infty$  has the convergence property if and only if $\{g_n^{-1}\}_{n=1}^\infty$ has the convergence property.
% \end{rmk} 

% (Here!!!)

In this section, we show that every pants-like $\COL_3$ group is a convergence group. Let $G$ be a subgroup of $\Homeop(S^1).$
We say a sequence of distinct elements of  $G$ to have the \emph{convergence property} if there exist two points $a$ and $b$ in $S^1$ and a subsequence $\{g_n\}_{n=1}^\infty$ such that the sequence $\{g_n |_{S^1-\{b\}}\}_{n=1}^\infty$ of maps restricted to $S^1-\{b\}$ converges uniformly to the constant map from $S^1-\{b\}$ to $\{a\}$ on every compact subset of $S^1-\{b\}$.
The group $G$ is called a \emph{convergence group} if every sequence of distinct elements of $G$ has the convergence property.

\begin{rmk}\label{rmk: the convergence property for inverse}
Let $\{g_n\}_{n=1}^\infty$ be a sequence of distinct elements of a subgroup $G< \Homeop(S^1).$ Then 
$\{g_n\}_{n=1}^\infty$  has the convergence property if and only if $\{g_n^{-1}\}_{n=1}^\infty$ has the convergence property.
\end{rmk} 

In the proofs of several theorems related to the convergence property, we frequently take subsequences. Hence for clarity we sometimes think of a sequence $\{s_n\}_{n=1}^\infty$ in a set $S$ as a function $s$ from $\NN$ to $S$ so that $s(n)=s_n$ for all $n\in \NN$. Also, in this sense, a subsequence $\{s_{n_k}\}_{k=1}^\infty$ of the sequence has a function $\alpha$ on $\NN$ increasing strictly so that $\alpha(k)=n_k$ for all $k\in \NN$. Let $\alpha$ be a strictly increasing function on $\NN$. A strictly increasing function $\gamma$ on $\NN$ is called a \emph{subindex} of $\alpha$ if there is a strictly increasing function $\beta$ on $\NN$ such that $\gamma =\alpha \circ \beta$. 

 In the pants-like $\COL_3$ case,  it is not so easy to directly show that a given sequence of distinct elements has the convergence property. Theorem \ref{Bowditch} gives another equivalent condition for convergence groups which makes the situation be slightly easier. Let $\Theta(S^1)$ be the set of all subsets of $S^1$ of cardinality $3$ or equivalently $\Theta(S^1)$ is the quotient space of $(S^1)^3-\Delta_3(S^1)$ by the symmetric group $S_3$ permuting coordinates (following the notation in \cite{Bowditch1999ConvergenceGA}). We say that a sequence $\{h_n\}_{n=1}^\infty$ of distinct elements of $G$ is \emph{properly discontinuous on triples} if for any two compact subsets $K$ and $L$ of $\Theta(S^1)$, the set $\{m \in \NN : h_m(K)\cap L \neq \emptyset \}$ is finite.

\begin{thm}[\cite{Bowditch1999ConvergenceGA}]\label{Bowditch}
Let $G$ be a subgroup of $\Homeop(S^1)$ and $\{g_n\}_{n=1}^\infty$ be a sequence of distinct elements of $G$. If $\{g_n\}_{n=1}^\infty$ is properly discontinuous on triples, then every subsequence of $\{g_n\}_{n=1}^\infty$ has the convergence property. The converse also holds.
\end{thm}

%a sequence of distinct elements of $G$ has the convergence property if and only if it is properly discontinuous on triples. 

%\begin{thm}[\cite{Tukia88}]\label{Tukia}
%A subgroup of $\Homeop(S^1)$ is a convergence group if and only if it acts properly discontinuously on $\Theta(S^1).$
%\end{thm}	

The strategy for proving that pants-like $\COL_3$ groups are  convergence groups is first to assume that there is a sequence $\{g_n\}_{n=1}^\infty$ not having the convergence property. By Theorem \ref{Bowditch}, the sequence is not properly discontinuous on triples. Then there are two compact subsets $K$ and $L$ of $\Theta(S^1)$ such that the set $\{m\in \NN : g_m(K)\cap L\neq \emptyset \}$ is infinite. Namely, there is a strictly increasing function $\alpha$ on $\NN$ such that for all $n\in \NN$, $g_{\alpha(n)}(K)\cap L \neq \emptyset$.  
\begin{lem}\label{not properly discontinuous}
Let $X$ be a topological space and $\{g_n\}_{n=1}^\infty$ a sequence of distinct homeomorphisms on $X$. Suppose that there are two sequentially compact subsets $K$ and $L$ of $X$ such that $g_n(K)\cap L \neq \emptyset$ for infinitely many $n\in \NN$.  Then there is a sequence $\{x_k\}_{k=1}^\infty$ in $K$ and a subsequence $\{g_{n_k}\}_{k=1}^\infty$ such that $\{x_k\}_{k=1}^\infty$ converges to a point $x$ in $K$ and $\{g_{n_k}(x_k)\}_{k=1}^\infty$ converges to a point $x'$ in $L.$
\end{lem}

Now by Lemma \ref{not properly discontinuous}, there is a sequence $\{ x_n\}_{n=1}^\infty$ in $K$ and a subindex $\beta$ of  $\alpha$ such that the sequence $\{x_n\}_{n=1}^\infty$ converges to a point $x_\infty$ in $K$ and the sequence $\{g_{\beta(n)}(x_n)\}_{n=1}^\infty$ converges to a point $y_\infty$ in $L.$ Set $x_n=(x_n^1, x_n^2, x_n^3)$, $x_\infty=(x_\infty^1,x_\infty^2,x_\infty^3)$ and $y_\infty=(y_\infty^1,y_\infty^2,y_\infty^3).$ In other words we have that for each $i\in \ZZ_3$, the sequence $\{x_n^i\}_{n=1}^\infty$ converges to $x_\infty^i$ and the sequence $\{g_{\beta(n)}(x_n^i)\}_{n=1}^\infty$ converges to $y_\infty^i.$ Then, using lamination systems, we show that the subsequence $\{g_{ \beta(n)}\}_{n=1}^\infty$ has the convergence property. This is in contradiction to the assumption.

Therefore we will focus on analyzing the sequences $\{x_n^i\}_{n=1}^\infty$ and $\{g_{\beta(n)}\}_{n=1}^\infty.$
To control the behavior of sequences of homeomorphisms, we need to consider a slightly generalized concept of rainbows.

%We consider sequences as functions from the set $\NN$ of natural numbers to some set. So, sometimes, for a sequence $\{a_n \}_{n=1}^{\infty}$, we write $\{a(n) \}_{n=1}^{\infty}$ instead of $\{a_n\}_{n=1}^{\infty}$ because when we take subsequences again and again, the exposition looks better. 

%%%%%%%%%%%%%%%%%%%%%%%%%%%%%%%%%%%%%%%%%%%%%%%%%%%%%%%%%%%%%%%%%%  

\begin{defn}
Let $p$ be a point in $S^1$ and $\LL$ a lamination system. Then a sequence $\{I_n\}_{n=1}^{\infty}$ in $\LL$ is called a \emph{quasi-rainbow} at the point $p$ if for  each $n\in \NN$, $I_{n+1}\subseteq I_n$ and $\displaystyle \bigcap_{n=1}^{\infty}\overline{I_n}=\{p\}.$
\end{defn}
Every rainbow at $p$ is also a quasi-rainbow at $p$. Let $\LL$ be a lamination system and $\{I_n\}_{n=1}^{\infty}$ a rainbow at $p\in S^1$. Then  $\displaystyle \{p\}=\bigcap_{n=1}^{\infty}I_n \subseteq \bigcap_{n=1}^{\infty}\overline{I_n}.$ Assume that  there is a element $q$ such that $\displaystyle q\in \bigcap_{n=1}^{\infty}\overline{I_n}-\{p\}$. Then $q\in \overline{I_n}$ for all $n\in \NN$. On the other hand, since  $\displaystyle \{p\}=\bigcap_{n=1}^{\infty}I_n$ and $I_{n+1}\subseteq I_n$ for all $n\in \NN$, there is $N$ such that $q\notin I_n$ for all $n>N.$ Therefore, for each $n>N$, $q\in \partial I_n$ so one of $(p,q)_{S^1}$ or $(q,p)_{S^1}$ is contained in $I_n$. If there is $M$ and $M'$ such that $M,M'>N$ , $(p,q)_{S^1} \subseteq I_M$, and $(q,p)_{S^1} \subseteq I_{M'}$, then one of $I_M$ and $I_{M'}$ must be $(q,q)_{S^1}$ since it is either $I_M\subseteq I_{M'}$ or $I_{M'}\subseteq I_M$. It is a contradiction since $(q,q)_{S^1} \notin \LL$. Therefore,  if $(p,q)_{S^1} \subseteq I_{N'}$ for some $N'>N$, then $(p,q)_{S^1}\subseteq I_n$ for all $n>N'$ but $\displaystyle \{p\}\subsetneq [p,q)_{S^1}\subseteq \bigcap_{n=1}^{\infty}I_n.$ It is a condtradiction. Similarly, if $(q,p)_{S^1} \subseteq I_{N''}$ for some $N''>N$, then $(q,p)_{S^1}\subseteq I_n$ for all $n>N''$ but $\displaystyle \{p\}\subsetneq (q,p]_{S^1}\subseteq \bigcap_{n=1}^{\infty}I_n.$ It is also a contradiction. Thus the rainbow $\{I_n\}_{n=1}^\infty$ is a quasi-rainbow at $p$. 
%%%%%%%%%%%%%%%%%%%%%%%%%%%%%%%%%%%%%%%%%%%%%%%%%%%%%%%%%%%%%%%%%%

\begin{defn}
Let $p$ be a point in $S^1$ and $\{p_n\}_{n=1}^{\infty}$ a sequence in $S^1$ converging to the point $p$. If $\varphi(p, p_1, p_2)\neq 0$ and $\varphi(p,p_1,p_2)=\varphi(p,p_n, p_{n+1})$ for all $n\in \NN$, then we say that the sequence $\{p_n\}_{n=1}^{\infty}$ \emph{converges monotonically} to the point $p$. 
\end{defn}
Let $\{p_n\}_{n=1}^\infty$ be a sequence in $S^1$ converging monotonically to a point $p$ in $S^1$. Observe that for each $n\in \NN$, $\varphi(p, p_n, p_m)=\varphi(p, p_1, p_2)$ for all $m\in \NN$ with $n<m.$ Hence every subsequence also converges monotonically. And if $\varphi(p, p_1, p_2)=-1$ and $\varphi(p, p_2, p_1)=1$, then $p_n\in (p,p_1)_{S^1}$ for all $n\in \NN$ with $1<n$ as $\varphi(p, p_n, p_1)=1$ for all $n\in \NN$ with $1<n.$ Moreover, for each $n\in \NN$, $p_m\in (p,p_n)_{S^1}$ for all $m\in \NN$ with $n<m.$ Then for any $n$ and $m$ in $\NN$ with $n<m$, $(p,p_m)_{S^1}\subsetneq (p,p_n)_{S^1}$ and since $\{p_n\}_{n=1}^\infty$ converges to $p$,$\displaystyle \bigcap_{n=1}^\infty (p, p_n)_{S^1}=\emptyset$ and  $\displaystyle \bigcap_{n=1}^\infty [p, p_n]_{S^1}=\{p\}.$ This looks like a quasi-rainbow at $p$. Likewise, if $\varphi(p, p_1, p_2)=1$, then we can derive similar results.

\begin{lem}\label{lem: making a quasi-rainbow}
Let $G$ be a pants-like $\COL_2$ group and $\mathcal{C}=\{\LL_1, \LL_2\}$ be a pants-like collection for $G$. Let $p$ be a point in $S^1$ and $\{p_n\}_{n=1}^\infty$ be a sequence in $S^1$ converging monotonically to $p$. Then there exist a subsequence $\{p_{n_k}\}_{k=1}^\infty$ and a quasi-rainbow $\{I_k\}_{k=1}^\infty$ in $\LL_i$ for some $i\in \ZZ_2$ such that $p_{n_k}\in I_k$ for all $k\in \NN.$
\end{lem}
\begin{proof}
First we consider the case where one of $\LL_1$ or $\LL_2$ has a rainbow $\{I_n\}_{n=1}^\infty$ at $p$. For each $n\in \NN$, there is a number $m(n)$ in $\NN$ such that for all $m$ with $m(n)< m$, $p_m\in I_n$  because $I_n$ is a neighborhood of $p$. Then a subindex $\alpha$ of the identity $\id_\NN$ on $\NN$ can be given by  $\alpha(n)=\displaystyle \sum_{k=1}^n m(k)$ so that $p_{\alpha(n)}\in I_n$ for all $n \in \NN.$ 

Then assume that there is no rainbow at $p$ in $\LL_1$ and in $\LL_2$. By Theorem \ref{er}, $p\in E(\LL_1) \cap E(\LL_2)$ so $p$ is a cusp point of $G$ since $\mathcal{C}$ is pants-like. As observed previously, there is a sequence $\{I_n\}_{n=1}^\infty$ of nondegenerate open intervals such that for each $n$ in $\NN$, $\partial I_n=\{p, p_n\}$ and $I_{n+1}\subseteq I_n.$ Choose a leaf $\ell$ of $\LL_1$ having the point $p$ as an end point. Let $q$ denote the other end point of $\ell$. 
Because $\displaystyle \bigcap_{n=1}^\infty \overline{I_n}=\{p\},$ there is a number $N$ such that $q\notin \overline{I_N}.$ Since $S^1-\{p,q\}=(p,q)_{S^1}\cup (q,p)_{S^1}$, $I_N\subseteq S^1-\{p,q\}$ and $I_N$ is connected, there is an element $J$ in $\ell$ such that $\overline{I_N} \cap \overline{J^*}=\{p\}.$ Therefore, $p_m \in J$ for all $m$ with $N\leq m$. Since $p$ is a cusp point of $G$, we can take a parabolic element $g$ in $G$ fixing $p$ so that $g(q)\in J$ and $g(J)\subset J.$ 

 Now we claim that  for each $n\in \NN$, there is a number $o(n)$ in $\NN$ such that $p_m\in g^{n-1}(J)$ for all $m\in \NN$ with $o(n)\leq m.$
 When $n=1$, $o(1)=N$. Fix a number $n$. Since $g^n(q)\in g^{n-1}(J)\subset S^1-\{p\}$, there is a number $M$ with $N<M$ such that $g^n(q)\notin \overline{I_M}.$ Then one of $g^n(J)$ or $g^n(J^*)$ contains $I_M$. If $I_M\subseteq g^n(J^*)$, then $I_M \subseteq g^n(J^*) \cap J$ and so $p_m\in g^n(J^*) \cap J $ for all $m$ with $M<m$. Hence  the open neighborhood $g^n(J)\cup \{p\} \cup J^*$ of $p$ contains only finitely many points of the sequence $\{p_n\}_{n=1}^\infty$ since  $g^n(J)\cup \{p\} \cup J^*$  and $g^n(J^*) \cap J $ are disjoint. It contradicts that $\{p_n\}_{n=1}^\infty$ converges to $p$. Therefore $I_M\subseteq g^n(J)$ and $\overline{I_M}\cap \overline{g^n(J^*)}=\{p\}.$ So $p_m\in g^n(J)$ for all $m$ with $M\leq m$. Thus if we set $o(n+1)=M$, then the claim is proved. 
 
 Finally, define a subindex $\beta$ of $\id_\NN$ by $\beta(n)=\displaystyle \sum_{k=1}^n o(k)$ so that $p_{\beta(n)}\in g^{n-1}(J)$ for all $n\in \NN.$ The sequence $\{g^{n-1}(J)\}_{n=1}^\infty$ is contained in $\LL_1$ and so it is a quasi-rainbow at $p$ in $\LL_1$. Thus we are done.
\end{proof} 

\begin{defn}
 Let $G$ be a laminar group and $\LL$ a $G$-invariant lamination system. Let $p\in S^1$. Assume that there is a quasi-rainbow $\{I_n\}_{n=1}^{\infty}$  at $p$.
 A sequence  $\{(g_n, I_n)\}_{n=1}^{\infty}$  of  elements of $G \times \LL$ is called a \emph{pre-approximation sequence} at $p$ if 
there is a point $x$ in $S^1$ such that $g_n(x)\in I_n$ for all $n\in \NN.$
 \end{defn}
\begin{rmk}
Any subsequence of a pre-approximation sequence at $p$ is also a pre-approximation sequence at $p$.
\end{rmk}

The following is a key lemma to prove the main result of this chapter. This lemma allows us to properly pass to subsequences. 

\begin{lem}\label{lem: taking a pre-approximation sequence}
Let $G$ be a pants-like $\COL_2$ group and $\mathcal{C}=\{\LL_1, \LL_2\}$ be a pants-like collection for $G$. Suppose that we have  a sequence $\{x_n\}_{n=1}^{\infty}$ of elements of $S^1$ converging to $x\in S^1$ and a sequence $\{g_n\}_{n=1}^{\infty}$ of distinct elements of $G$ such that $\{g_n(x_n)\}_{n=1}^\infty$ converges to $x'\in S^1$. Then we can have one of the following cases:
\begin{enumerate}
\item there is a subsequence $\{g_{n_k}\}_{k=1}^{\infty}$ such that $g_{n_k}(x)=x'$ for all $k\in\NN$; 
\item there is a subsequence $\{g_{n_k}\}_{k=1}^{\infty}$ and a quasi-rainbow $\{I_k\}_{k=1}^\infty$ at $x'$ in $\LL_i$ for some $i\in \ZZ_2$ such that the sequence $\{(g_{n_k},I_k)\}_{k=1}^{\infty}$ is a pre-approximation sequence at $x'$;
\item  there is a subsequence $\{g_{n_k}\}_{k=1}^{\infty}$ and  a quasi-rainbow $\{I_k\}_{k=1}^\infty$ at $x$ in $\LL_i$ for some $i\in \ZZ_2$  such that the sequence $\{(g_{n_k}^{-1},I_k)\}_{k=1}^{\infty}$ is a pre-approximation  sequence at $x$. 
\end{enumerate} 
%Moreover each of two sequences $\{x_{n_k}\}_{k=1}^\infty$ and $\{g_{n_k}(x_{n_k})\}_{k=1}^\infty$ is either a constant sequence or a sequence converging monotonically. 
\end{lem}
\begin{proof}
Observe that every convergence sequence in $S^1$ has  a constant subsequence or a subsequence converging monotonically. First we will take a subsequence of $\{x_n\}_{n=1}^\infty.$ Then we need to consider two possible cases by the observation. Assume that there is a subindex $\alpha$ of the identity $\id_\NN$ such that the subsequence $\{x_{\alpha(n)}\}_{n=1}^\infty$ is constant i.e. $x_{\alpha(n)}=x$ for all $n\in \NN.$ The sequence $\{g_{\alpha(n)}(x_{\alpha(n)})\}_{n=1}^\infty$ still converges to $x'$. We will take again a subindex of $\alpha$ for $\{g_{\alpha(n)}(x_{\alpha(n)})\}_{n=1}^\infty.$ If there is a subindex $\beta$ of $\alpha$ such that  $\{g_{\beta(n)}(x_{\beta(n)})\}_{n=1}^\infty$ is constant, then $g_{\beta(n)}(x)=g_{\beta(n)}(x_{\beta(n)})=x'$ for all $n\in \NN$ and this is the first case of the statement. If there is a subindex $\beta$ of $\alpha$ such that  the sequence $\{g_{\beta(n)}(x_{\beta(n)})\}_{n=1}^\infty$ converges monotonically to $x'$, then by Lemma \ref{lem: making a quasi-rainbow}, we can take a subindex $\gamma$ of $\beta$ and a quasi-rainbow $\{I_n\}_{n=1}^\infty$ at $x'$ in $\LL_i$ for some $i\in \ZZ_2$ so that $g_{\gamma(n)}(x_{\gamma(n)})\in I_n$ for all $n\in \NN.$ Since  the sequence $\{x_{\gamma(n)}\}_{n=1}^\infty$ is constant, the sequence $\{(g_{\gamma(n)}, I_n)\}_{n=1}^\infty$ is a pre-approximation sequence at $x'$ and this is the second case of the statement. 

Next we consider the case where there is a subindex $\alpha$ of the identity $\id_\NN$ such that the subsequence $\{x_{\alpha(n)}\}_{n=1}^\infty$ converges monotonically to $x'$. Similarly we will take a subsequence of the sequence $\{g_{\alpha(n)}(x_{\alpha(n)})\}_{n=1}^\infty$.
If there is a subindex $\beta$ of $\alpha$ such that the sequence  $\{g_{\beta(n)}(x_{\beta(n)})\}_{n=1}^\infty$  is constant, then $x_{\beta(n)}=g_{\beta(n)}^{-1}(x')$ for all $n\in \NN$. Because the sequence $\{x_{\beta(n)}\}_{n=1}^\infty$ still converges monotonically to $x$, Lemma \ref{lem: making a quasi-rainbow} can apply so there is a subindex $\gamma$ of $\beta$ and a quasi-rainbow $\{I_n\}_{n=1}^\infty$ at $x$ in $\LL_i$ for some $i\in \ZZ_2$ such that $g_{\gamma(n)}^{-1}(x')=x_{\gamma(n)}\in I_n$ for all $n\in \NN.$ Then the sequence $\{(g_{\gamma(n)}^{-1},I_n)\}_{n=1}^\infty$ is a pre-approximation sequence at $x$ and this is the third case. 

Finally we assume that there is a subindex $\beta$ of $\alpha$ such that the sequence  $\{g_{\beta(n)}(x_{\beta(n)})\}_{n=1}^\infty$ converges monotonically to $x'.$ By Lemma \ref{lem: making a quasi-rainbow}, there is a subindex $\gamma$ of $\beta$ and a quasi-rainbow $\{J_n\}_{n=1}^\infty$ at $x'$  in $\LL_i$ for some $i\in \ZZ_2$ such that  $g_{\gamma(n)}(x_{\gamma(n)})\in J_n$ for all $n\in \NN.$ For brevity, we write $h_n=g_{\gamma(n)}$ and $y_n=x_{\gamma(n)}.$ Let us consider the sequence $\{h_n^{-1}(J_n)\}_{n=1}^\infty.$

First, assume that  $x\in  \overline{h_n^{-1}(J_n)}$ for infinitely many $n\in \NN$. Then there is a subindex $\delta$ of the identity $\id_\NN$ such that $x\in \overline{h^{-1}_{\delta(n)}(J_{\delta(n)})}=h^{-1}_{\delta(n)}(\overline{J_{\delta(n)}})$ for all $n\in \NN.$ Then $h_{\delta(n)}(x)\in \overline{J_{\delta(n)}}$ for all $n\in \NN$ and since $\{J_{\delta(n)}\}_{n=1}^\infty$ is a quasi-rainbow at $x',$ the sequence $\{h_{\delta(n)}(x)\}_{n=1}^\infty$ converges to $x'.$ Then there are two possible cases. First, if there is a subindex $\epsilon$ of $\delta$ such that $\{h_{\epsilon(n)}(x)\}_{n=1}^\infty$ is a constant sequence, then $\{g_{\gamma\circ \epsilon(n)}\}_{n=1}^\infty$ is in the first case. Second, if we can take a subindex $\epsilon$ of $\delta$ so that the sequence $\{h_{\epsilon(n)}(x)\}_{n=1}^\infty$ converges monotonically to $x',e$ then by Lemma \ref{lem: making a quasi-rainbow}, there  are a subindex $\zeta$ of $\epsilon$ and a quasi-rainbow $\{I_n\}_{n=1}^\infty$ at $x'$ in $\LL_i$ for some $i\in \ZZ_2$ such that $h_{\zeta(n)}(z)\in I_n$ for all $n\in \NN.$ Then the sequence $\{(g_{\gamma \circ \zeta(n)},I_n)\}_{n=1}^\infty$ is a pre-approximation sequence at $x'$ and this is the second case.

Next, assume that   $x\in  \overline{h_n^{-1}(J_n)}$ for only finitely many $n\in \NN$. Then we can take a subindex $\delta$ of $\id_\NN$ such that $x\notin \overline{h_{\delta(n)}^{-1}(J_{\delta(n)})}$ for all $n\in \NN$. Now, we consider the case where there is $n_0$ in $\NN$ such that $y_{\delta(n_0)}\in h_{\delta(n)}^{-1}(J_{\delta(n)})$ for infinitely many $n$ in $\NN.$ Then we can take a subindex $\zeta$ of $\delta$ so that $\zeta(1)=\delta(n_0)$ and $y_{\zeta(1)}\in h_{\zeta(n)}^{-1}(J_{\zeta(n)})$   for all $n\in \NN.$ Observe that if $y_{\zeta(m)}\in h_{\zeta(n)}^{-1}(J_{\zeta(n)})$ for infinitely many $m\in \NN,$ then $x\in \overline{h_{\zeta(n)}^{-1}(J_{\zeta(n)})},$  so, by assumption, for each $n\in \NN$,  $y_{\zeta(m)}\in h_{\zeta(n)}^{-1}(J_{\zeta(n)})$
for only finitely many $m\in\NN.$ Then there is a number $M$ in $\NN$ such that $y_{\zeta(m)}\notin   h_{\zeta(1)}^{-1}(J_{\zeta(1)})$ for all $m \in \NN$ with $M<m.$ Choose a number $m$ with $M<m.$ Now there are four cases:
 $h_{\zeta(m)}^{-1}(J_{\zeta(m)})\subseteq h_{\zeta(1)}^{-1}(J_{\zeta(1)})$ ;
 $h_{\zeta(m)}^{-1}(J_{\zeta(m)})^*\subseteq h_{\zeta(1)}^{-1}(J_{\zeta(1)})$;
 $h_{\zeta(1)}^{-1}(J_{\zeta(1)})\subseteq h_{\zeta(m)}^{-1}(J_{\zeta(m)})^*$;
 $h_{\zeta(1)}^{-1}(J_{\zeta(1)})\subseteq h_{\zeta(m)}^{-1}(J_{\zeta(m)})$.
 Since $y_{\zeta(m)}\in h_{\zeta(m)}^{-1}(J_{\zeta(m)})-h_{\zeta(1)}^{-1}(J_{\zeta(1)})$, the first case that $h_{\zeta(m)}^{-1}(J_{\zeta(m)})\subseteq h_{\zeta(1)}^{-1}(J_{\zeta(1)})$ is not possible. Also, since $x\in h_{\zeta(m)}^{-1}(J_{\zeta(m)})^*-h_{\zeta(1)}^{-1}(J_{\zeta(1)})$, the second case is not possible. If  $h_{\zeta(1)}^{-1}(J_{\zeta(1)})\subseteq h_{\zeta(m)}^{-1}(J_{\zeta(m)})^*$, then  $h_{\zeta(1)}^{-1}(J_{\zeta(1)}) \cap h_{\zeta(m)}^{-1}(J_{\zeta(m)})=\emptyset$ but this is in contradiction with that $y_{\zeta(1)}\in  h_{\zeta(1)}^{-1}(J_{\zeta(1)})\cap h_{\zeta(m)}^{-1}(J_{\zeta(m)}).$ Therefore,  $h_{\zeta(1)}^{-1}(J_{\zeta(1)})\subseteq h_{\zeta(m)}^{-1}(J_{\zeta(m)})$, and  $y_{\zeta(1)}\in  h_{\zeta(m)}^{-1}(J_{\zeta(m)})$ for all $m\in \NN$ with $M<m.$ So we have that $h_{\zeta(m)}(y_{\zeta
 (1)})\in J_{\zeta(m)}$ for all $m\in \NN$ with $M<m$ and we define a subindex $\eta$ of $\zeta$ as $\eta(n)=\zeta(n+M)$ for all $n\in \NN.$ Then, $\{(g_{\gamma \circ \eta(n)}, J_{\eta(n)})\}_{n=1}^\infty$ is a pre-approximation sequence at $x'$ since $g_{\gamma \circ \eta(n)}(y_{\zeta(1)})=h_{\zeta(n+M)}(y_{\zeta(1)})\in  J_{\eta(n)}$  for all $n\in \NN.$ Thus, this is the second case.

Now we also assume that there is no such a $n_0.$ 
Then for each $n\in \NN$, we define $L_n=\{h_{\delta(m)}^{-1}(J_{\delta(m)}) : m\in \NN \  \text{and} \ y_{\delta(n)}\in h_{\delta(m)}^{-1}(J_{\delta(m)})\}.$ Obviously, any $L_n$ is not empty since $y_{\delta(n)}\in h_{\delta(n)}^{-1}(J_{\delta(n)})$  and $ h_{\delta(n)}^{-1}(J_{\delta(n)})\in L_n$, and is finite by the assumption. Fix a number $n\in \NN$. First we show that $L_n$ is totally ordered by the set inclusion $\subseteq.$  Let $M$ and $N$ be elements of $L_n.$ Then since each element belong to the same lamination system, $M\subseteq N$, $M^*\subseteq N$, $M \subseteq N^*$ or $M^* \subseteq N^*$ by unlinkedness. If $M^*\subseteq N$, then $x\in M^* \subseteq N$ as $x\notin \overline{M}=(M^*)^c.$ It is in contradiction to that $x\notin N$. The case where $M \subseteq N^*$ is also impossible since $y_{\delta(n)}\in M$ but $y_{\delta(n)}\notin N^*.$ Therefore $M\subseteq N$ or $M^* \subseteq N^*$, or equivalently $M\subseteq N$ or $N \subseteq M$. Thus the claim is proved. 

Then we want to show that for any $m$ and $n$ in $\NN$, $\bigcup L_m$ and $\bigcup L_n$ are same or disjoint. Note that $\bigcup L_n$ and $\bigcup L_m$ are the maximal elements of $L_m$ and $L_n$ respectively. By unlinkedness, $\bigcup L_m \subseteq \bigcup L_n$,$\big[\bigcup L_m\big]^* \subseteq \bigcup L_n$, $\bigcup L_m \subseteq \big[\bigcup L_n\big]^*$ or $\big[\bigcup L_m\big]^* \subseteq \big[ \bigcup L_n\big]^*.$  If  $\bigcup L_m \subseteq \bigcup L_n$ or $\big[\bigcup L_m\big]^* \subseteq \big[ \bigcup L_n\big]^*$, then 
$\bigcup L_m$ and $ \bigcup L_n$ are same by the maximality.  If  $\big[\bigcup L_m\big]^* \subseteq \bigcup L_n$, then $x\in \big[\bigcup L_m\big]^* \subseteq \bigcup L_n$ but $x\notin \bigcup L_n$. So this is not a possible case. If $\bigcup L_m \subseteq \big[\bigcup L_n\big]^*$, then $\bigcup L_m$ and $ \bigcup L_n$ are disjoint. Therefore the claim can be attained.  

Next we claim that for any $n$ and $m$ in $\NN$ with $n\leq m$, if $\bigcup L_n=\bigcup L_m$, then $\bigcup L_n=\bigcup L_k$ for all $k$ with $n\leq k \leq m.$ Assume that for some $n$ and $m$ with $n< m$, there is a number $k$ in $\NN$ with $n<k< m$ such that  $\bigcup L_n$ and $\bigcup L_k$ are disjoint. There is a nondegenerate open interval $I$ such that $\partial I =\{y_{\delta(n)}, y_{\delta(m)}\}$ and $x\in I^*.$ Then $[\bigcup L_n]^*\subseteq I^*$ or equivalently  $I\subseteq \bigcup L_n$ since $x\in[\bigcup L_n]^*$, $[\bigcup L_n]^* \subseteq  S^1- \{y_{\delta(n)}, y_{\delta(m)}\}$ and $[\bigcup L_n]^*$ is connected. Hence $I$ and $ \bigcup L_k$ are disjoint as $\bigcup L_n$ and $ \bigcup L_k$ are disjoint. This implies that $y_{\delta(k)}\in I^c=I^* \cup \partial I$. Because $y_{\delta(k)}\neq y_{\delta(n)}$ and $y_{\delta(k)}\neq y_{\delta(m)}$, we can get that $y_{\delta(k)}\in I^*$ and so $\varphi(y_{\delta(n)},y_{\delta(k)}, y_{\delta(m)})=\varphi(y_{\delta(n)},x, y_{\delta(m)}).$ On the other hand 
$\varphi(x, y_{\delta(1)}, y_{\delta(2)})=\varphi(x, y_{\delta(n)}, y_{\delta(m)})=\varphi(x,  y_{\delta(n)}, y_{\delta(k)})=\varphi(x,  y_{\delta(k)}, y_{\delta(m)}) $ since the sequence $\{y_{\delta(n)}\}_{n=1}^\infty$ converges monotonically to $x.$ By the cocycle condition for the $4$-tuple $(x, y_{\delta(n)}, y_{\delta(k)}, y_{\delta(m)} )$, 
\begin{align*}
0&=\varphi( y_{\delta(n)}, y_{\delta(k)}, y_{\delta(m)} )-\varphi(x, y_{\delta(k)}, y_{\delta(m)} )+\varphi(x, y_{\delta(n)},y_{\delta(m)} )-\varphi(x, y_{\delta(n)}, y_{\delta(k)} )\\
&=\varphi( y_{\delta(n)}, y_{\delta(k)}, y_{\delta(m)} )-\varphi(x, y_{\delta(1)}, y_{\delta(2)})+\varphi(x, y_{\delta(1)}, y_{\delta(2)})-\varphi(x, y_{\delta(1)}, y_{\delta(2)})\\
&=\varphi( y_{\delta(n)}, y_{\delta(k)}, y_{\delta(m)} )-\varphi(x, y_{\delta(1)}, y_{\delta(2)}).
\end{align*}
Therefore $\varphi( y_{\delta(n)}, y_{\delta(k)}, y_{\delta(m)} )=\varphi(x, y_{\delta(1)}, y_{\delta(2)})=\varphi(x, y_{\delta(n)}, y_{\delta(m)})
=-\varphi( y_{\delta(n)},x, y_{\delta(m)}).$ However this is in contradiction to that $\varphi(y_{\delta(n)},y_{\delta(k)}, y_{\delta(m)})=\varphi(y_{\delta(n)},x, y_{\delta(m)})$ since $\varphi(y_{\delta(n)},x, y_{\delta(m)})\neq 0.$
Thus we can get the result which we wanted.

Then we define $C=\{\bigcup L_n: n\in \NN \}$ and can define a partial order $\leq$ on $C$ as follows. For any pair of two elements $U$ and $V$ in $C$, $U\leq V$ if and only if there are numbers $n$ and $m$ in $\NN$ such that $n\leq m$, $y_{\delta(n)}\in U$, and $y_{\delta(m)}\in V.$ This order is well-defined by the previous claims and is a total order.  Recall that as mentioned before, for each $n\in \NN$, $y_{\delta(m)}\in h_{\delta(n)}^{-1}(J_{\delta(n)})$ for only finitely many $m\in \NN.$ Then there is the order preserving bijection $\mu$ from $(C,\leq)$ to $\NN$ with the standard total order.  Also we can define a function $\nu$ from $\{y_{\delta(n)}: n \in \NN\}$ to $C$ so that $\nu(y_{\delta(n)})=\bigcup L_n.$ Then the composition map $\mu \circ \nu \circ \xi$ is a monotonically increasing surjection on $\NN$, where the function $\xi$ from $\NN$ to $\{y_{\delta(n)}: n \in \NN\}$ is defined by $n \mapsto y_{\delta(n)}.$ Hence the map $\mu \circ \nu \circ \xi$ has a right inverse $\pi$. It follows at once that the map $\pi$ on $\NN$ is an order preserving injection, that is the map $\pi$ is a strictly increasing function on $\NN.$ Let $\epsilon$ be a subindex of $\delta$ defined by $\delta \circ \pi.$ 

Now we consider the sequence $\{h_{\epsilon(n)}^{-1}(J_{\epsilon(n)})\}_{n=1}^\infty.$ Since the sequence  $\{y_{\epsilon(n)}\}_{n=1}^\infty$ converges monotonically to $x$, as discussed previously, there is the sequence $\{K_n\}_{n=1}^\infty$ of nondegenerate open intervals such that for each $n\in \NN,$ $\partial K_n=\{x, y_{\epsilon(n)}\}$ and $K_{n+1} \subsetneq K_n.$ We claim that for each $n\in \NN$, $h_{\epsilon(m)}^{-1}(J_{\epsilon(m)})\subseteq K_n$ for all $m\in \NN$ with $n<m$. Choose a pair of two numbers $n$ and $m$ in $\NN$ with $n<m.$ Note that $ \nu \circ \xi \circ \pi$ is a bijection since  $ \nu \circ \xi \circ \pi =\mu^{-1}\circ \id_\NN.$ It follows that $U_n$ and  $U_m$ are disjoint where  we write $U_n=\nu \circ \xi \circ \pi(n)$ and $U_m=\nu \circ \xi \circ \pi(m)$. Because $U_m\subset S^1-\{x,y_{\epsilon(n)}\}$, $y_{\epsilon(m)}\in K_n$ and $U_m$ is connected, we can conclude that $U_m\subset K_n.$ Therefore $h_{\epsilon(m)}^{-1}(J_{\epsilon(m)})\subseteq U_m \subset K_n.$ Thus the claim is proved. 

Finally we consider the sequence $\{h_{\epsilon(n)}^{-1}(x')\}_{n=1}^\infty.$ Since the sequence $\{J_{\epsilon(n)}\}_{n=1}^\infty$ is a quasi-rainbow at $x'$, by the previous claim, for each $n\in \NN$, $h_{\epsilon(m)}^{-1}(x')\in h_{\epsilon(m)}^{-1}(\overline{J_{\epsilon(m)}})= \overline{h_{\epsilon(m)}^{-1}(J_{\epsilon(m)})}\subset \overline{K_n}$ for all $m\in \NN$ with $n<m.$ Hence the sequence $\{h_{\epsilon(n)}^{-1}(x')\}_{n=1}^\infty$ converges to $x$ as $\displaystyle \bigcap_{n=1}^\infty \overline{K_n}=\{x\}.$ Note that for each $n\in \NN$, $h_{\epsilon(n)}^{-1}(x')\neq x$  since $x\notin \overline{h_{\epsilon(n)}^{-1}(J_{\epsilon(n)})}.$ Then we can take a subindex $\zeta$ of $\gamma \circ \epsilon$ so that the sequence $\{g_{\zeta(n)}^{-1}(x')\}_{n=1}^\infty$ converges monotonically to $x.$ By Lemma \ref{lem: making a quasi-rainbow}, there is a subindex  $\eta$ of $\zeta$ and a quasi-rainbow $\{I_n\}_{n=1}^\infty$ at $x$ in $\LL_i$ for some $i\in \ZZ_2$ such that $g_{\eta(n)}^{-1}(x')\in I_n$ for all $n\in \NN.$ Then the sequence $\{(g_{\eta(n)}^{-1}, I_n)\}_{n=1}^\infty$ is a pre-approximation sequence at $x$. Thus this is the third case.
 \end{proof}

Now we will apply this lemma in turn to $\{x_n^1\}_{n=1}^\infty$, $\{x_n^2\}_{n=1}^\infty$ and $\{x_n^3\}_{n=1}^\infty$ with $\{g_{\beta(n)}\}_{n=1}^\infty.$ Then each subsequencing step has three possible cases as Lemma \ref{lem: taking a pre-approximation sequence}. Then it turns out that after passing to a subsequence of $\{g_{\beta(n)}\}_{n=1}^\infty$, the subsequence has the following property. 
\begin{defn}
Let $G$ be a Möbius-like subgroup of $\Homeop(S^1)$ and $\{g_n\}_{n=1}^\infty$ be a sequence of distinct elements of $G.$ Let $p$ and $q$ be two distinct points in $S^1$. Then the sequence $\{g_n\}_{n=1}^\infty$ is called an \emph{approximation sequence} to $\{p,q\}$ if for any pair of two numbers $m$ and $k$ of $\NN$ with $m<k$, $g_k\circ g_m^{-1}$ is hyperbolic and the sequence $\{\Fix_{g_{\alpha(n+1)}\circ g_{\alpha(n)}^{-1}}\}_{n=1}^\infty$ in $\mathcal{M}$ converges to $\{p,q\}$ for any subindex $\alpha$ of $\id_\NN.$
\end{defn}

\begin{rmk}
Every subsequence of an approximation sequence to $\{p,q\}$ is an approximation sequence to $\{p,q\}.$
\end{rmk}

First we consider the following lemmas. 
\begin{lem}\label{lem: the first type}
Let $G$ be a Möbius-like subgroup of $\Homeop(S^1)$ and  $\{g_n\}_{n=1}^\infty$ be a sequence of distinct elements of $G.$ Let  $x'$ and $y'$ be two distinct points in $S^1$ . Assume that there are two distinct points $x$ and $y$ in $S^1$ such that for any $n\in \NN,$ $g_n(x)=x'$ and $g_n(y)=y'.$ Then $\{g_n\}_{n=1}^\infty$ is an approximation sequence to $\{x',y'\}.$  
\end{lem}
\begin{proof}
Choose two distinct numbers $n$ and $m$ in $\NN$ with $n<m.$ Then $g_m\circ g_n^{-1}(x')=x'$ and $g_m\circ g_n^{-1}(y')=y'$ and since $\{g_n\}_{n=1}^\infty$ is a sequence of distinct elements so $g_n\neq g_m,$ the element $g_m\circ g_n^{-1}$ is a nontrivial element having a fixed point. Hence $g_m\circ g_n^{-1}$ must be a hyperbolic element since $G$ is Möbius-like, $\{x',y'\}\subseteq \Fix_{g_m\circ g_n^{-1}}$ and $x'\neq y'.$ Moreover for any subindex $\alpha$ of $\id_\NN,$ $g_{\alpha(n+1)}\circ g_{\alpha(n)}^{-1}$ are hyperbolic and $\Fix_{g_{\alpha(n+1)}\circ g_{\alpha(n)}^{-1}}=\{x',y'\}.$ Thus $\{g_n\}_{n=1}^\infty$ is an approximation sequence to $\{x',y'\}.$
\end{proof}

\begin{lem}\label{lem: the second type}
Let $G$ be a Mobius-like laminar group and $\LL$ be a $G$-invariant lamination system. Let $\{g_n\}_{n=1}^\infty$ be a sequence of distinct elements of $G,$ and $x'$ and $y'$ be two distinct points in $S^1.$ Suppose that there is a point $x$ in $S^1$ such that $g_n(x)=x'$ for all $n\in \NN.$ Assume that we can take a quasi-rainbow $\{I_n\}_{n=1}^\infty$ at $y'$ in $\LL$  so that the sequence $\{(g_n, I_n)\}_{n=1}^\infty$ is a pre-approximation sequence at $y'.$ Further assume that $x' \notin \overline{I_1}.$ Then $\{g_n\}_{n=1}^\infty$ is an approximation sequence to $\{x',y'\}.$  
\end{lem}
\begin{proof}
Since  $\{(g_n, I_n)\}_{n=1}^\infty$ is a pre-approximation sequence at $y',$ there is a point $y$ in $S^1$ such that $g_n(y)\in I_n$ for all $n\in \NN.$ Choose two distinct numbers $n$ and $m$ in $\NN$ with $n<m.$ Then $g_m\circ g_n^{-1}(x')=x'$ and $x'\in \Fix_{g_m\circ g_n^{-1}}.$ Since $\{g_n\}_{n=1}^\infty$ is a sequence of distinct elements so $g_n\neq g_m,$ the element $g_m\circ g_n^{-1}$ is a nontrivial element having a fixed point. 

Now we consider $I_n.$ The element $g_m(y)$ belongs to $(g_m\circ g_n^{-1}(I_n))\cap I_n$ as $y \in g_n^{-1}(I_n)$ and $g_m(y)\in I_m\subset I_n.$ Hence  $g_m\circ g_n^{-1}(I_n)$ and  $I_n$ are not disjoint so by the unlinkedness of them, there are three cases: 
$g_m\circ g_n^{-1}(I_n)\subset I_n$; $(g_m\circ g_n^{-1}(I_n))^* \subset I_n$; $(g_m\circ g_n^{-1}(I_n))^* \subset I_n^*.$ If $(g_m\circ g_n^{-1}(I_n))^* \subset I_n$ or equivalently $g_m\circ g_n^{-1}(I_n^*)\subset I_n$ , then $x'=g_m\circ g_n^{-1}(x')\in g_m\circ g_n^{-1}(I_n^*)\subset I_n$ as $x'\in I_n^*.$ This is a contradiction since $I$ and $I^*$ are disjoint. 

Therefore $g_m\circ g_n^{-1}(I_n)\subseteq I_n$ or $(g_m\circ g_n^{-1}(I_n))^* \subset I_n^*,$ namely $g_m\circ g_n^{-1}(\overline{I_n})\subseteq \overline{I_n}$ or $\overline{I_n} \subset g_m\circ g_n^{-1}(\overline{I_n}).$ Then the Brouwer's fixed point theorem can apply and implies that there must be a fixed point of $g_m\circ g_n^{-1}$ in $\overline{I_n}.$  Since $x' \notin \overline{I_1}$, $\overline{I_n}\subset \overline{I_1},$ and  $G$ is Möbius-like, the element $g_m\circ g_n^{-1}$ is a hyperbolic element and there is a unique fixed point in $\overline{I_n}.$ 

Finally let $\alpha$ be a subindex of $\id_\NN.$ Note that for all $n\in \NN,$ the element $g_{\alpha(n+1)}\circ g_{\alpha(n)}^{-1}$ is hyperbolic and $x'\in \Fix_{g_{\alpha(n+1)}\circ g_{\alpha(n)}^{-1}}.$ We can write $\Fix_{g_{\alpha(n+1)}\circ g_{\alpha(n)}^{-1}}=\{x',p_n\}$ for all $n\in\NN.$ As we saw, $p_n\in \overline{I_{\alpha(n)}}.$ Moreover the sequence $\{p_n\}_{n=1}^\infty$ converges to $y'$ as $\{I_{\alpha(n)}\}_{n=1}^\infty$ is a quasi-rainbow at $y'.$ Therefore the sequence  $\{x',p_n\}_{n=1}^\infty$ in $\mathcal{M}$ converges to $\{x',y'\}.$ Thus we are done.  
\end{proof}

\begin{lem}\label{lem: the third type}
Let $G$ be a Mobius-like laminar group. Let $\{g_n\}_{n=1}^\infty$ be a sequence of distinct elements of $G,$ and $x_1'$ and $x_2'$ be two distinct points in $S^1.$ Suppose that for each $i\in \ZZ_2,$ we can take a quasi-rainbow $\{I_n^i\}_{n=1}^\infty$ at $x_i'$ in some $G$-invariant lamination system so that the sequence $\{(g_n, I_n^i)\}_{n=1}^\infty$ is a pre-approximation sequence at $x_i'.$ Assume that $\overline{I_1^1}$ and $\overline{I_1^2}$ are disjoint. Then $\{g_n\}_{n=1}^\infty$ is an approximation sequence to $\{x_1',x_2'\}.$  
\end{lem}
\begin{proof}
Choose two distinct numbers $n$ and $m$ in $\NN$ with $n<m.$ For each $i\in \ZZ_2,$ there is a point $x_i$ in $S^1$ such that $g_n(x_i)\in I_n^i$ for all $n\in \NN$ since  $\{(g_n, I_n^i)\}_{n=1}^\infty$ is a pre-approximation sequence at $x_i'.$ Hence  for each $i\in \ZZ_2,$  $g_m(x_i)\in g_m\circ g_n^{-1}(I_n^i)$ so $g_m\circ g_n^{-1}(I_n^i)\cap I_n^i \neq \emptyset$ since $g_m(x_i)\in I_m^i\subset I_n^i.$ Fix $i\in \ZZ_2.$ By the unlinkedness of $I_n^i$ and $g_m\circ g_n^{-1}(I_n^i),$ there are three cases:  $g_m\circ g_n^{-1}(I_n^i)\subset I_n^i$;  $g_m\circ g_n^{-1}(I_n^i)^* \subset I_n^i$; $g_m\circ g_n^{-1}(I_n^i)^*\subset (I_n^i)^*.$ If $g_m\circ g_n^{-1}(I_n^i)^* \subset I_n^i,$ then $g_m\circ g_n^{-1}(I_n^{i+1}) \subset g_m\circ g_n^{-1}((I_n^i)^*) \subset I_n^i$ since $\overline{I_1^1}$ and $\overline{I_1^2}$ are disjoint and $I_n^{i+1} \subset (I_n^i)^*.$ Hence  $g_m(x_{i+1})\in g_m\circ g_n^{-1}(I_n^{i+1}) \subset I_n^i$ as $x_{i+1}\in g_n^{-1}(I_n^{i+1}),$ and so $g_m(x_{i+1})\in  (I_m^{i+1}\cap I_n^i) \subset  (I_n^{i+1}\cap I_n^i).$ However this is a contradiction since $\overline{I_1^1}$ and $\overline{I_1^2}$ are disjoint. Therefore $g_m\circ g_n^{-1}(I_n^i)\subset I_n^i$ or $g_m\circ g_n^{-1}(I_n^i)^*\subset (I_n^i)^*,$ namely $g_m\circ g_n^{-1}(\overline{I_n^i})\subset \overline{I_n^i}$ or $\overline{I_n^i} \subset g_m\circ g_n^{-1}(\overline{I_n^i}).$  Then the Brouwer's fixed point theorem can apply and implies that for each $i\in\ZZ_2,$ $I_n^i$ has a fixed point of the element $g_m\circ g_n^{-1}.$ Therefore the element $g_m\circ g_n^{-1}$ has at least two fixed points since  $\overline{I_1^1}$ and $\overline{I_1^2}$ are disjoint. Note that  $g_m\circ g_n^{-1}$ is nontrivial since $\{g_n\}_{n=1}^\infty$ is a sequence of distinct elements so $g_n\neq g_m.$   Since $G$ is Möbius-like, $g_m\circ g_n^{-1}$ must be hyperbolic and so  for each $i\in\ZZ_2,$ $I_n^i$ has exactly one fixed point of the element $g_m\circ g_n^{-1}.$

Finally let $\alpha$ be a subindex of $\id_\NN.$ Note that for each $n\in \NN,$ the element $g_{\alpha(n+1)}\circ g_{\alpha(n)}^{-1}$ is hyperbolic and for each $i\in \ZZ_i,$ $\overline{I_{\alpha(n)}^i}$ has exactly one fixed point $p_n^i$ of $g_{\alpha(n+1)}\circ g_{\alpha(n)}^{-1}$ . Hence we can write $\Fix_{g_{\alpha(n+1)}\circ g_{\alpha(n)}^{-1}}=\{p_n^1, p_n^2\}$ for all $n\in\NN.$ Therefore the sequence $\{p_n^1, p_n^2\}_{n=1}^\infty$ in $\mathcal{M}$ converges to $\{x_1',x_2'\}$ as $\{I_{\alpha(n)}^1\}_{n=1}^\infty$ and $\{I_{\alpha(n)}^2\}_{n=1}^\infty$  are quasi-rainbows at $x_1'$ and $x_2'$ respectively. Thus we are done. 
 \end{proof}
 
 \begin{theorem}{\ref{A}}
 Every pants-like $\COL_3$ group is a convergence group. 
 \end{theorem}
 \begin{proof}
Let $G$ be a pants-like $\COL_3$ group and $\mathcal{C}=\{\LL_i\}_{i\in \ZZ_3}$ be a pants-like collection for $G.$ We want to show that every sequence of distinct elements of $G$ has the convergence property. Now we assume that there is a sequence $\{g_n\}_{n=1}^\infty$ of distinct elements of $G$ that does not have the convergence property. Then by Proposition \ref{Bowditch}, the sequence is not properly discontinuous on triples, that is there are two compact subsets $K$ and $L$ of $\Theta (S^1)$ such that $\{m\in \NN : g_m(K)\cap L \neq \emptyset \}$ is infinite. Hence there is a sequence $\{x_n\}_{n=1}^\infty$ in $K$ converging to a point $x_\infty$ in $K$ and a subindex $\alpha$ of $\id_\NN$ such that the sequence $\{g_{\alpha(n)}(x_n)\}_{n=1}^\infty$ converges to a point $y_\infty$ in $L$ by Lemma \ref{not properly discontinuous}. Then we write $x_n=(x_n^1, x_n^2, x_n^3)$ for all $n\in \NN \cup \{\infty\}$ and $y_\infty =(y_\infty^1, y_\infty ^2, y_\infty^3)$ so that for each $i\in \ZZ_3,$ the sequence $\{x_n^i\}_{n=1}^\infty$ in $S^1$ converges to the point $x_\infty^i$ of  $S^1$ and the sequence $\{g_{\alpha(n)}(x_n^i)\}_{n=1}^\infty$ converges to the point $y_\infty^i$ of $S^1.$ Choose a number $\epsilon$ in $\RR$ so that $0<\epsilon<d_{\CC}(x_\infty^i, x_\infty^{i+1})/3$ and $0<\epsilon<d_{\CC}(y_\infty^i, y_\infty^{i+1})/3$ for all $i\in \ZZ_3.$ For each $p\in x_\infty \cup y_\infty,$ we write $N_\epsilon(p)=\{z\in S^1 : d_\CC(p,z)<\epsilon \}.$ 

First if there is a subindex $\beta$ of $\alpha$ such that  for some $i_0\in \ZZ_3,$ $g_{\beta(n)}(x_\infty^{i_0})= y_\infty^{i_0}$ and $g_{\beta(n)}(x_\infty^{{i_0}+1})= y_\infty^{{i_0}+1}$ for all $n\in \NN.$ Then by Lemma \ref{lem: the first type}, the sequence $\{g_{\beta(n)}\}_{n=1}^\infty$ is an approximation sequence to $\{y_\infty^{i_0}, y_\infty^{i_0+1}\}$ and also $\{g_{\beta(n)}^{-1}\}_{n=1}^\infty$  is an approximation sequence to $\{x_\infty^{i_0}, x_\infty^{i_0+1}\}$. Applying Lemma \ref{lem: taking a pre-approximation sequence} to the index $i_0+2,$ we can take a subindex $\gamma$ of $\beta$ so that $\{g_{\gamma(n)}\}_{n=1}^\infty$ is one of the three cases in Lemma \ref{lem: taking a pre-approximation sequence}. In the case where $g_{\gamma(n)}(x_\infty^{i_0+2})=y_\infty^{i_0+2}$ for all $n\in \NN,$  we get that $y_\infty \subset \Fix_{g_{\gamma(n+1)}\circ g_{\gamma(n)}^{-1}}$ for all $n\in \NN.$ This implies that $g_{\gamma(n+1)}\circ g_{\gamma(n)}^{-1}=\id_{S^1}$ for all $n\in \NN$ since every pants-like $\COL_3$ group is Möbius-like. This is in contradiction to that the sequence $\{g_n\}_{n=1}^\infty$ is a sequence of distinct elements of $G.$ The first case is not possible. Hence there is a quasi-rainbow $\{I_n\}_{n=1}^\infty$ in some lamination system in $\mathcal{C}$ such that either $\{(g_{\gamma(n)}, I_n)\}_{n=1}^\infty$ is a pre-approximation sequence at $y_\infty^{{i_0}+2}$ or $\{(g_{\gamma(n)}^{-1}, I_n)\}_{n=1}^\infty$ is a pre-approximation sequence at $x_\infty^{{i_0}+2}.$ When  $\{(g_{\gamma(n)}, I_n)\}_{n=1}^\infty$ is a pre-approximation sequence at $y_\infty^{{i_0}+2},$ there is a number $N$ in $\NN$ such that $\overline{I_N} \subset N_\epsilon(y_\infty^{{i_0}+2})$ since $\{I_n\}_{n=1}^\infty$ is a quasi-rainbow at $y_\infty^{{i_0}+2}.$ Then we can take a subindex $\delta$ of $\id_\NN$ so that  $\overline{I_{\delta(1)}}\subset  N_\epsilon(y_\infty^{{i_0}+2}).$ Note that the sequence $\{(g_{\gamma\circ \delta(n)}, I_{\delta(n)})\}_{n=1}^\infty$ is still a pre-approximation sequence at $y_\infty^{{i_0}+2}.$ Since $y_\infty^{i_0}\notin   N_\epsilon(y_\infty^{{i_0}+2})$ and $y_\infty^{i_0}\notin \overline{I_{\delta(1)}},$ the sequence $\{g_{\gamma\circ \delta(n)}\}_{n=1}^\infty$ is an approximation sequence to $\{y_\infty^{i_0}, y_\infty^{{i_0}+2}\}$ by Lemma \ref{lem: the second type}. Then there is a number $n_0$ in $\NN$ such that the hyperbolic element $g_{\gamma\circ \delta(n_0+1)}\circ g_{\gamma\circ \delta(n_0)}^{-1}$ has exactly one fixed point in each of $N_\epsilon(y_\infty^{i_0})$ and $N_\epsilon(y_\infty^{{i_0}+2}).$ Since  $y_\infty^{{i_0}+1}\in \Fix_{g_{\gamma\circ \delta(n_0+1)}\circ g_{\gamma\circ \delta(n_0)}^{-1}}, $ the element $g_{\gamma\circ \delta(n_0+1)}\circ g_{\gamma\circ \delta(n_0)}^{-1}$ has at least three fixed point by the choice of $\epsilon.$ Hence $g_{\gamma\circ \delta(n_0+1)}\circ g_{\gamma\circ \delta(n_0)}^{-1}$ must be $\id_{S^1}$ since $G$ is Möbius-like. This is a contradiction since $\{g_n\}_{n=1}^\infty$ is a sequence of distinct elements of $G.$ Likewise the case where $\{(g_{\gamma(n)}^{-1}, I_n)\}_{n=1}^\infty$ is a pre-approximation sequence at $x_\infty^{{i_0}+2}$ is also impossible. Therefore there is no such a subindex $\beta.$

Then we consider the case where there is a subindex $\beta$ of $\alpha$ such that for some $i_0\in \ZZ_3,$ $g_{\beta(n)}(x_\infty^{i_0})= y_\infty^{i_0}$ for all $n\in \NN.$ Applying Lemma \ref{lem: taking a pre-approximation sequence} to the index $i_0+1,$ we can take a subindex $\gamma$ of $\id_\NN$ so that $\{g_{\beta \circ \gamma(n)}\}_{n=1}^\infty$ is one of the second and third cases in Lemma \ref{lem: taking a pre-approximation sequence} since the first case is ruled out by the previous case. Namely there is a quasi-rainbow $\{J_n^{i_0+1}\}_{n=1}^\infty$ in some lamination system of $\mathcal{C}$ such that either $\{(g_{\beta \circ \gamma(n)}, J_n^{i_0+1})\}_{n=1}^\infty$ is a pre-approximation sequence at $y_\infty^{i_0+1}$ or  $\{(g_{\beta \circ \gamma(n)}^{-1}, J_n^{i_0+1})\}_{n=1}^\infty$ is a pre-approximation sequence at $x_\infty^{i_0+1}.$  Applying Lemma \ref{lem: taking a pre-approximation sequence} again to the index $i_0+2,$ we can take a subindex $\delta$ of $\id_\NN$ so that $\{g_{\beta \circ \gamma\circ \delta(n)}\}_{n=1}^\infty$ is one of the second and third cases in Lemma \ref{lem: taking a pre-approximation sequence}, that is  there is a quasi-rainbow $\{I_n^{i_0+2}\}_{n=1}^\infty$ in some lamination system of $\mathcal{C}$ such that either $\{(g_{\beta \circ \gamma\circ \delta(n)}, I_n^{i_0+2})\}_{n=1}^\infty$ is a pre-approximation sequence at $y_\infty^{i_0+2}$ or  $\{(g_{\beta \circ \gamma\circ \delta(n)}^{-1}, I_n^{i_0+2})\}_{n=1}^\infty$ is a pre-approximation sequence at $x_\infty^{i_0+2}.$ Now for brevity we write $h_n=g_{\beta \circ \gamma\circ \delta(n)}$ and  $I_n^{i_0+1}=J_{\delta(n)}^{i_0+1}$ for all $n\in \NN.$

First we consider the case where $\{(h_n, I_n^{i_0+1})\}_{n=1}^\infty$ and $\{(h_n, I_n^{i_0+2})\}_{n=1}^\infty$ are pre-approximation sequences at  $y_\infty^{i_0+1}$ and  $y_\infty^{i_0+2},$ respectively. Since $\{I_n^{i_0+1}\}_{n=1}^\infty$ and $\{I_n^{i_0+2}\}_{n=1}^\infty$ are quasi-rainbows, there is a number $N$ in $\NN$ such that $\overline{I_N^{i_0+1}}\subset N_\epsilon (y_\infty^{i_0+1})$ and $\overline{I_N^{i_0+2}}\subset N_\epsilon (y_\infty^{i_0+2}).$ Hence we can take a subindex $\zeta$ of $\id_\NN$ so that  $\overline{I_{\zeta(1)}^{i_0+1}}\subset N_\epsilon (y_\infty^{i_0+1})$ and $\overline{I_{\zeta(1)}^{i_0+2}}\subset N_\epsilon (y_\infty^{i_0+2}).$ Since $\overline{I_{\zeta(1)}^{i_0+1}}\cap \overline{I_{\zeta(1)}^{i_0+2}}=\emptyset,$ the sequence $\{h_{\zeta(n)}\}_{n=1}^\infty$ is an approximation sequence to $\{y_\infty^{i_0+1}, y_\infty^{i_0+2}\}$ by Lemma \ref{lem: the third type}. Hence we can choose a number $M$ in $\NN$ such that the element $h_{\zeta(M+1)}\circ h_{\zeta(M)}^{-1}$ has exactly one fixed point in each of $N_\epsilon (y_\infty^{i_0+1})$ and $N_\epsilon (y_\infty^{i_0+2}).$ On the other hand, the element $y_\infty^{i_0} \in \Fix_{h_{\zeta(M+1)}\circ h_{\zeta(M)}^{-1}}.$ Then the element  $h_{\zeta(M+1)}\circ h_{\zeta(M)}^{-1}$ has at least three fixed points by the choice of $\epsilon$ and this  is a contradiction since $h_{\zeta(M+1)}\circ h_{\zeta(M)}^{-1}$  is hyperbolic. Likewise  the case where $\{(h_n^{-1}, I_n^{i_0+1})\}_{n=1}^\infty$ and $\{(h_n^{-1}, I_n^{i_0+2})\}_{n=1}^\infty$ are pre-approximation sequences at  $x_\infty^{i_0+1}$ and  $x_\infty^{i_0+2}$ respectively is not possible.

Next we consider the case where  $\{(h_n^{-1}, I_n^{i_0+1})\}_{n=1}^\infty$ and $\{(h_n, I_n^{i_0+2})\}_{n=1}^\infty$ are pre-approximation sequences at  $x_\infty^{i_0+1}$ and  $y_\infty^{i_0+2},$ respectively. As $\{I_n^{i_0+1}\}_{n=1}^\infty$ and $\{I_n^{i_0+2}\}_{n=1}^\infty$ are quasi-rainbows at $x_\infty^{i_0+1}$ and $y_\infty^{i_0+2},$ respectively, there is a subindex $\zeta$ of $\id_\NN$ such that $\overline{I_{\zeta(1)}^{i_0+1}}\subset N_\epsilon(x_\infty^{i_0+1})$ and $\overline{I_{\zeta(1)}^{i_0+2}}\subset N_\epsilon(y_\infty^{i_0+2}).$ Then since $x_\infty^{i_0}\notin \overline{I_{\zeta(1)}^{i_0+1}}$ and $y_\infty^{i_0}\notin \overline{I_{\zeta(1)}^{i_0+2}},$ by Lemma \ref{lem: the second type}, the sequence $\{h_{\zeta(n)}^{-1}\}_{n=1}^\infty$ is an approximation sequence to $\{x_\infty^{i_0}, x_\infty^{i_0+1}\}$ and  the sequence  $\{h_{\zeta(n)}\}_{n=1}^\infty$ is an approximation sequence to $\{y_\infty^{i_0}, y_\infty^{i_0+2}\}.$ Now we write $\Fix_{h_{\zeta(n+1)}^{-1}\circ h_{\zeta(n)}}=\{x_\infty^{i_0}, a_n\}$ and $\Fix_{h_{\zeta(n+1)}\circ h_{\zeta(n)}^{-1}}=\{y_\infty^{i_0}, b_n\}.$ Then the two sequences $\{a_n\}_{n=1}^\infty$ and $\{b_n\}_{n=1}^\infty$ converge to $x_\infty^{i_0+1}$ and $y_\infty^{i_0+2}.$ Note that for each $n\in \NN,$
$$h_{\zeta(n)}(\Fix_{h_{\zeta(n+1)}^{-1}\circ h_{\zeta(n)}})=h_{\zeta(n)}(\Fix_{h_{\zeta(n)}^{-1}\circ h_{\zeta(n+1)}})
=\Fix_{h_{\zeta(n)}\circ(h_{\zeta(n)}^{-1}\circ h_{\zeta(n+1)})\circ h_{\zeta(n)}^{-1}}=\Fix_{ h_{\zeta(n+1)}\circ h_{\zeta(n)}^{-1}}.$$
Hence $h_{\zeta(n)}(a_n)=b_n$ for all $n\in \NN.$ Now we have that $\{x_{ \beta' \circ \gamma\circ \delta \circ \zeta(n)}^{i_0+1}\}_{n=1}^\infty$ and $\{x_{\beta' \circ  \gamma\circ \delta \circ \zeta(n)}^{i_0+2}\}_{n=1}^\infty$ converge to $x_\infty ^{i_0+1}$ and $x_\infty^{i_0+2},$ respectively and that $\{h_{\zeta(n)}(x_{\beta' \circ \gamma\circ \delta \circ \zeta(n)}^{i_0+1})\}_{n=1}^\infty$ and $\{h_{\zeta(n)}(x_{\beta' \circ  \gamma\circ \delta \circ \zeta(n)}^{i_0+2})\}_{n=1}^\infty$ converge to $y_\infty ^{i_0+1}$ and $y_\infty^{i_0+2},$ respectively, where $\beta'$ is the subindex of $\id_\NN$ satisfying that $\beta=\alpha\circ \beta'$ . Moreover for each $n\in \NN,$ $a_n$ is a fixed point of the hyperbolic element $h_{\zeta(n+1)}^{-1}\circ h_{\zeta(n)}$, and  $\{a_n\}_{n=1}^\infty$ and $\{h_{\zeta(n)}(a_n)\}_{n=1}^\infty$ converge to $x_\infty^{i_0+1}$ and $y_\infty^{i_0+2},$ respectively. Also the sequence $\{h_{\zeta(n)}\}_{n=1}^\infty$ is an approximation sequence.  Therefore by Lemma \ref{lem: having the convergence property}, the sequence  $\{h_{\zeta(n)}\}_{n=1}^\infty$ has the convergence property and so $\{g_n\}_{n=1}^\infty$ has the convergence property.  However this is in contradiction to the assumption that $\{g_n\}_{n=1}^\infty$ does not have the convergence property. Likewise we can see that the case where $\{(h_n, I_n^{i_0+1})\}_{n=1}^\infty$ and $\{(h_n^{-1}, I_n^{i_0+2})\}_{n=1}^\infty$ are pre-approximation sequences at  $y_\infty^{i_0+1}$ and  $x_\infty^{i_0+2},$ respectively, is not possible. Therefore, there is no such subindex $\beta.$ 

Then we consider the case where there is a subindex $\beta$ of $\alpha$ such that for each $i\in \ZZ_3,$ we can take a quasi-rainbow $\{I_n^{i}\}_{n=1}^\infty$ in some lamination system of $\mathcal{C}$ at $y_\infty^i$ so that $\{(g_{\beta(n)}, I_n^i)\}_{n=1}^\infty$ is a pre-approximation sequence at $y_\infty^i.$ Then there is a subindex $\gamma$ of $\beta$ such that for each $i\in \ZZ_3,$ $\overline{I_{\gamma(1)}^i} \subset N_\epsilon(y_\infty^i).$ By the choice of $\epsilon,$ for each $i\in \ZZ_3,$ $\overline{I_{\gamma(1)}^{i}}\cap \overline{I_{\gamma(1)}^{i+1}}=\emptyset.$ Hence by Lemma \ref{lem: the third type},  for each $i\in \ZZ_3,$ the sequence $\{g_{\gamma(n)}\}_{n=1}^\infty$ is an approximation sequence to $\{y_\infty^{i}, y_\infty^{i+1}\}.$ Then there is a number  $n_0\in \NN$ such that for each $i\in \ZZ_3,$  $g_{\gamma(n_0+1)}\circ g_{\gamma(n_0)}^{-1}$ has a fixed point in $N_\epsilon(y_\infty^i).$ This implies that the element  $g_{\gamma(n_0+1)}\circ g_{\gamma(n_0)}^{-1}$ has at least three fixed point. This is a contradiction since $g_{\gamma(n_0+1)}\circ g_{\gamma(n_0)}^{-1}$ is hyperbolic. Therefore, there is no such $\beta.$ Likewise, there is no subindex $\beta$ of $\alpha$ such that for each $i\in \ZZ_3,$ we can take a quasi-rainbow $\{I_n^{i}\}_{n=1}^\infty$ in some lamination system of $\mathcal{C}$ at $x_\infty^i$ so that $\{(g_{\beta(n)}^{-1}, I_n^i)\}_{n=1}^\infty$ is a pre-approximation sequence at $x_\infty^i.$ 

Now we apply Lemma \ref{lem: taking a pre-approximation sequence} to the index $1.$ Then by the previous discussion, the first case of Lemma  \ref{lem: taking a pre-approximation sequence} can not occur so there is a subindex $\beta$ of $\id_\NN$ and a quasi-rainbow $\{I_n^1\}_{n=1}^\infty$ in some lamination system of $\mathcal{C}$ such that either $\{(g_{\alpha\circ \beta(n)}, I_n^1)\}_{n=1}^\infty$ is a pre-approximation sequence at $y_\infty^1$ or  $\{(g_{\alpha\circ \beta(n)}^{-1}, I_n^1)\}_{n=1}^\infty$ is a pre-approximation sequence at $x_\infty^1.$ Then we apply Lemma \ref{lem: taking a pre-approximation sequence} to the index $2.$ Like the previous case,  there is a subindex $\gamma$ of $\id_\NN$ and a quasi-rainbow $\{I_n^2\}_{n=1}^\infty$ in some lamination system of $\mathcal{C}$ such that either $\{(g_{\alpha\circ \beta\circ\gamma(n)}, I_n^2)\}_{n=1}^\infty$ is a pre-approximation sequence at $y_\infty^2$ or $\{(g_{\alpha\circ \beta\circ\gamma(n)}^{-1}, I_n^2)\}_{n=1}^\infty$ is a pre-approximation sequence at $x_\infty^2.$ Again we apply Lemma \ref{lem: taking a pre-approximation sequence} to the index $3.$ Then there is a subindex $\delta$ of $\id_\NN$  and a quasi-rainbow $\{I_n^3\}_{n=1}^\infty$ such that either  $\{(g_{\alpha\circ \beta\circ\gamma\circ \delta(n)}, I_n^3)\}_{n=1}^\infty$ is a pre-approximation sequence at $y_\infty^3$ or $\{(g_{\alpha\circ \beta\circ\gamma\circ \delta(n)}^{-1}, I_n^3)\}_{n=1}^\infty$ is a pre-approximation sequence at $x_\infty^3.$ For brevity, we write $h_n=g_{\alpha\circ \beta\circ \gamma \circ \delta(n)},$ $J_n^1=I_{\gamma\circ \delta(n)}^1,$ $J_n^2=I_{\delta(n)}^2,$ and $J_n^3=I_n^3$ for all $n\in \NN.$ Also we write $u_n^i=x_{ \beta\circ \gamma \circ \delta(n)}^i$ and $v_n^i=y_{ \beta\circ \gamma \circ \delta(n)}^i$ for all $n\in \NN$ and for all $i\in \ZZ_3.$ 

Then by the previous discussion, there are only two cases which we need to deal with.
The first case is that there is an index $i_0 \in \ZZ_3$ such that  the two sequences $\{(h_n, J_n^{i_0})\}_{n=1}^\infty$ and $\{(h_n, J_n^{i_0+1})\}_{n=1}^\infty$ are pre-approximation sequences at $y_\infty^{i_0}$ and $y_\infty^{i_0+1},$ respectively, and the sequence  $\{(h_n^{-1}, J_n^{i_0+2})\}_{n=1}^\infty$ is a pre-approximation sequence at $x_\infty^{i_0+2}.$ 
 The second case is that there is an index $i_0 \in \ZZ_3$ such that  the two sequences $\{(h_n^{-1}, J_n^{i_0})\}_{n=1}^\infty$ and $\{(h_n^{-1}, J_n^{i_0+1})\}_{n=1}^\infty$ are pre-approximation sequences at $x_\infty^{i_0}$ and $x_\infty^{i_0+1},$ respectively, and the sequence  $\{(h_n, J_n^{i_0+2})\}_{n=1}^\infty$ is a pre-approximation sequence at $y_\infty^{i_0+2}.$ 

First we consider the first case. There is a subindex $\zeta$ of $\id_\NN$ such that $\overline{J_{\zeta(1)}^{i_0}} \subset N_\epsilon(y_\infty^{i_0})$ and $\overline{J_{\zeta(1)}^{i_0+1}} \subset N_\epsilon(y_\infty^{i_0+1}).$  By Lemma \ref{lem: the third type}, $\{h_{\zeta(n)}\}_{n=1}^\infty$ is an approximation sequence to $\{y_\infty^{i_0}, y_\infty^{i_0+1}\}.$ Note that for each $n\in \NN,$ $h_{\zeta(n+1)}\circ h_{\zeta(n)}^{-1}$ is hyperbolic so $h_{\zeta(n+1)}^{-1}\circ h_{\zeta(n)}$ is also hyperbolic since $$h_{\zeta(n+1)}^{-1}\circ h_{\zeta(n)}=h_{\zeta(n)}^{-1}\circ(h_{\zeta(n)}\circ h_{\zeta(n+1)}^{-1})\circ h_{\zeta(n)}=h_{\zeta(n)}^{-1}\circ(h_{\zeta(n+1)}\circ h_{\zeta(n)}^{-1})^{-1}\circ h_{\zeta(n)}.$$ Moreover, since $\{(h_{\zeta(n)}^{-1}, J_{\zeta(n)}^{i_0+2})\}_{n=1}^\infty$ is a pre-approximation sequence at $x_\infty^{i_0+2},$ there is a point $\tilde{x}$ in $S^1$ such that $h_{\zeta(n)}^{-1}(\tilde{x})\in J_{\zeta(n)}^{i_0+2}$ for all $n\in\NN$ and so $$h_{\zeta(n+1)}^{-1}(\tilde{x}) =(h_{\zeta(n+1)}^{-1}\circ h_{\zeta(n)})( h_{\zeta(n)}^{-1}(\tilde{x}))\in (h_{\zeta(n+1)}^{-1}\circ h_{\zeta(n)})(J_{\zeta(n)}^{i_0+2}) \cap J_{\zeta(n+1)}^{i_0+2}\subset (h_{\zeta(n+1)}^{-1}\circ h_{\zeta(n)})(J_{\zeta(n)}^{i_0+2}) \cap J_{\zeta(n)}^{i_0+2}.$$ Hence we can claim that for each $n\in \NN,$ $\overline{J_{\zeta(n)}^{i_0+2}}$ has a fixed point of the hyperbolic element $h_{\zeta(n+1)}^{-1}\circ h_{\zeta(n)}.$  For each $n\in \NN,$ since $(h_{\zeta(n+1)}^{-1}\circ h_{\zeta(n)})(J_{\zeta(n)}^{i_0+2}) \cap J_{\zeta(n)}^{i_0+2}\neq \emptyset,$ there are three possible cases: $(h_{\zeta(n+1)}^{-1}\circ h_{\zeta(n)})(J_{\zeta(n)}^{i_0+2}) \subset J_{\zeta(n)}^{i_0+2};$ $(h_{\zeta(n+1)}^{-1}\circ h_{\zeta(n)})(J_{\zeta(n)}^{i_0+2})^* \subset  J_{\zeta(n)}^{i_0+2};$ $J_{\zeta(n)}^{i_0+2}\subset (h_{\zeta(n+1)}^{-1}\circ h_{\zeta(n)})(J_{\zeta(n)}^{i_0+2}).$ In the first and third cases, the claim follows at once so we only need to consider the second case. Fix $n_0\in \NN.$ Assume that  $(h_{\zeta(n_0+1)}^{-1}\circ h_{\zeta(n_0)})(J_{\zeta(n_0)}^{i_0+2})^* \subset  J_{\zeta(n_0)}^{i_0+2}.$ If there is a point $a$ in $\Fix_{h_{\zeta(n_0+1)}^{-1}\circ h_{\zeta(n_0)}}-\overline{J_{\zeta(n_0)}^{i_0+2}},$ then   
$$a\in (J_{\zeta(n_0)}^{i_0+2})^* \subset (h_{\zeta(n_0+1)}^{-1}\circ h_{\zeta(n_0)})(J_{\zeta(n_0)}^{i_0+2}).$$ This implies that $ a=(h_{\zeta(n_0+1)}^{-1}\circ h_{\zeta(n_0)})^{-1}(a)\in J_{\zeta(n_0)}^{i_0+2}$ as  $\Fix_{h_{\zeta(n_0+1)}^{-1}\circ h_{\zeta(n_0)}} =\Fix_{(h_{\zeta(n_0+1)}^{-1}\circ h_{\zeta(n_0)})^{-1}},$ so this is a contraction. Therefore $\Fix_{h_{\zeta(n_0+1)}^{-1}\circ h_{\zeta(n_0)}}\subset \overline{J_{\zeta(n_0)}^{i_0+2}}$ and the claim is proved. Therefore we can take a sequence $\{a_n\}_{n=1}^\infty$ so that for each $n\in \NN,$ $a_n\in \overline{J_{\zeta(n)}^{i_0+2}} \cap \Fix_{h_{\zeta(n+1)}^{-1}\circ h_{\zeta(n)}}.$ Then as $\{J_{\zeta(n)}^{i_0+2}\}_{n=1}^\infty$ is a quasi-rainbow at $x_\infty^{i_0+2},$ the sequence $\{a_n\}_{n=1}^\infty$ converges to $x_\infty^{i_0+2}.$   Note again that for each $n\in \NN,$
$$h_{\zeta(n)}(\Fix_{h_{\zeta(n+1)}^{-1}\circ h_{\zeta(n)}})=h_{\zeta(n)}(\Fix_{h_{\zeta(n)}^{-1}\circ h_{\zeta(n+1)}})
=\Fix_{h_{\zeta(n)}\circ(h_{\zeta(n)}^{-1}\circ h_{\zeta(n+1)})\circ h_{\zeta(n)}^{-1}}=\Fix_{ h_{\zeta(n+1)}\circ h_{\zeta(n)}^{-1}}.$$
Hence the sequence $\{h_{\zeta(n)}(a_n)\}_{n=1}^\infty$ has at most two subsequential limits belonging to $\{y_\infty^{i_0}, y_\infty^{i_0+1}\}$ as the sequence $\{\Fix_{h_{\zeta(n+1)}\circ h_{\zeta(n)}^{-1}}\}_{n=1}^\infty$ converges to $\{y_\infty^{i_0}, y_\infty^{i_0+1}\}.$ Therefore,  there is a subindex $\eta$ of $\id_\NN$ such that the sequence $\{h_{\zeta\circ \eta(n)}(a_{\eta(n)})\}_{n=1}^\infty$ converges to $y_\infty^{j_0}$ for some $j_0\in \{i_0, i_0+1\}.$ Now we have that the two sequences $\{u_{\zeta\circ \eta(n)}^{i_0+2}\}_{n=1}^\infty$ and $\{u_{\zeta\circ \eta(n)}^{j_0}\}_{n=1}^\infty$ converge to $x_\infty^{i_0+2}$ and $x_\infty^{j_0},$ respectively, and  two sequences $\{h_{\zeta\circ \eta(n)}(u_{\zeta\circ \eta(n)}^{i_0+2})\}_{n=1}^\infty$ and $\{h_{\zeta\circ \eta(n)}(u_{\zeta\circ \eta(n)}^{j_0})\}_{n=1}^\infty$ converge to $y_\infty^{i_0+2}$ and $y_\infty^{j_0},$ respectively. Moreover, for each $n\in \NN,$ $a_{\eta(n)}$ is a fixed point of the hyperbolic element $h_{\zeta(\eta(n)+1)}^{-1}\circ h_{\zeta(\eta(n))},$ and  the two sequences $\{a_{\eta(n)}\}_{n=1}^\infty$ and $\{h_{\zeta\circ \eta(n)}(a_{\eta(n)})\}_{n=1}^\infty$ converge to $x_\infty^{i_0+2}$ and $y_\infty^{j_0},$ respectively. Also the sequence $\{h_{\zeta\circ \eta(n)}\}_{n=1}^\infty$ is an approximation sequence as  the sequence $\{h_{\zeta(n)}\}_{n=1}^\infty$ is an approximation sequence. Therefore by Lemma \ref{lem: having the convergence property}, the sequence $\{h_{\zeta\circ \eta(n)}\}_{n=1}^\infty$ has the convergence property so the sequence $\{g_n\}_{n=1}^\infty$ has the convergence property. This is in contradiction to the assumption that  the sequence $\{g_n\}_{n=1}^\infty$ does not have the convergence property. Hence, the first case is not possible. Likewise, in the second case, we can show that   the sequence $\{g_n^{-1}\}_{n=1}^\infty$ has the convergence property and this implies that the sequence $\{g_n\}_{n=1}^\infty$ has the convergence property. This is also a contradiction. Therefore, we can conclude that   there is no sequence $\{g_n\}_{n=1}^\infty$ of distinct elements of $G$ that does not have the convergence property. Thus, $G$ is a convergence group.

 \end{proof}

\section{The elementary geometry and topology of $2$-dimensional hyperbolic orbifolds}\label{sec:elementary facts of orbifolds}
\subsection{Simple closed curves and simple closed geodesics in $2$-dimensional hyperbolic orbifolds.}
Let $S$ be a complete $2$-dimensional hyperbolic orbifold and $\alpha$ be a closed curve from $I=[0,1]$ to $S.$ Fix $G(S).$ Then
a continuous map $\overline{\alpha}$ from $\RR$ to $\HH$ is called a \emph{full lift} of $\alpha$ if $\pi_S(\overline{\alpha})$ is periodic with period $1$ and  the following diagram commutes. 
 \begin{center}
\begin{tikzcd}
\RR \arrow[r, "\overline{\alpha}"] & \HH \arrow[d, "\pi_S"] \\
{[0,1]} \arrow[r, "\alpha"] \arrow[u, "i"] & S
 \end{tikzcd}
\end{center}

For our purposes, we say that the curve $\alpha$ is called \emph{regular} if it does not pass through a cone point, that is, $\alpha(I) \cap \Sigma(S)=\emptyset.$ Namely, a regular curve is a curve in $S-\Sigma(S)$. Assume that $\alpha$ is regular and simple. Note that as $\pi_S$ is a branched covering, the restriction $\pi_S|(\HH-\pi_S^{-1}(\Sigma(S)))$ is a covering map onto $S^o.$ Hence once a lift of $\alpha(0)$ is chosen, a lift $\tilde{\alpha}$ of the curve $\alpha$ is uniquely determined and there is a unique element $g \in G(S)$ such that $\tilde{\alpha}(1)=g(\tilde{\alpha}(0)).$ Fix a lift of $\alpha(0).$ Then we can define a map $\beta$ from $\RR$ to $\HH$ as 
$$\beta(t)=\tilde{\alpha}(t)$$
for all $t\in [0,1]$
and 
$$\beta(t+1)=g(\beta(t))$$
for all $t\in \RR$ where $g$ is the element such that $\tilde{\alpha}(1)=g(\tilde{\alpha}(0)).$
By construction, $\beta$ is obviously a full lift of $\alpha.$ Moreover, since the restriction $\pi_S|(\HH-\pi_S^{-1}(\Sigma(S)))$ is a covering map onto $S^o,$ any full lift $\overline{\alpha}$ of $\alpha$ is $h\beta$ for some $h\in G(S).$
Now we summarize these properties in the following propositions. 

\begin{prop}
Let $S$ be a complete $2$-dimensional hyperbolic orbifold and $\alpha$ be a regular simple closed curve in $S.$ Fix $G(S).$ Suppose that $\tilde{p}$ is a lift of $\alpha(0).$ Then there is a unique full lift $\overline{\alpha}$ of $\alpha$ such that $\overline{\alpha}(0)=\tilde{p}.$
\end{prop}

\begin{prop}
Let $S$ be a complete $2$-dimensional hyperbolic orbifold and $\alpha$ be a regular simple closed curve in $S.$ Fix $G(S).$ Suppose that $\overline{\alpha}$ is a full lift of $\alpha.$ Then there is a unique element $g$ in $G(S)$ such that $\overline{\alpha}(t+1)=g(\overline{\alpha}(t))$ for all $t\in \RR.$ We call such a $g$ the \emph{holonomy} of $\overline{\alpha}$ and denote $g$ by $\hol(\overline{\alpha}).$
\end{prop}  

\begin{prop}
Let $S$ be a complete $2$-dimensional hyperbolic orbifold and $\alpha$ be a regular simple closed curve in $S.$ Fix $G(S).$ Suppose that $\beta_1$ and $\beta_2$ are full lifts of $\alpha.$ Then there is a unique element $g$ such that $\beta_2=g\beta_1.$ Therefore, $\hol(\beta_2)=g\hol(\beta_1)g^{-1}.$
\end{prop}

Now we can classify regular simple closed curves depending on their holonomies. 

\begin{defn} Let $S$ be a complete $2$-dimensional hyperbolic orbifold and $\alpha$ be a regular simple closed curve in $S.$ Fix $G(S).$ 
\begin{enumerate}
\item If the holonomy of a full lift of $\alpha$ is trivial, $\alpha$ is said to be \emph{trivial}.
\item  If the holonomy of a full lift of $\alpha$ is not trivial and is elliptic,  $\alpha$ is said to be \emph{elliptic}.
\item If the holonomy of a full lift of $\alpha$ is not trivial and is parabolic, $\alpha$ is said to be \emph{parabolic}.
\item If the holonomy of a full lift of $\alpha$ is not trivial and is hyperbolic, $\alpha$ is said to be \emph{hyperbolic}.
\end{enumerate}
\end{defn}

Note that above definitions does not depends on the choice of a full lift since holonomies of full lifts of $\alpha$ are in same conjugacy class. 

\begin{lem}\label{lem : non essential curves}
Let $S$ be a complete $2$-dimensional hyperbolic orbifold and $\alpha$ be a regular simple closed curve from $I=[0,1]$ to  $S.$ Fix $G(S).$ Then the following holds.
\begin{enumerate}
\item If $\alpha$ is trivial, then $\alpha$ is freely homotopic to a point in $S^o.$
\item If $\alpha$ is elliptic, then $\alpha$ bounds a  cone point.
\item If $\alpha$ is parabolic, then $\alpha$ bounds a  puncture.
\end{enumerate}

\end{lem}
\begin{proof}
First, we assume that $\alpha$ is trivial. Choose a full lift $\overline{\alpha}$ and denote $\hol(\overline{\alpha})$ by $g.$ As $\alpha$ is trivial, $g$ is trivial. Then by definition, $\overline{\alpha}|[0,1]$ is a simple closed curve in $\HH.$ Hence there is a closed disk $B$ whose boundary is $\overline{\alpha}([0,1]).$ Now we want to show that there is no ramification point in $B$.  Obviously, the boundary $\partial B$ has no ramification point as $\alpha$ is regular. Assume that there is a ramification point $x$ in the interior $\intr B$ and $h$ is a generator of the stabilizer $G_x.$ Since $\alpha$ is a regular simple closed curve, the pre-image $\pi_S^{-1}(\alpha(I))$ is a $1$-dimensional manifold. Hence $\pi_S^{-1}(\alpha(I))=\cup\{f(\partial B) : f\in G(S)\}$ and $\pi_S^{-1}(\alpha(I))$ is a disjoint union of simple closed curves. Therefore, if $h(\partial B)\neq \partial  B,$ then there are two possible cases: $h(B) \subsetneq B$ or $B \subsetneq h(B).$ Assume that  $h(B)\neq B$ and  choose a shortest geodesic segment $\ell$ joining $x$ and $\partial B.$ Now we say that  the other end point of $\ell$ is  $y.$ As $h$ is an isometry,  $h(\ell)$ is also a shortest geodesic segment joining $x$ and $h(\partial B).$  If $h(y)$ is in the interior $\intr B$ of $B,$ then $h(B) \subset \intr B$ but this implies that the length of shortest geodesic segments joining $x$ and $h(\partial B)$ is less than the length of $\ell.$ It is a contradiction. Similarly,  $h(y)$ cannot be in the exterior $\extr B$. Therefore,   $h(B)=B$ and so $h(\partial B)=\partial B.$ However, this is also a contradiction since $\alpha$ is simple and  $\overline{\alpha}|[0,1]$ is a simple closed curve in $\HH.$ Thus, there is no ramification point in $B$ so $\pi_S(B)$ is a disk with $\partial \pi_S(B)=\alpha(I).$ 
 
Next, we assume that $\alpha$ is elliptic. Choose a full lift $\overline{\alpha}$ and denote $\hol(\overline{\alpha})$ by $g.$ As $\alpha$ is elliptic, $g$ is an elliptic element of order $n$ for some $n\in\NN$ with $2\leq n.$ Then the full lift is periodic with period $n$ and so $\overline{\alpha}|[0,n]$ is a simple closed curve in $\HH.$ Moreover, $\overline{\alpha}|[0,n]$ bounds a disk with a ramification point $x$ which is the fixed point of $g.$  Let $B$ be the closed disk with $\partial B = \overline{\alpha}([0,n]).$ As $\alpha$ is regular, $\partial B$ has no ramification point. Now we claim that there is no ramification point except $x$ in $\intr B.$ Assume that there is a ramification point $y$ in $\intr B -\{x\}$ and $h$ is a generator of the stabilizer $G_y.$ Since $\alpha$ is a regular simple closed curve, the pre-image $\pi_S^{-1}(\alpha(I))$ is a $1$-dimensional manifold. Hence $\pi_S^{-1}(\alpha(I))=\cup\{f(\partial B) : f\in G(S)\}$ and $\pi_S^{-1}(\alpha(I))$ is a disjoint union of simple closed curves. Assume that $h(\partial B)\neq \partial B$ and choose a shortest geodesic segment $\ell$ joining $y$ and $\partial B.$ Then $h(\ell)$ is also a shortest geodesic segment joining $y$ and $h(\partial B)$ as $h$ is an isometry. Now let $z$ be the other end point of $\ell.$ If $h(z)\in \intr B,$ then $h(B) \subset \intr B$ and this implies that the length of shortest geodesic segments joining $y$ and $h(\partial B)$ is less than the length of $\ell.$ This is a contradiction. Similarly, the case where $h(z)\in \extr B$ is not possible, so $h(\partial B)=\partial B.$ Note that $A=\{g^i\overline{\alpha}(0) : i \in \{0,1, \cdots, n-1\}\}$ equals $\pi_S^{-1}(\alpha(0))\cap \partial B.$ Hence $h$ preserves $A$ so $h\overline{\alpha}(0)=g^m\overline{\alpha}(0)$ for some $m\in  \{1, \cdots, n-1\}.$ Then $g^{-m}h$ fixes  $\overline{\alpha}(0).$ If $ g^{-m}h$ is not trivial and is elliptic, this is in contradiction with that $\alpha$ is regular. Therefore,  $g^{-m}h$ is trivial but this is also a contradiction since $x\neq y.$ Thus, there is no such a $y$ so we can conclude that $\pi_S(B)$ is a closed disk with one cone point and $\partial \pi_S(B) = \alpha(I). $

Finally, we consider the case where $\alpha$ is parabolic. Choose a full lift $\overline{\alpha}$ and denote $\hol{\overline{\alpha}}$ by $g.$ As $\alpha$ is parabolic, $g$ is parabolic. Then by definition, $\overline{\alpha}$ is an embedding and there is a point $x_\infty$ in $\partial_\infty \HH$ which is the fixed point of $g$ such that $\displaystyle \lim_{t\to \infty} \overline{\alpha}(t)=\lim_{t\to -\infty} \overline{\alpha}(t)=x_\infty.$ Hence in $\hat{\CC},$ there is a closed disk $B$ such that $\intr B \subset \HH$ and $\partial B= \overline{\alpha}(\RR)\cup\{x_\infty\}.$ Now we want to show that $B$ has no ramification point. Obviously, there is no ramification point in $\partial B$ as $\alpha$ is regular. Assume that there is a ramification point $y$ in $\intr B$ which is the fixed point of a non-trivial elliptic element $h\in G(S).$ Then $h(B)\cap B\neq \emptyset$ and since $h(x_\infty)\neq x_\infty,$  $h(\partial B) \neq \partial B$  and $h(\partial B)\cap \partial B\neq \emptyset$ (See Figure \ref{fig:parabolic}). This is a contradiction since $\alpha$ is simple. Therefore, there is no ramification point in $B.$ Thus, $\pi_S(B)$ is a disk with one puncture and $\partial \pi_S(B)=\alpha(I).$ 
\end{proof}

\begin{rmk}
The converse of each statement in the proposition is also true. 
\end{rmk}

\begin{figure}
\begin{center}
\includegraphics[width=0.5\textwidth]{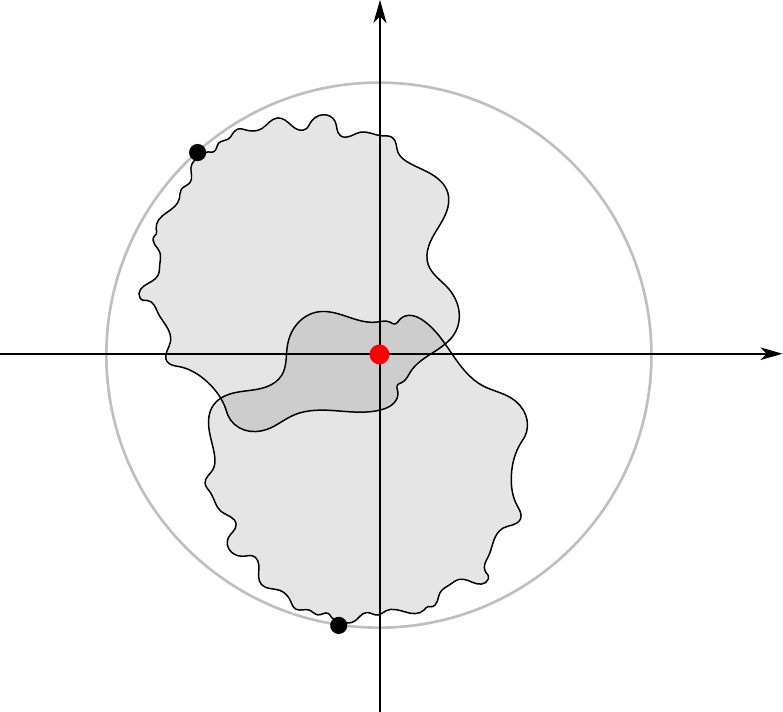}
\end{center}
\caption{This is a schematic picture on the Poincaré disk when $B$ contains a ramification point. The black marked points correspond to $x_\infty$ and $h(x_\infty)$, and the shaded regions are $B$ and $h(B).$ The red point at the origin is the ramification point and $h$ is the rational rotation fixing the red point.}
\label{fig:parabolic}
\end{figure}

A regular simple closed curve is said to be \emph{essential} if it does not bound a disk, a disk with one cone point, or a disk with one puncture. 

\begin{prop}
Let $S$ be a complete $2$-dimensional hyperbolic orbifold and $\alpha$ be an essential regular simple closed curve in $S.$ Fix $G(S).$ Then $\alpha$ is hyperbolic and every full lift of $\alpha$ is a quasi-geodesic in $\HH.$
\end{prop}
When $\alpha$ is a hyperbolic simple closed curve, each full lift is a quasi-geodesic so it has a unique geodesic connecting two end points of the full lift. So $\pi_S^{-1}(\alpha)$ is a disjoint union of quasi-geodesics and is $G(S)$-invariant. Now let $Q$ be the set of images of full lifts of $\alpha$ and  $Q'$  be the set of geodesic representatives of elements of $Q,$ namely $Q'$ is the set of hyperbolic axises of holonomies of all full lifts of $\alpha.$ Then $Q'$ is also a disjoint union of bi-infinite geodesics and is $G(S)$-invariant. By  construction, for any geodesic $\ell$ in $Q'$, $Q'=\displaystyle \bigcup G(S)\cdot \ell.$  Now we say that  $\pi_S(Q')$ is the \emph{geodesic realization} of $\alpha.$ Note that geodesic realizations are simple geodesics. 

The following lemma was proved in \cite{Suzzi_Valli_2015} in a  slightly different form. 
\begin{lem}[\cite{Suzzi_Valli_2015}] \label{lem : finding the geodesic realization}
Let $S$ be a complete $2$-dimensional hyperbolic orbifold and $\alpha$ be an essential regular simple closed curve in $S.$ Fix $G(S).$ The following holds. 
\begin{enumerate}
\item If $\alpha$ does not bound a disk with two cone points of order two, then the geodesic realization of $\alpha$ is in the free homotopy class of $\alpha$ in $S^o.$
\item If $\alpha$ bounds a disk with two cone points of order two, then a geodesic segment connecting the cone points is the geodesic realization of $\alpha.$
\end{enumerate}
\end{lem}
\begin{proof}
First assume that  $\alpha$ is a hyperbolic simple closed curve. Choose a full lift $\overline{\alpha}$ of $\alpha$ and denote $\hol(\overline{\alpha})$ by $g.$ Since $\overline{\alpha}$ is a quasi-geodesic and $J=\overline{\alpha}(\RR)$ is preserved by $g,$ the hyperbolic axis $\ell$ of $g$ is the geodesic representative of $\overline{\alpha}.$
Now we say that $x$ and $y$ in $\partial_\infty \HH$ are two end points of $\ell$ and write $\overline{\ell}$ and $\overline{J}$ for $\ell\cup\{x, y\}$ and $J\cup \{x,y\},$ respectively. Now we think of $\overline{\ell}$ and $\overline{J}$ as images of two proper embedding $\delta_\ell$ and $\delta_J$, respectively,  from $[0,1]$ to $\overline{\HH}$ such that $\delta_\ell(0)=\delta_J(0)=x$ and $\delta_\ell (1)=\delta_J(1)=y.$ 

Choose a ramification point $r$ in $\HH-(\ell \cup J)$ and a generator $h$ of the stabilizer $G_r.$ Note that since $\alpha$ is regular, there is no ramification point in $J$. Now we write $\pi_{G_r}$ for the quotient map from $\HH$ to $\HH/G_r.$ If $h$ preserves $\ell$, then the order of $h$ must be two and so $r$ must be in $\ell$. This is a contradiction. Hence,  $h$ does not preserve $\ell.$ Therefore,  $\pi_{G_r}(\delta_\ell)$ and $\pi_{G_r}(\delta_J)$ are paths  in $\overline{\HH}/G_r$ which do not pass through the cone point in $\HH/G_r$.  Moreover, as $\alpha$ is simple,  $\pi_{G_r}(\delta_\ell)$ and $\pi_{G_r}(\delta_J)$ are simple. Now we consider the case where $\ell$ has no ramification point.  Then the stabilizer $\Stab_{G(S)}(\ell)$ of $\ell$ is generated by the hyperbolic element $g$ so  $\pi_S(\ell)$ is a regular simple closed geodesic.
On the other hand,  the choice of $r$ is arbitrary and so for any ramification point $r$, we get that $\pi_{G_r}(\delta_\ell)$ and $\pi_{G_r}(\delta_J)$ are simple. This implies that the regular simple closed geodesic $\pi_S(\ell)$ and $\alpha(I)$ are in same free homotopy class  in $S^o$ (see Figure $3$ in \cite{Suzzi_Valli_2015}). Therefore, we can take a free homotopy between  $\pi_S(\ell)$ and $\alpha(I)$ in $S^o$. Thus, the first statement is proven. 

Then we consider the case where there is a ramification point $r'$ in $\ell.$ Then the ramification index of $r'$ must be  two and the stabilizer $G_{r'}$ preserves $\ell.$ Now we focus on the stabilizer group $\operatorname{Stab}_{G(S)}(\ell)$ of $\ell.$ Since we may think of $\operatorname{Stab}_{G(S)}(\ell)$ as a discrete subgroup of the isometry group of the real line,  $\operatorname{Stab}_{G(S)}(\ell)$ is isomorphic to $\ZZ_2,$ $\ZZ,$ or the infinite dihedral group $\operatorname{Dih_\infty}.$ Obviously, $G_{r'}\subset \operatorname{Stab}_{G(S)}(\ell)$ and $g \in \operatorname{Stab}_{G(S)}(\ell).$ Let $s$ be a geodesic segment joining $r'$ and $g^{-1}(r')$ and $f$ be the non-trivial element of $G_{r'}.$ Then the end points of $g(s)$ are $r'$ and $g(r').$ Therefore, $(f\circ g)(g^{-1}(r'))=f(r')=r',$ $(f\circ g)(r')=g^{-1}(r'),$ and $(f\circ g)(s)=s$ since $f$ is an isometry in $G_{r'}$. Hence, $fg$ fixes the midpoint $r''$ between $r'$ and $g^{-1}(r')$ and preserves $\ell.$ 
 Then we can see that $f$ and $fg$ are elliptic elements of order two and generate $\operatorname{Stab}_{G(S)}(\ell).$ Therefore, $\operatorname{Stab}_{G(S)}(\ell)$  is isomorphic to $\operatorname{Dih_\infty}$, namely  $\operatorname{Stab}_{G(S)}(\ell)=\langle f, fg| f^2=(fg)^2=\id \rangle.$

Now we write $\pi_{\operatorname{Stab}_{G(S)}(\ell)}$ for the projection from $\HH$ to $\HH/\operatorname{Stab}_{G(S)}(\ell).$ Then $$\Sigma(\HH/\operatorname{Stab}_{G(S)}(\ell))=\{\pi_{\operatorname{Stab}_{G(S)}(\ell)}(r'), \pi_{\operatorname{Stab}_{G(S)}(\ell)}(r'')\}$$ and $\pi_{\operatorname{Stab}_{G(S)}(\ell)}(\ell)$ is the geodesic segment joining two cone points. Also, $\pi_{\operatorname{Stab}_{G(S)}(\ell)}(\overline{\alpha})|[0,1]$ is a simple closed curve bounding a disk with the cone points. This implies the second statement. 
\end{proof}

The proof of the following proposition is similar with the case in complete hyperbolic surfaces. 
\begin{prop}
Let $S$ be a complete $2$-dimensional hyperbolic orbifold, and $\alpha$ and $\beta$ be essential regular simple closed curves in $S.$ Fix $G(S).$ If $\alpha$ and $\beta$ are disjoint and are not freely homotopic in $S^o$, then the geodesic realizations are disjoint.  
\end{prop}

\subsection{Elementary $2$-dimensional hyperbolic orbifolds}

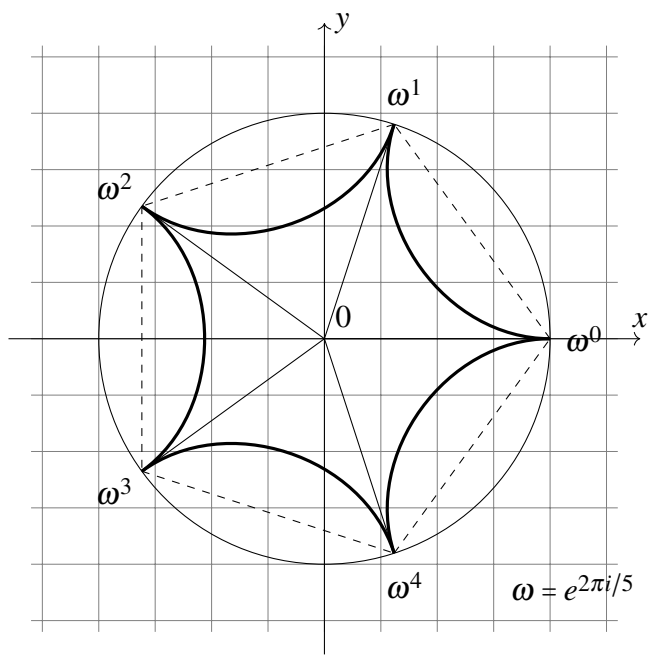
\begin{figure}
\begin{center}
\begin{tikzpicture}[scale=3]
\draw[step=.25, gray, very thin](-1.3,-1.3) grid (1.3,1.3);

\draw[thin,->](-1.4,0)--(1.4,0) node[above]{$x$};
\draw[thin,->](0,-1.4)--(0,0) node[above right]{$0$}--(0,1.4) node[right]{$y$};
\draw (0,0) circle [radius=1];
\node at (1.1, -1.1){$\omega=e^{2\pi i/5}$};
\def\deg{72}      
\foreach \t/\x in {0/0*\deg,1/1*\deg, 2/2*\deg,3/3*\deg,4/4*\deg}
{\draw[thin] (0:0)--(\x:1);
\node[anchor=center] at (\x:1.15) {$\omega^{\t}$};}
\foreach \from/\to in {0/1,1/2,2/3,3/4,4/0}
{\draw [very thick] (\from*\deg:1) to [out=180+\from*\deg, in=180+\to*\deg](\to*\deg:1);
\draw [thin, dashed] (\from*\deg:1)--(\to*\deg:1);}
\end{tikzpicture}
\end{center}
\caption{
Thick arcs are hyperbolic geodesics in the Poincaré disk $\DD$ joining consecutively the points $\{1,\omega, \omega^2, \omega^3, \omega^4\}$ in the infinite circle $\partial \DD$. Hence, the geodesics bound an ideal $5$-gon. Moreover, it is  preserved by an elliptic isometry defined as $z\mapsto \omega z.$ Therefore, the ideal $5$-gon is regular and $0$ is the center.
}
\label{fig:5gon}
\end{figure}

\subsubsection{Ideal monogons}
Let $n$ be a natural number which is greater than $2.$ Recall that an \emph{ideal $n$-gon} is a hyperbolic surface $S$ with geodesic boundary if the interior of $S$ is homeomorphic to the disk, the boundary of $S$ is a disjoint union of $n$ bi-infinite geodesics, and the area of $S$ is finite (see Figure \ref{fig:5gon}). Each boundary component of $S$ is call a \emph{side} and each end point of a side is called a \emph{vertex}. We think of an ideal $n$-gon $S$ as a closed convex subset of $\HH.$ The ideal $n$-gon $S$ is called \emph{regular} if there is an elliptic element $r_S$ in $\PR$ which is a rotation through angle $2\pi/n$ about some point $c_S$ in $S$ and preserves $S$. We call $c_S$ the \emph{center} of $S.$  Then the quotient space $S/\langle r_S \rangle$ of $S$ by the group generated by $r_S$ is called an \emph{ideal monogon} with a cone point of order $n$ (see Figure \ref{fig:monogon}). The quotient space $S/\langle r_S \rangle$ is a $2$-dimensional hyerpbolic orbifold with one geodesic boundary and the cone point is the image of the center $c_S.$  Hence, the order of the cone point is $n$ and there is a unique shortest geodesic segment $\ell_S$ joining the cone point with the geodesic boundary. Note that the segment $\ell_S$ intersects perpendicularly with the geodesic boundary.

\begin{figure}
\centering
\begin{subfigure}[b]{0.3\textwidth}
\includegraphics[width=\textwidth]{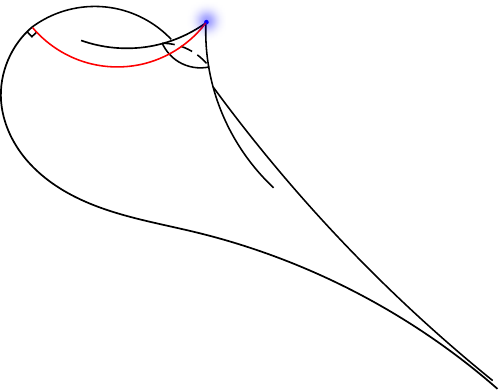}
\caption{An ideal monogon}
  \label{fig:monogon}
\end{subfigure}
\hfill
\begin{subfigure}[b]{0.3\textwidth}
\includegraphics[width=\textwidth]{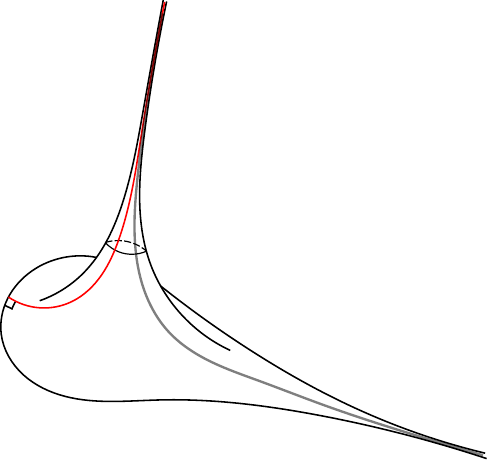}
   \caption{A monogon with one puncture}
   \label{fig:monogoncusp}
\end{subfigure}
\hfill
 \begin{subfigure}[b]{0.3\textwidth}
  \includegraphics[width=\textwidth]{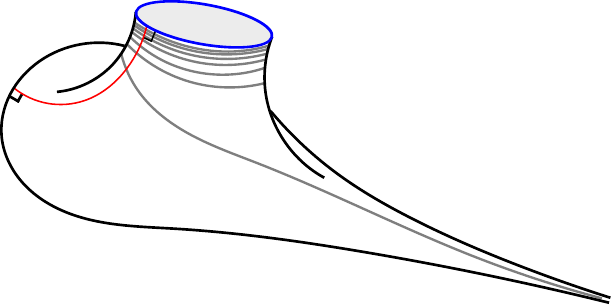}
    \caption{A monogon with one hole}
    \label{fig:monogonhole}
 \end{subfigure}
\hfill
  \caption{In Figure \ref{fig:monogon}, the blue point is the cone point. In Figure \ref{fig:monogoncusp} and Figure \ref{fig:monogonhole}, the gray lines are the images of $s_1$ and $s_2$ under the identification by $f_{12}$. Hence, the gray lines are geodesics in each $2$-dimensional hyperbolic orbifold with geodesic boundary. In particular, the blue circle in Figure \ref{fig:monogonhole} is the simple closed geodesic added to be complete. And the red arcs are the geodesics $\ell_S$.}
\label{fig:annuli}
\end{figure}

\subsubsection{Geometric annuli}
Let $P_3$ be an ideal triangle in $\HH$. Note that every ideal triangle is regular. We say that $s_1$, $s_2$, and $s_3$ are the sides of $P_3.$ For each $i\in \{1,2,3\},$ there is the unique perpendicular from the center $c_{P_3}$ to the geodesic $s_i$ and so we say that $p_i$ is the foot of the perpendicular. Now we glue $s_1$ and $s_2$ by an element $f_{12}$ in $\PR$ such that $f_{12}(s_1)=s_2.$ Then $f_{12}$ fixes the end point $q$ in $\partial_\infty \HH$ at which $s_1$ and $s_2$ intersect. 

If $f_{12}$ is a parabolic element and $f_{12}(p_1)=p_2$, the resulting surface $S$ is Cauchy complete under the metric induced by the hyperbolic metric of $P_3,$ and has one bi-infinite geodesic boundary. Hence, we call this resulting orbifold a \emph{monogon with one puncture} (see Figure \ref{fig:monogoncusp}).  In this case, there is a unique geodesic $\ell_S$ joining the puncture with $s_3$ which intersect perpendicularly with $s_3.$ Moreover, $\ell_S \cap s_3=\{p_3\}.$

If $f_{12}$ is hyperbolic and so $f_{12}(p_1)\neq p_2$, then the resulting space is incomplete. By adding a simple closed geodesic, we can make the resulting space be complete, and in this case, the resulting space $S$ is a hyperbolic surface which has one bi-infinite geodesic boundary and one closed geodesic boundary. Hence, we call this space a \emph{monogon with one hole}(see Figure \ref{fig:monogonhole}) The length of the closed geodesic boundary is determined by the translation length of  $f_{12}.$ Hence, for any given positive number $l$,  there is a monogon with one hole whose closed geodesic boundary is of length $l.$ Note that  there is a unique shortest geodesic segment $\ell_S$ joining two geodesic boundaries which intersect perpendicularly with both geodesic boundaries. Moreover, it turns out that $\ell_S\cap s_3=\{p_3\}.$  

From the construction of a monogon $S$ with one hole or with one puncture, there is a bi-infinite geodesic decomposing $S$ into an ideal triangle which is the gray geodesics in Figure \ref{fig:annuli}. Now we call a $2$-dimensional hyperbolic orbifold $S$ with geodesic boundary a \emph{geometric annulus} if $S$ is isometric to an ideal monogon with one cone point,  a monogon with one puncture, or a monogon with one hole (see Figure \ref{fig:annuli}). Note that a geometric interior of a geometric annulus is homeomorphic to the annulus.  Also, we call the intersection point of $\ell_S$ and the bi-infinite geodesic boudary the \emph{shearing point} of $S.$

\subsubsection{Geometric pair of pants} 

\begin{figure}
\centering
\begin{subfigure}[b]{0.2\textwidth}
  \includegraphics[width=\textwidth]{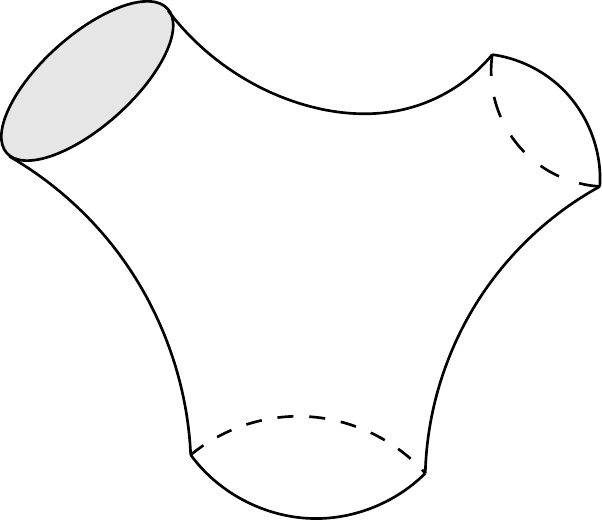}
  \caption{Three geodesic boundaries}
  \label{fig:pants1}
\end{subfigure}
\hfill
\begin{subfigure}[b]{0.20\textwidth}
  \includegraphics[width=\textwidth]{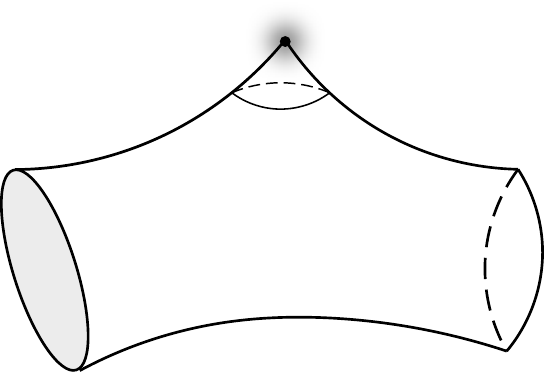}
  \caption{Two geodesic boundaries}
  \label{fig:pants2}
\end{subfigure}
\hfill\begin{subfigure}[b]{0.2\textwidth}
  \includegraphics[width=\textwidth]{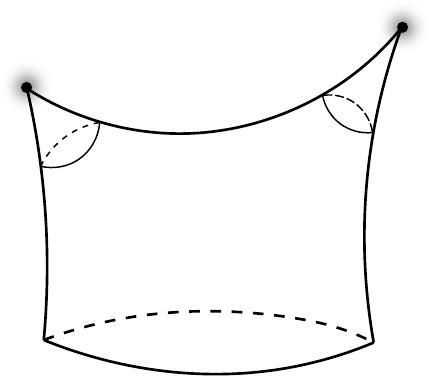}
  \caption{One geodesic boundary}
  \label{fig:pants3}
\end{subfigure}
\hfill\begin{subfigure}[b]{0.2\textwidth}
  \includegraphics[width=\textwidth]{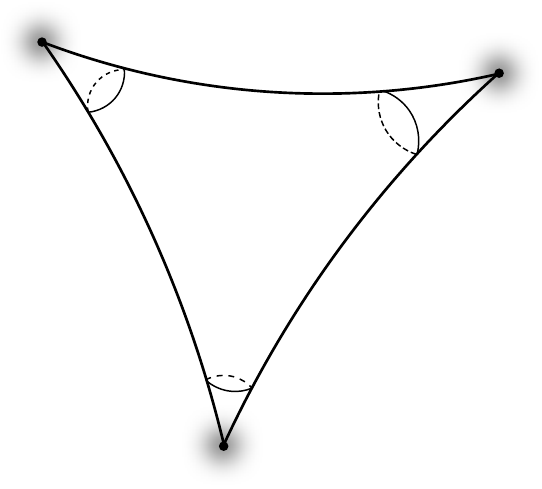}
  \caption{No geodesic boundary}
  \label{fig:pants4}
\end{subfigure}
\caption{There are four types of geometric pairs of pants according to the number of closed geodesic boundary components. Each marked point is a cone point or a puncture. However, in Figure \ref{fig:pants3}, two marked points can not be cone points of order two simultaneously.}
\label{fig:pants}
\end{figure}

\begin{defn}
A \emph{geometric pair of pants} is a $2$-dimensional hyperbolic orbifold with geodesic boundary satisfying the following:
\begin{enumerate}
\item The geometric interior is homeomorphic to the thrice-punctured sphere; 
\item The geometric boundary is compact and has exactly three connected components; 
\item The hyperbolic area is finite.
\end{enumerate}
\end{defn}
\begin{figure}
\centering
\begin{subfigure}[b]{0.2\textwidth}
  \includegraphics[width=1.2\textwidth]{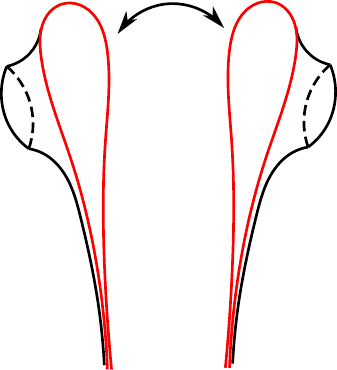}
  \caption{Two monogons with one hole}
  \label{fig:twomonogonhole}
\end{subfigure}
\hfill
\begin{subfigure}[b]{0.2\textwidth}
  \includegraphics[width=1.1\textwidth]{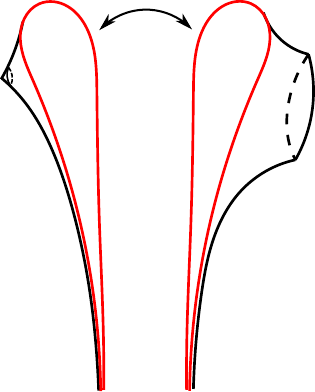}
  \caption{One monogon with one hole}
  \label{fig:onemonogonhole}
\end{subfigure}
\hfill
\begin{subfigure}[b]{0.2\textwidth}
  \includegraphics[width=\textwidth]{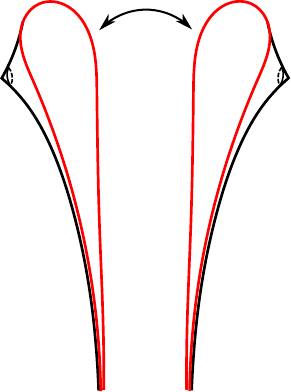}
  \caption{No monogon with one hole}
  \label{fig:nomonogonhole}
\end{subfigure}

\caption{Each pointy point represents a puncture or a cone of order $n>2.$ Then there are three cases according to the number of used monogons with one hole in the glueing process. The red lines are the bi-infinite geodesics $b_1$ and $b_2$ along which we glue two monogons.}
\label{fig:gluing}
\end{figure}

\begin{figure}
\centering
\begin{subfigure}[b]{0.2\textwidth}
  \includegraphics[width=\textwidth]{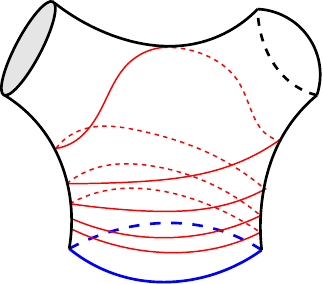}
  \caption{Three geodesic boundaries}
  \label{fig:threegeodesiccase}
\end{subfigure}
\hfill
\begin{subfigure}[b]{0.2\textwidth}
  \includegraphics[width=\textwidth]{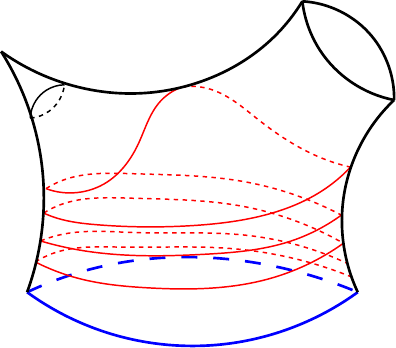}
  \caption{Two geodesic boundaries}
  \label{fig:twogeodesiccase}
\end{subfigure}
\hfill
\begin{subfigure}[b]{0.2\textwidth}
  \includegraphics[width=0.9\textwidth]{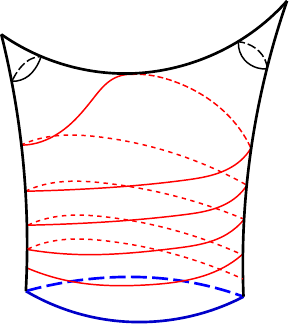}
  \caption{One geodesic boundaries}
  \label{fig:onegeodesiccase}
\end{subfigure}

\caption{Each pointy point represents a puncture or a cone of order $n>2.$ These pairs of pants are corresponding to the first three types in Figure \ref{fig:pants}. In Figure \ref{fig:gluing}, after gluing along red lines, the resulting metric spaces are not usually complete. Therefore, we need to add a simple closed geodesic to be complete. The blue curves represent the additional simple closed geodesics. Then, after gluing and completion, we get the geometric pairs of pants. Precisely, Figure \ref{fig:threegeodesiccase}, Figure \ref{fig:twogeodesiccase} or Figure \ref{fig:onegeodesiccase}  are obtained from Figure \ref{fig:twomonogonhole}, Figure \ref{fig:onemonogonhole} or Figure \ref{fig:nomonogonhole}, respectively. }

\label{fig:completionofnewpairs}
\end{figure}
All geometric pairs of pants fall into four types according to the number of closed geodesic boundary components (See \ref{fig:pants}). For brevity, until the end of this section,  we think of a puncture as a cone point of order $\infty.$ Then in Figure \ref{fig:pants},  each marked point is a cone point of order $n$ for some $n \in \NN \cup \{\infty \}$ with $n\neq 1.$ More precisely,  in Figure \ref{fig:pants3}, two marked points can not be cone points of order two simultaneously. In Figure \ref{fig:pants4}, if three marked points are of order $n_1$, $n_2$ and $n_3$, respectively, then $n_1$, $n_2$ and $n_3$ satisfy
 $$2-\left\{(1-\frac{1}{n_1})+(1-\frac{1}{n_2})+(1-\frac{1}{n_3})\right \} <0 $$ 
 ,or 
 $$\frac{1}{n_1}+\frac{1}{n_2}+\frac{1}{n_3}<1$$
 where $1/\infty=0.$
 In other words, the Euler characteristic must be negative.

Let $A_1$ and $A_2$ be geometric annuli and say that the bi-infinite geodesic boundaries of $A_1$ and $A_2$ are $b_1$ and $b_2,$ respectively.  Each boundary $b_i$ has the shearing point $sh_i.$ Now we glue $A_1$ and $A_2$ along $b_1$ and $b_2$ by an isometry $f_{12}$ of real line such that $f_{12}(b_1)=b_2$ (See Figure \ref{fig:gluing}). Then we get a space $S$ which may be incomplete. If $S$ is complete, then $S$ is a geometric pair of pants. If $S$ is not complete, by adding a simple closed geodesic as a boundary, we can make $S$ be complete. Then $S$ is also a geometric pair of pants (see Figure \ref{fig:completionofnewpairs} ).
The length of the additional closed geodesic is determined by the distance between $f_{12}(sh_1)$ and $sh_2.$ Note that for any given positive real number $l,$ we can find a map $f_{12}$ which makes the additional closed geodesic be of length $l.$ Using this construction, we may get any given geometric pair of pants without cone point of order two that is one of the first three types in Figure \ref{fig:pants}. Also, observe that in the geometric pairs of pants, the red geodesics in Figure \ref{fig:completionofnewpairs} and the gray geodesics in Figure \ref{fig:monogon} with the boundary geodesics decompose the geodesic pairs of pants into ideal monogons and ideal triangles. 
\begin{figure}
\centering
\begin{subfigure}[b]{0.2\textwidth}
  \includegraphics[width=0.5\textwidth]{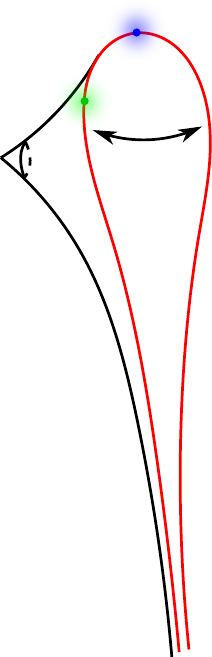}
  \caption{Not a monogon with one hole case}
  \label{fig:onemonogon}
\end{subfigure}
\hspace{10em}
\begin{subfigure}[b]{0.2\textwidth}
  \includegraphics[width=0.8\textwidth]{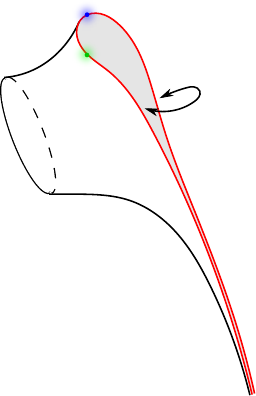}
  \caption{A monogon with one hole case}
  \label{fig:onemonogonwithhole}
\end{subfigure}
\caption{ The pointy point in Figure \ref{fig:onemonogon} represents a puncture or a cone point of order $n>2.$ The red lines are the bi-infinite geodesics $b_1$. The blue marked points are the fixed points $c_1$ of the isometries $f_{11}.$ The green marked points are the shearing points $sh_1$ of the geometric annuli. }
\label{fig:gluingone}
\end{figure}

\begin{figure}
\centering
\begin{subfigure}[b]{0.2\textwidth}
  \includegraphics[width=\textwidth]{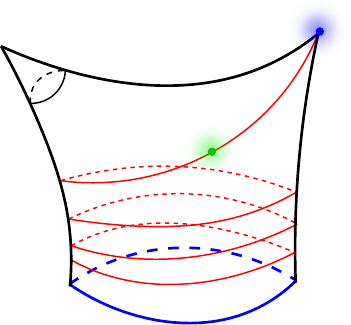}
  \caption{One geodesic boundary}
  \label{fig:twomonogons}
\end{subfigure}
\hspace{10em}
\begin{subfigure}[b]{0.2\textwidth}
  \includegraphics[width=\textwidth]{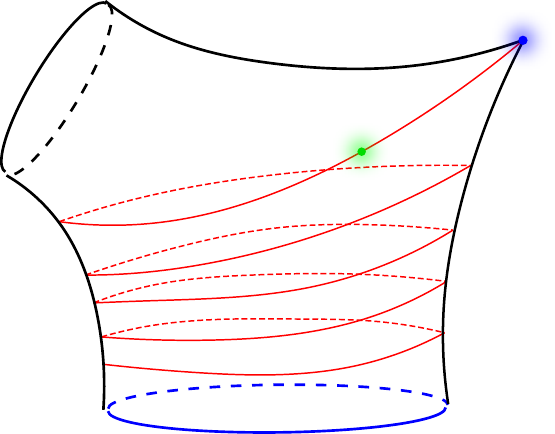}
  \caption{Two geodesic boundaries}
  \label{fig:twomonogons}
\end{subfigure}
\caption{The blue marked points are the images of $b_1$ and the green ones are the images of $sh_1$. Moreover, the blue points are cone points of order two. The blue curves are the additional simple closed geodesics.}
\end{figure}

Similarly, we consider the case where a geometric pair of pants has a cone point of order $2.$
Let $f_{11}$ be an isometry of $b_1$ such that $f_{11}$ is a reflection on $b_1$ and so there is a fixed point $c_1$ (see Figure \ref{fig:gluingone}). Now a quotient space $S$ is obtained from $A_1$ by identifying $x\in b_1$ with $f_{11}(x)$. Then if $c_1=sh_1$, then $S$ is complete and is a geometric pair of pants  with one cone point of order two. If $c_1\neq sh_1,$ then $S$ is not complete and so by adding  a closed geodesic as a boundary, we make $S$ be complete (see \ref{fig:completionofnewpairs}). Hence, $S$ is a geometric pair of pants with one cone point of order two. The length of the additional closed geodesic boundary component is determined by the distance between $c_1$ and $sh_1.$ Like a previous case, using this construction, we can get any given geometric pair of pants with a cone point of order two that is second or third type in Figure \ref{fig:pants}. Also, observe that  in the geometric pairs of pants, the red geodesics in Figure \ref{fig:completionofnewpairs} and the gray geodesics in Figure \ref{fig:monogon} with the boundary geodesics decompose the geodesic pair of pants into ideal monogons and ideal triangles.

From observations, we may conclude that any geometric pair of pants, which does not have three cone points, can be decomposed into two or one geometric annuli. Moreover, the geometric pairs of pant are decomposed into ideal monogons and ideal triangles.  Now we summarize the above in the following lemma.
\begin{lem} \label{lem : a lamination in a geometric pair of pants}
Let $S$ be a geometric pair of pants. Suppose that $S$ does not have three cone points of finite order.(i.e. $S$ is not the type of Figure \ref{fig:pants4})  There is a geodesic lamination such that the metric completion of each connected component of the complement of the lamination is isometric to an ideal triangle or an ideal monogon with one cone point. 
\end{lem}

\section{Geodesic laminations in complete $2$-dimensional hyperbolic orbifolds}\label{sec:main2}

In this section, we prove the structure theorem of complete $2$-dimensional hyperbolic orbifolds following the paper \cite{BASMAJIAN_2017}. Then we show that any Fuchsian group $G$ of the first kind such that $\HH/G$ is not a geometric pair of pants is pants-like $\COL_3.$

\subsection{A structure theorem for complete $2$-dimensional hyperbolic orbifolds}\label{subsec:the structure theorem}

\subsubsection{Half-planes and funnels in a complete $2$-dimensional hyperbolic orbifold}
Let $S$ be a hyperbolic surface with geodesic boundary. We define $H=\{z\in \HH : 0\leq \operatorname{Re}(z)\}.$ Then if $S$ is isometric to $H$ with the metric induced from $\rho,$ then $S$ is called a \emph{half-plane}. If $S$ is isometric to $H/G$ where $G$ is generated by $z\mapsto e^{h}z $ for some $h\in \RR_{>0},$ then $S$ is called a \emph{funnel} with boundary length $h.$ 

Let $G$ be a Fuchsian group and $S$ be the complete $2$-dimensional hyperbolic orbifold $\HH/G.$ For a closed subset $A$ of $\partial \HH$, there is a unique smallest convex subset $\overline{CH(A)}$ in $\overline{\HH}$ such that $\overline{CH(A)}\cap \partial \HH =A$. Then $CH(A)=\overline{CH(A)}\cap \HH$ is also a convex subset. Hence, we call $CH(A)$ the \emph{convex hull} of $A.$ When $G$ is not elementary, the \emph{convex core} $CC(S)$ of $S$ is the quotient of the convex hull of $L(G)$ by $G$, $CH(L(G))/G.$ In general, the convex core of a complete $2$-dimensional hyperbolic orbifold  is a $2$-dimensional hyperbolic orbifold with geodesic boundary. In particular, when $G$ is a Fuchsian group of the first kind, $S$ itself is the convex core of $S.$ 

Now we discuss the case where $G$ is non-elementary and  of the second kind. Since $L(G)$ is a cantor set in $\partial \HH,$ the complement $\Omega(G) \cap \partial \HH$ is a countable disjoint union of nondegenerate open intervals. Let $I$ be a connected component of $\Omega(G) \cap \partial \HH.$ Then the stabilizer $\stab_G(I)$ of $I$ under the $G$-action is a trivial group or an infinite cyclic group. Now we consider $CH(\overline{I})$ where $\overline{I}$ is the closure of $I$ in $\partial{\HH}.$ Then $CH(\overline{I})$ is a half-plane under the Poincaré metric and the stabilizer $\stab_G(CH(\overline{I}))$ of $CH(\overline{I})$  is exactly $\stab_G(I).$ Then we can see that $CH(\overline{I})/\stab_G(CH(\overline{I}))$ is isometrically embedded in $S.$ Moreover, 
$CH(\overline{I})/\stab_G(CH(\overline{I}))$ is a half-plane if $\stab_G(CH(\overline{I}))$ is trivial, and a funnel if $\stab_G(CH(\overline{I}))$ is an infinite cyclic group. Note that when $\stab_G(CH(\overline{I}))$ is an infinite cyclic group, $\stab_G(CH(\overline{I}))$ is generated by a hyperbolic element $g\in \PR$ and the hyperbolic axis of $g$ is the geodesic boundary of $CH(\overline{I}).$ Therefore, when $\stab_G(CH(\overline{I}))$ is an infinite cyclic group, the boundary length of the funnel $CH(\overline{I})/\stab_G(CH(\overline{I}))$ is equal to  the translation length of $g.$ Also note that $\pi_1(CC(S))=\pi_1(S)$ where $S=\HH/G.$

\subsubsection{Pants decompositions}
Let $S$ be a $2$-dimensional hyperbolic orbifold with geodesic boundary. Then a collection $\mathcal{C}$ of non-trivial simple closed curves in $S^o$ which are homotopically distinct and disjoint is called  a \emph{pants decomposition} of $S$ if each connected component of $S^o-\bigcup \mathcal{C}$ is homeomorphic to the sphere without three points and $\mathcal{C}$ is locally finite. Let $\mathcal{C}_1$ and $\mathcal{C}_2$ be two pants decompositions of $S.$ We say that $\mathcal{C}_1$ and $\mathcal{C}_2$ are \emph{distinct} if for each $(c_1,c_2)\in \mathcal{C}_1 \times \mathcal{C}_2,$ $c_1$ and $c_2$ are not freely homotopic in $S^o.$   

A geodesic lamination $\Lambda$ of $S$ is called a \emph{geometric pants decomposition} if the Cauchy completion of each connected component of $S-\Lambda$ is a geometric pair of pants. Let $\Lambda_1$ and $\Lambda_2$ be two geodesic laminations in $S$. We say  that $\Lambda_1$ and $\Lambda_2$ are \emph{transverse} if there is no leaf shared by two geodesic laminations $\Lambda_1$ and $\Lambda_2.$ 

Let $G$ be  a Fuchsian group that is not a finite group and $S$ be the complete $2$-dimensional hyperbolic orbifold $\HH/G$. 
If $G$ is torsion free, then $S$ is a complete hyperbolic surface. If not, there is a ramification point $p$ in $\HH$ and  any point in the orbit $G\cdot p$ is also a ramification point. Then, since $G$ is not finite and the set of ramification points is discrete in $\HH,$ there are countably many ramification points. Hence, the complement $\HH-\pi_G^{-1}(\Sigma(S))$ of ramification points in $\HH$ is a Riemann surface whose fundamental group is a free group of infinite rank. Note that $(\pi_G)_*(\pi_1(\HH-\pi_G^{-1}(\Sigma(S)))$ is a normal subgroup of $\pi_1(S^o).$ Then the universal covering of $\HH-\pi_G^{-1}(\Sigma(S))$ is $\HH$ with a holomorphic covering map $\tilde{\pi}_G$ such that the deck transformation group of $\tilde{\pi}_G$ is isomorphic to $(\pi_G)_*(\pi_1(\HH-\pi_G^{-1}(\Sigma(S))).$ We write $K_G$ for the deck transformation group of $\tilde{\pi}_G.$  Then we define $$H_G=\{ h \in \operatorname{Aut}(\HH) : \tilde{\pi}_G\circ h= g\circ \tilde{\pi}_G \ \text{for some} \ g\in G \}.$$ Then $H_G$ is isomorphic to $\pi_1(S^o)$ and is a Fuchsian group. 
Hence, $ \pi_G \circ \tilde{\pi}_G$ is a universal covering map from $\HH$ to $S^o$ whose deck transformation group is $H_G.$ 

Now we define a homomorphism $\tilde{\pi}_G^*$ from $H_G$ to $G$ so that for each $h\in H_G,$ $\tilde{\pi}_G \circ h=\tilde{\pi}_G^*(h)\circ \tilde{\pi}_G.$  Then we have an exact sequence

 \begin{center}
\begin{tikzcd}
\{1\} \arrow[r] & K_G \arrow[r, "i"] & H_G\arrow[r, "\tilde{\pi}_G^*"]& G \arrow[r] &\{1\}
 \end{tikzcd}
\end{center}
where $i$ is the inclusion map. Hence, $H_G/K_G$ is isomorphic to $G.$ 
For more detail, consult the Kra's book \cite{kra1972automorphic}.

Now we prove the following lemma which is analogous to Proposition 3.1 in  \cite{BASMAJIAN_2017}.

\begin{thm}[The structure theorem for complete hyperbolic 2-dimensional orbifolds]\label{thm : straightening a pants decomposition}
Let $G$ be a non-elementary Fuchsian group and $S$ be the complete $2$-dimensional hyperbolic orbifold $\HH/G.$ Suppose that there is a pants decomposition $\mathcal{C}$ of $S.$ Let $\Lambda$ be the union of all geodesic realizations in $S$ of elements of $\mathcal{C}.$ Then $\Lambda \cup (\partial CH(L(G))/G)$ is a geometric pants decomposition of $CC(S).$
\end{thm}
Note that by definition, each element of $\mathcal{C}$ is essential in $S$.
\begin{proof}
Note that any pants decomposition induces an exhaustion by finite type subsurfaces. Here being finite type means that the fundamental group is finitely generated. Hence there is a sequence $\{S_n\}_{n=1}^\infty$ of subsurfaces of $S^o$ such that for each $n\in \NN,$ $S_n \subset S_{n+1}$ and the manifold boundary $\partial S_n$ consists of curves of $\mathcal{C}$ and $\displaystyle \bigcup_{n=1}^\infty int(S_n) = S^o.$ We may assume that for each $n\in \NN,$ every boundary curve of $S_n$ does not bound a disk with two cone points of order two in $S.$ Therefore, there is a sequence $\{H_n\}_{n=1}^\infty$ of subgroups of $H_G$ such that for all $n\in \NN,$ $H_n < H_{n+1}$ and $H_n$ is a corresponding subgroup of $H_G$ to $\pi_1(S_n).$ Note that $H_G=\displaystyle \bigcup_{n=1}^\infty H_n$ as $S^o=\displaystyle \bigcup_{n=1}^\infty S_n.$ Moreover, for each $n\in \NN,$ there is a unique connected component $\tilde{S}_n$ of $(\pi_G \circ \tilde{\pi}_G)^{-1}(S_n)$ such that the stabilizer $\stab_G(\tilde{S}_n)$ is $H_n$ and so $\tilde{S}_n \subseteq \tilde{S}_{n+1}.$ Note that for each $n\in \NN,$ the manifold boundary $\partial \tilde{S}_n $ consists of full lifts of the boundary curves of $S_n.$

 Now we consider the sequence $\{G_n\}_{n=1}^\infty$ of subgroups of $G$ where $G_n=\tilde{\pi}_G^*(H_n).$ Obviously, $G=\displaystyle \bigcup_{n=1}^\infty G_n.$ Also, for each $n\in \NN,$ the stabilizer of $\tilde{\pi}_G(\tilde{S}_n)$ is $G_n$ and the manifold boundary of  $\tilde{\pi}_G(\tilde{S}_n)$ consists of full lifts of the boundary curves of $S_n.$ Fix $n\in \NN.$ By assumption, each boundary component of $\tilde{\pi}_G(\tilde{S}_n)$ is a quasi-geodesic under the Poincaré metric in $\HH$ and $\tilde{\pi}_G(\tilde{S}_n)$ is a connected component of $\pi_G^{-1}(S_n).$ Let $B_n$ be the set of all hyperbolic axises of holonomies of boundary components of $\tilde{\pi}_G(\tilde{S}_n).$ Then there is a closed convex subset $C_n$ of $\HH$ such that $\partial C_n= B_n.$ Then by the choice of $C_n,$ $G_n$ preserves $C_n$ and the stabilizer of $C_n$ is $G_n$ under the $G$-action. Therefore, $\pi_G(C_n)$ is isometric to the $2$-dimensional hyperbolic orbifold $C_n/G_n$ with geodesic boundary. Moreover, $\pi_G(C_n)$ is the $2$-dimensional hyperbolic orbifold which is bounded by the geodesic realizations of the boundary curves of $S_n.$

Since $C_n$ is a closed convex subset of $\HH$ preserved by $G_n,$ $CH(L(G_n)) \subseteq C_n.$ Since each boundary component of $C_n$ is a hyperbolic axis for some hyperbolic element of $G_n,$ the end points of each boundary component of $C_n$ are contained in $L(G_n).$ Hence, $\partial C_n \subset \partial CH(L(G_n)).$ Since $G_n$ is finitely generated, for each boundary component $\ell$ in  $\partial CH(L(G_n))-\partial C_n,$ $\pi_G(\ell)$ bounds a funnel in $\pi_G(C_n).$ (See Section 10.4. in \cite{beardon1983geometry})

Now we claim that $CH(L(G))=\overline{\displaystyle \bigcup_{n=1}^\infty CH(L(G_n))}.$ As for each $n\in \NN,$ $L(G_n)\subseteq L(G_{n+1}),$ for each $n\in \NN,$ $CH(L(G_n)) \subseteq CH(L(G_{n+1}))$ and so  $\displaystyle \bigcup_{n=1}^\infty CH(L(G_n))$ is convex. Then $G$ preserves the closed convex subset $\overline{\displaystyle \bigcup_{n=1}^\infty CH(L(G_n))}$ of $\HH.$ Hence, $CH(L(G)) \subseteq \overline{\displaystyle \bigcup_{n=1}^\infty CH(L(G_n))}.$ On the other hand, for any $n\in \NN,$ $L(G_n) \subseteq L(G)$ and $CH(L(G_n)) \subseteq CH(L(G)).$ Therefore, $$\displaystyle \bigcup_{n=1}^\infty CH(L(G_n))\subseteq CH(L(G)) \subseteq \overline{\displaystyle \bigcup_{n=1}^\infty CH(L(G_n))}.$$ The claim is proved since $CH(L(G))$ is closed in $\HH$.

Let $\ell$ be a bi-infinite geodesic in $\partial CH(L(G)).$ If  $\ell$ is a manifold boundary component of $\displaystyle \bigcup_{n=1}^\infty CH(L(G_n))$, there is a natural number $n_0$ such that $\ell \subset \partial CH(L(G_{n_0})).$ Then $\pi_G(\ell)$ bounds a funnel in $S.$ If $\ell$ is in  $\displaystyle CH(L(G)) - \bigcup_{n=1}^\infty CH(L(G_n)),$ then there is a sequence $\{\ell_n\}_{n=1}^\infty$ of geodesics such that for each $n\in \NN,$ $\ell_n \subset \partial CH(L(G_n))$ and $\{\ell_n\}_{n=1}^\infty$ converges to $\ell.$ Then $\{\pi_G(\ell_n)\}_{n=1}^\infty$ converges to $\pi_G(\ell).$ If $\pi_G(\ell)$ is a regular simple closed geodesic, then $\pi_G(\ell)\subset int(S_{n_1})$ for some $n_1\in \NN$ since $\displaystyle \bigcup_{n=1}^\infty int(S_n) = S^o.$ Hence, $\pi_G(\ell) \subset \pi_G(C_{n_1})$ and $\pi_G(\ell)\subset \pi_G(CH(L(G_{n_1})))$, so $\ell \subset CH(L(G_{n_1})).$ This is a contradiction. Therefore, $\pi_G(\ell)$ is a bi-infinite geodesic and bounds a half-plane in $S.$ 

Now we have that $\{\pi_G(CH(L(G_n)))\}_{n=1}^\infty$ is an increasing sequence of $2$-dimensional hyperbolic orbifolds with geodesic boundary whose orbifold fundamental group are finitely generated and such that $\displaystyle \bigcup_{n=1}^\infty \pi_G(CH(L(G_n)))$ is $CC(S)$ without bi-infinite geodesic boundaries. Finally, by the definition of a pants decomposition, each element in $\mathcal{C}$ has a geodesic realization and the geodesic realization is contained in $\pi_G(CH(L(G_{n_2})))$ for some $n_2 \in \NN.$ Therefore, for $n\in \NN,$ $\Lambda \cap \pi_G(CH(L(G_{n})))$ is a geometric pants decomposition of $ \pi_G(CH(L(G_{n}))).$ Thus, we are done. 
\end{proof}

\begin{rmk}
If $G$ is of the first kind, then $\Lambda$ itself is a geometric pants decomposition of $S.$ 
\end{rmk}

\subsection{Fuchsian groups of the first kind are pants-like $\COL_3$}\label{subsec:main2b}
\begin{lem}\label{lem : three geometric pants decompositions}
Let $G$ be a Fuchsian group of the first kind and $S$ be the complete $2$-dimensional hyperbolic orbifold $\HH/G.$ Suppose that $S$ is not a geometric pants. Then there are three pairwise transverse geometric pants decompositions. 
\end{lem}
\begin{proof}
If $S^o$ is a Riemann surface of finite type, then there is a nonempty pants decomposition since $S$ is not a geometric pants. 
Otherwise,  by Richards classification result \cite{Ian}, $S^o$ has a topological exhaustion by finite type surfaces. Hence there is a pants decomposition $\mathcal{C}_1.$ Note that  Lemma 4.2 and Proposition 4.3 in \cite{BaikFuchsian} does not depend on the metric of surfaces. Hence by applying the argument of Lemma 4.2 and Proposition 4.3 in \cite{BaikFuchsian} to $S^o$ and $\mathcal{C}_1,$ we can get pants decompositions $\mathcal{C}_2$ and $\mathcal{C}_3$ such that $\mathcal{C}_1,$ $\mathcal{C}_2,$ and $\mathcal{C}_3$ are  pairwise distinct. By Theorem \ref{thm : straightening a pants decomposition}, there are geometric pants decompositions $\Lambda_1,$ $\Lambda_2$ and $\Lambda_3$ which are pairwise transverse.  
\end{proof}

\begin{thm}\label{thm: of the first kind}
Let $G$ be a Fuchsian group of the first kind. Suppose that $\HH/G$ is not a geometric pair of pants. Then $G$ is a pants-like $\COL_3$ group.  
\end{thm}
\begin{proof}
By Lemma \ref{lem : three geometric pants decompositions}, there are three geometric pants decompositions $\Lambda_1,$ $\Lambda_2,$ and $\Lambda_3$ of $S$ which are pairwise transverse. By Lemma  \ref{lem : a lamination in a geometric pair of pants}, we can add geodesics in each geometric pair of pants and so we can get three geodesic laminations $\Lambda_1',$ $\Lambda_2',$ and $\Lambda_3'$ such that the completion of any connected component of the complement of each geodesic lamination is an ideal triangle or a monogon with one cone point. Therefore, $\pi_G^{-1}(\Lambda_1'),$$\pi_G^{-1}(\Lambda_2'),$ and $\pi_G^{-1}(\Lambda_3')$ are geodesic laminations whose complements are ideal polygons. Therefore, they induce three laminations of $S^1$ and so they also induce three very full laminations systems. Moreover, we can see that the very full lamination systems comprise a pant-like collection for $G.$  Thus, $G$ is a pants-like $\COL_3$ group. 
\end{proof}

\subsection{The proof :  Fuchsian groups of the second kind are pants-like $\COL_3$.}
\label{subsec:main2c}
Now, we consider the case of the second kind and finish the proof of Theorem \ref{B}. 
Let $G$ be a non-elementary Fuchsian group of the second kind. The idea of the proof of the case of the second kind is first to take the doubling the $\HH/G.$ Then, we can see that the doubling is a hyperbolic surface whose fundamental group is a Fuchsian group of the first kind. Now, we can use Theorem \ref{thm : straightening a pants decomposition}, so we get three pairwise transverse geometric pants decompositions. Finally, restricting the laminations to the original surface gives the original surface  three geodesic laminations which are not very full in general. Hence, we need to add more geodesics to make the laminations be very full. 

First, we construct geodesic laminations for elementary pieces. 
\begin{lem}\label{lem: the half farey triangulation}
Let $H$ be the half plane and $\overline{H}$ be the closure of $H$ in $\overline{\HH}=\HH\cup \hat{\RR}.$ Suppose that $Q$ is a countable dense subset of the set $\overline{H}-H$ of non-negative real numbers, containing $\{0, \infty \}.$ There is a geodesic lamination $\Lambda$ in $H$ such that the Cauchy completion of  each connected component  of $H-\Lambda$ is the ideal 3-gon and such that in the corresponding lamination system $\LL$,  every leaf is isolated and $E(\LL)=Q$. 
\end{lem}
\begin{proof}
First, we say that $Q$ is  a sequence $\{q_n\}_{n=1}^\infty$ such that $q_1=0$ and $q_2=\infty.$
Then there is a geodesic joining $q_1$ and $q_2.$ See Figure \ref{fig:the disk with a red line}. The red line represents the geodesic and the geodesic is the boundary of the half plane $H.$ Note that $H=CH([0,\infty]).$
  
Now, we consider $q_3.$ We can draw two geodesics. One is the geodesic joining two points $q_1$ and $q_3$ and the other is the geodesic joining $q_2$ and $q_3.$ The blue lines in Figure \ref{fig:the disk with two blue lines} represent the geodesics. Then, $q_4$ is contained in $(q_1, q_3)$ or $(q_3, q_2)$ and both $CH([q_1, q_3])$ and $CH([q_3, q_2])$ are half planes. So we can do similarly in $CH([q_1, q_3])$ or $CH([q_3, q_2]).$ Obviously, we can repeat this process. Finally, we can get Figure \ref{fig:a half Farey triangulation}. These geodesics give a geodesic lamination in $H$ that we want. 
\end{proof}

\begin{figure}
\centering
\begin{subfigure}[b]{0.2\textwidth}
  \includegraphics[width=\textwidth]{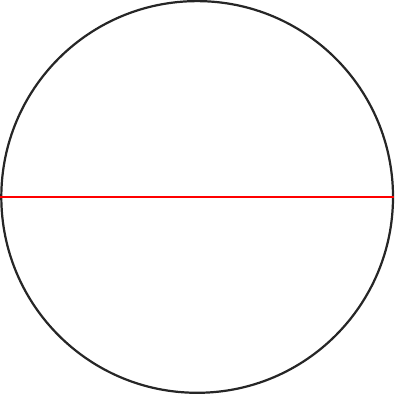}
  \caption{The unit disk with the geodesic connecting two points, $0$ and $\infty.$}
 \label{fig:the disk with a red line}
\end{subfigure}
\hspace{4em}
\begin{subfigure}[b]{0.2\textwidth}
  \includegraphics[width=\textwidth]{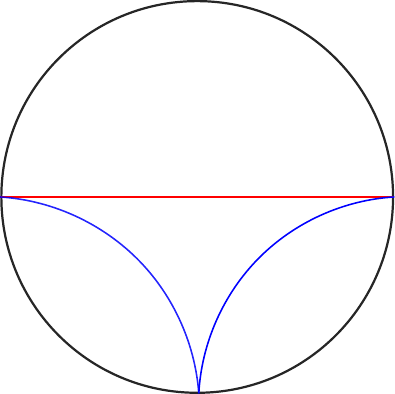}
  \caption{The first step}
 \label{fig:the disk with two blue lines}
\end{subfigure}
\hspace{4em}
\begin{subfigure}[b]{0.2\textwidth}
  \includegraphics[width=\textwidth]{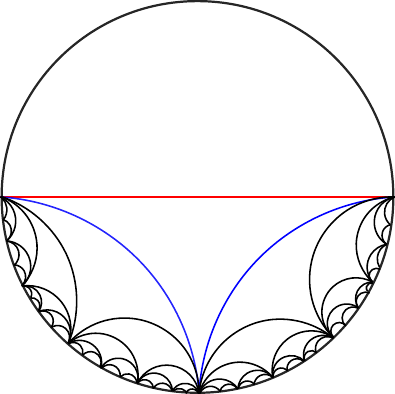}
  \caption{A half Farey triangulation}
  \label{fig:a half Farey triangulation}
\end{subfigure}
\caption{These pictures show the steps to draw a half Farey triangulation.}
\end{figure}

\begin{lem}\label{lem: a triangulation of a square}
Let $I$ and $J$ be two nondegenerate open intervals in $\partial \HH$.  Suppose that $\overline{I}$ and $\overline{J}$ are disjoint. We say that $S$ is the hyperbolic surface $CH(\overline{I}\cup \overline{J}).$ (See Figure \ref{fig:the convex hull of two intervals})  Assume that there is a countable set $Q$ such that $Q$ is dense in $S$ and $\partial I \cup \partial J \subset Q$. Then there is a geodesic lamination $\Lambda$ in $S$ such that the Cauchy completion of  each connected component  of $S-\Lambda$ is the ideal 3-gon and the geodesic boundaries of $S$ are contained in $\Lambda$ and such that in the corresponding lamination system $\LL$,  every leaf is isolated and $E(\LL)=Q$. 
\end{lem}
\begin{proof}
$S$ is a complete hyperbolic surface with two geodesic boundaries which are the blue lines in Figure \ref{fig:the convex hull of two intervals}. Now, we say that $S'$ is the convex hull of $\partial I \cup \partial J$ which is an ideal $4$-gon.(See Figure \ref{fig:the convex core of the convex hull}) Then $S$ is the union of $S'$, $CH(\overline{I})$ and $CH(\overline{J}).$

Now, we say that $Q_I =Q \cap \overline{I}$ and $Q_J=Q\cap \overline{J}.$ Then $Q_I$ and $Q_J$ are countable dense subsets of $\overline{I}$ and $\overline{J}$, respectively. Since $CH(\overline{I})$ and $CH(\overline{J})$ are half planes, by Lemma \ref{lem: the half farey triangulation}, there are two lamination systems $\LL_I$ and $\LL_J$ in  $CH(\overline{I})$ and $CH(\overline{J})$, respectively, such that $E(\LL_I)=Q_I$ and $E(\LL_J)=Q_J$.  

Finally, we consider $S'.$ In $S',$ we just add one diagonal geodesic like in Figure \ref{fig:the convex core with a diagonal}. The boundary geodesics of $S'$ with the diagonal geodesic give a lamination system $\LL_{S'}$in $S'$ such that  every gap of $\LL_{S'}$ is an ideal triangle, every leaf is isolated, and $E(\LL_{S'})=\partial I \cup \partial J.$ Therefore, $\LL_{S'}\cup \LL_I \cup \LL_J$ is a lamination system which we want. 
 
\end{proof}

\begin{figure}
\centering
\begin{subfigure}[b]{0.2\textwidth}
  \includegraphics[width=\textwidth]{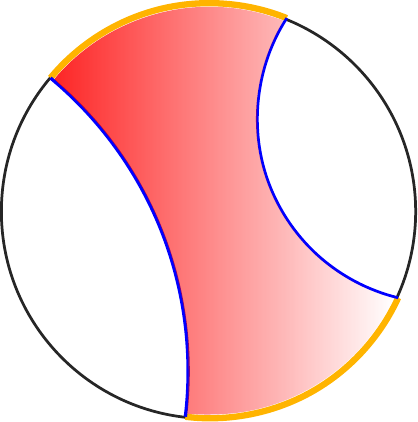}
  \caption{The red region is the convex hull of orange arcs which represent $\overline{I}$ and $\overline{J}$}
 \label{fig:the convex hull of two intervals}
\end{subfigure}
\hspace{4em}
\begin{subfigure}[b]{0.2\textwidth}
  \includegraphics[width=\textwidth]{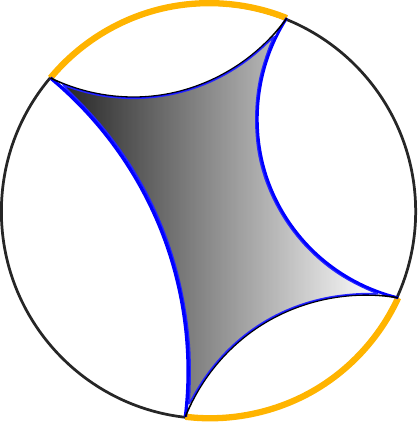}
  \caption{The shaded region is the convex hull of $\partial I$ and $\partial J$. This is a kind of convex core of $S$.}
 \label{fig:the convex core of the convex hull}
\end{subfigure}
\hspace{4em}
\begin{subfigure}[b]{0.2\textwidth}
  \includegraphics[width=\textwidth]{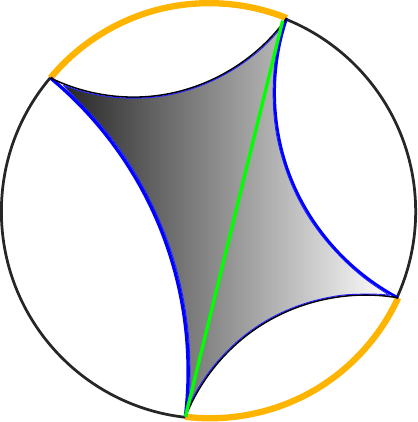}
  \caption{The green line is one of the diagonal geodesics of the shaded region. }
\label{fig:the convex core with a diagonal}
\end{subfigure}
\caption{These pictures show the steps to draw a  triangulation.}
\end{figure}

Let $G$ be an Fuchsian group of the second kind. Then $\partial \HH-L(G)$ is non-empty. Now, we consider the surface $(\overline{\HH}-L(G))/G$ with boundary.  Then $(\overline{\HH}-L(G))/G-\HH/G$ is exactly $(\partial \HH-L(G))/G$ since $G$ acts properly discontinuously on $\partial \HH-L(G).$ Therefore, we call each connected component of  $(\partial \HH-L(G))/G$ an \textsf{ideal
 boundary} of $\HH/G$ and each point of ideal boundaries an \textsf{ideal point} of $\HH/G.$ We denote the quotient map from $\overline{\HH}-L(G)$ to $(\overline{\HH}-L(G))/G$ by $\overline{\pi}_G.$ Note that $\overline{\pi}_G|_{\HH}=\pi_G.$
 
\begin{defn}
Let $G$ be a Fuchsian group. Suppose that there is a cone point $c_1$ in $\HH/G.$ 
A subset  $\ell$ of $\HH/G$ is called  a \textsf{regular simple geodesic ray} emanating from $c_1$ if there is a geodesic ray $\alpha_\ell$ from $[0,\infty)$ to $\HH$ such that $\pi_G\circ \alpha_\ell$ is injective, $(\pi_G\circ \alpha_\ell)([0,\infty))=\ell,$ $\pi_G\circ \alpha_\ell(0)=c_1,$ and for each $s\in (0,\infty)$, $\pi_G \circ \alpha_\ell(s)$ is a regular point in $\HH/G.$ In particular, when $\displaystyle \lim_{s\to \infty} \alpha_\ell(s)$ is a point $p_\infty$ in $\partial \HH - L(G),$  the point $\overline{\pi}_G(p_\infty)$ in an ideal boundary is called the \textsf{ideal end point} of the geodesic ray $\ell.$ Also, when $\displaystyle \lim_{s\to \infty} \alpha_\ell(s)$ is a cusp point of $G,$ we say that $\ell$ is \textsf{emanating} to a puncture which corresponds to the cusp point.
\end{defn} 

Let $G$ be a Fuchsian group of the second kind. Suppose that there is a regular simple geodesic ray $\ell$ emanating from a cone point $c$ having $p$ as the ideal end point. Then there is an injective continuous map $\beta_\ell$ from $[0,1]$ to $\overline{\HH}/G$ such that $\beta_\ell(0)=c,$ $\beta_\ell(1)=p,$ and $\beta_\ell([0,1))=\ell.$

\begin{prop}\label{prop: representative of geodesic rays}
Let $G$ be a Fuchsian group of the second kind and $p$ be an ideal point of  $\HH/G.$  Suppose that there is a cone point $q$ in $\HH/G.$ If there is a continuous map $\alpha$ from $[0,1]$ to $\overline{\HH}/G$ such that $\alpha(0)=q$, $\alpha(1)=p$ and  $\alpha([0,1))\cap \Sigma(\HH/G)=\{q\},$ and $\alpha|(0,1)$ is an embedding in the geometric interior $(\HH/G)^o,$ then there is an isotopy $H$ on $\overline{\HH}/G$ satisfying the following. 
\begin{enumerate}
\item $H(s,0)=\alpha(s),$
\item $H(0,t)=q,$ $H(1, t)=p,$
\item $H((0,1)\times [0,1])\cap \Sigma(\HH/G)=\emptyset$ and 
\item $H([0,1),1)$ is a regular simple geodesic ray emanating from $q$ with the ideal end point $p.$
\end{enumerate}
\end{prop}

By Proposition \ref{prop: representative of geodesic rays}, any simple embedded arc in $\overline{\HH}/G$ connecting a cone point and an ideal point has a geodesic realization. 
Now, we are ready to prove the elementary cases.

\begin{lem}\label{thm:elementary}
Let $G$ be an elementary Fuchsian group. Then $G$ is a pants-like $\COL_3$ group. 
\end{lem}
\begin{proof}

There are following four cases. 
\begin{enumerate}
\item $G$ is the trivial group. 
\item $G$ is a finite cyclic group which is generated by some elliptic element. 
\item $G$ is an infinite cyclic group which is generated by some non-elliptic element. 
\item $G$ is an infinite dihedral group which is generated by two elliptic elements of order two. 
\end{enumerate}

First, assume that $G$ is the trivial group. Assume that $P$ is a countable dense subset of $\partial \HH.$ Choose two distinct points $p_1$ and $p_2$ in $P.$ Set $I=[p_1, p_2]_{S^1}$ and $J=[p_2, p_1]_{S^1},$ and set $P_I=I\cap P$ and $P_J=J\cap P.$ Then $P_I$ and $P_J$ are countable dense subsets of $I$ and $J$, respectively, and $CH(I)$ and $CH(J)$, containing $p_1$ and $p_2$, are half planes. Therefore, by Lemma \ref{lem: the half farey triangulation}, there are two lamination systems $\LL_I$ and $\LL_J$ such that $E(\LL_I)=P_I$ and $E(\LL_J)=
P_J.$ Now, we write $\LL_P$ for $\LL_I \cup \LL_J.$ Then $\LL_P$  is a very full lamination system such that $E(\LL_I \cup \LL_J)= P.$ Observe that we can choose three countable dense subsets $P_1$, $P_2$ and $P_3$ in $\partial \HH$ that are pairwise disjoint. Therefore, $\{\LL_{P_1}, \LL_{P_2}, \LL_{P_3} \}$ is a pants-like collection. Therefore, we can conclude that $G$ is a pants-like $\COL_3$ group with  $\{\LL_{P_1}, \LL_{P_2}, \LL_{P_3} \}.$

Next, we consider the case where $G$ is a non-trivial finite cyclic group which is generated by some elliptic element. Let $n$ be the order of $G.$ We may assume that in $\DD$, the elliptic element $\phi$ defined by $z\mapsto e^{\frac{2\pi i}{n}}z $ generates $G.$ Let $Q^0$ be the set $\{ e^{2\pi i q}:q\in \QQ \}.$ Then, $Q^0$ is a countable dense subset of $\partial \DD.$ Then $\GG = \{(1, \phi(1))_{S^1}, \cdots, (\phi^{n-1}(1), \phi^{n}(1))_{S^1} \}$ is an ideal polygon with $v(\GG)\subset Q^0.$ For each $k\in \{0, \cdots, n-1\},$ $Q^0_k=[\phi^{k}(1), \phi^{k+1}(1)]_{S^1}\cap Q^0$ is a countable dense subset of $[\phi^{k}(1), \phi^{k+1}(1)]_{S^1}$ and $CH([\phi^{k}(1), \phi^{k+1}(1)]_{S^1})$ is a half plane. Therefore,  by Lemma \ref{lem: the half farey triangulation},  for each $k\in \{0, \cdots, n-1\},$  there is a lamination system $\LL^0_k$ such that $E(\LL^0_k)=Q^0_k.$ Then $\LL^0=\displaystyle \bigcup_{k=0}^{n-1} \LL^0_k$ is a very full lamination system such that $E(\LL^0)=Q^0$ and $\GG$ is a gap of $\LL^0.$ Moreover, by construction, $G$ preserves $\LL^0.$

Now, we set $Q^{\sqrt 2}=\{ e^{2\pi i (q+\sqrt 2)}:q\in \QQ \}$ and $Q^{\sqrt 3}=\{ e^{2\pi i (q+\sqrt 3)}:q\in \QQ \}.$ We can define $\LL^{\sqrt 2}$ by the set $\{(e^{2\pi i \sqrt 2}a, e^{2\pi i \sqrt 2}b )_{S^1} : (a, b)_{S^1}\in \LL^0\}$ and $\LL^{\sqrt 3}$ by the set $\{(e^{2\pi i \sqrt 3}a, e^{2\pi i \sqrt 3}b )_{S^1} : (a, b)_{S^1}\in \LL^0\}.$ Then $E(\LL^{\sqrt2})=Q^{\sqrt2}$ and $E(\LL^{\sqrt3})=Q^{\sqrt3}.$ Therefore, $\{\LL^0, \LL^{\sqrt2}, \LL^{\sqrt 3}\}$ is a pants-like $\COL_3$ collection. Moreover, these lamination systems are preserved by $G$ by construction. Therefore, $G$ is a pants-like $\COL_3$ group. 

Then we consider the case where $G$ is an infinite cyclic group which is generated by some non-elliptic element. First, we assume that in $\HH,$ $G$ is generated by the parabolic element $\phi$ defined by $z\mapsto z+1.$ We set $Q^0=\QQ \cup \infty.$ Since $Q^0\cap [0,1]$ is a countable dense subset of $[0,1],$ and $CH([0,1])$ is a half plane, by Lemma \ref{lem: the half farey triangulation}, there is a lamination system $\LL^0_0$ such that $E(\LL^0_0)=p^{-1}(Q^0\cap [0,1]).$ Now, we define
 $$\LL^0=\displaystyle \left[\bigcup_{k\in \ZZ}(p^{-1}\phi^{k}p)(\LL_0^0)\right] \cup \left[\bigcup_{k\in \ZZ}\{ \phi^k(p^{-1}((0, \infty))),\phi^k(p^{-1}((0,\infty))^*) \}) \right].$$ Obviously, $\LL^0$ is a very full $G$-invariant lamination system such that $E(\LL^0)=p^{-1}(Q^0).$
 
 Now we set $Q^{\sqrt 2}=\sqrt2 + \QQ$ and $Q^{\sqrt 3}=\sqrt3 + \QQ.$ Then, we define 
 $$\LL^{\sqrt2 }=\{(a, b)_{S^1}: p^{-1}((-\sqrt2)+p((a,b)_{S^1}))\in \LL^0 \}$$ and 
  $$\LL^{\sqrt3 }=\{(a, b)_{S^1}: p^{-1}((-\sqrt3)+p((a,b)_{S^1}))\in \LL^0  \}.$$
Then $E(\LL^{\sqrt2})=p^{-1}(Q^{\sqrt2})$ and $E(\LL^{\sqrt3})=p^{-1}(Q^{\sqrt3})$, and both $\LL^{\sqrt2}$ and $\LL^{\sqrt3}$ are very full $G$-invariant lamination systems. Therefore, $\{\LL^0, \LL^{\sqrt2}, \LL^{\sqrt3}\}$ is a pants-like $\COL_3$ collection since for two distinct indices  $i$ and $j$ in $\{0, \sqrt2, \sqrt3\},$  $Q^i \cap Q^j=\{\infty\}$ and $\infty$ is the fixed point of $\phi.$ Therefore, $G$ is a pants-like $\COL_3$ group. 

Then we assume that in $\HH,$ $G$ is generated by a hyperbolic element $\phi$ defined by $z \mapsto e^a z$ for some positive real number $a.$ We set $Q^0=\{e^{aq}: q\in \QQ\} \cup  \{-e^{aq}: q\in \QQ\}.$ Observe that $Q^0$ is preserved by $G$ and $\phi$ maps the geodesic connecting $-1$ and $1$ to the geodesic connecting $-e^a$ and $e^a.$ Moreover, $p^{-1}(Q^0)$ is a countable dense subset of $\partial \DD.$ Lemma \ref{lem: a triangulation of a square} can apply to two closed interval $p^{-1}([-e^a, -1])$ and $p^{-1}([1, e^a])$ since $\{-e^a, -1, 1, e^a\} \subset Q^0$ and both $[-e^a, -1]\cap Q^0$ and $[1, e^a]\cap Q^0$ are countable dense subsets of $[-e^a, -1]$ and $[1, e^a],$ respectively, Therefore, there is a lamination system $\LL_0^0$ such that  $E(\LL_0^0)=p^{-1}(Q^0\cap ([-e^a, -1]\cup[1, e^a])).$ Then, $$\LL^0=\displaystyle \bigcup_{k\in \ZZ} (p^{-1}\phi^kp)(\LL_0^0)$$ is a $G$-invariant very full lamination system such that $E(\LL^0)=p^{-1}(Q^0).$  

Now, we set $Q^{\sqrt2}=\{e^{a(\sqrt2+q)}: q\in \QQ\} \cup  \{-e^{a(\sqrt2+q)}: q\in \QQ\},$ and  $Q^{\sqrt3}=\{e^{a(\sqrt3+q)}: q\in \QQ\} \cup  \{-e^{a(\sqrt3+q)}: q\in \QQ\}.$ In the similar way, we can find two very full $G$-invariant lamination systems $\LL^{\sqrt2}$ and $\LL^{\sqrt3}$ such that 
$\LL^{\sqrt2}=p^{-1}(Q^{\sqrt2})$ and $\LL^{\sqrt3}=p^{-1}(Q^{\sqrt3}).$ Then, by construction, $\{\LL^0, \LL^{\sqrt2}, \LL^{\sqrt 3}\}$ is a pants-like $\COL_3$ collection. Therefore, $G$ is a pants-like $\COL_3$ group.

Finally, we consider the case where $G$ is an infinite dihedral group generated by two elliptic elements of order two. In $\HH,$ for each non-negative real number $a,$ there is a unique order two elliptic isometry $\phi_{a/2}$ that fixes $ie^{a/2}.$ We may assume that $G$ is generated by $\phi_0$ and $\phi_{a/2}$ for some positive real number $a.$ Note that $\phi_{a/2} \circ \phi_0$ is the hyperbolic element defined by $z\mapsto e^az.$

Then we can see that there are exactly two cone points $c_1=\pi_G(i)$ and $c_2=\pi_G(ie^{a/2})$ of order two and that there is the only one ideal boundary $B$ in $\HH/G$ which is homeomorphic to the circle. We can find three countable dense subsets $P_1$, $P_2$ and $P_3$ of $B$ that are pairwise disjoint. First, we consider $P_1.$ Choose two distinct points $p_1^1$ and $p_1^2$ in $P_1.$ Then, by Proposition \ref{prop: representative of geodesic rays}, there are two disjoint regular simple geodesic rays $\ell_1^1$ and $\ell_1^2$ such that for each $i\in \{1, 2\},$ $\ell_1^i$ is a regular simple geodesic ray emanating from $c_i$ with the ideal end point $p_1^i.$ Note that by definition, $\ell_1^1$ and $\ell_1^2$ are simple geodesics in $\HH/G$ since $c_1$ and $c_2$ are of order two. Hence,  there are two bi-infinite geodesics $\tilde{\ell}_1^1$ and $\tilde{\ell}_1^2$ such that for each $k\in \{1, 2\}$, $\tilde{\ell}_1^k$ is the connected component of $\pi_G(\ell_1^k)$ passing through $ie^{(k-1)a/2}.$ 
Also, for each $k\in \{1,2\}$, there are two points $p_-^k$ and $p_+^k$ in $\RR-\{0\}$ such that $p_-^k<0$, $p_+^k>0,$ and $CH(\{p_-^k,p_+^k \})=\tilde{\ell}_1^k.$

Now, we write $I_1=[p_-^2, p_-^1]$ and $J_1=[p_+^1, p_+^2].$ Observe that $L(G)=\{0, \infty \}$ and that $\overline{\pi}_G^{-1}(P_1)$ is a countable dense subset of $\RR-\{0\}.$ 
Then $\partial I_1$ and $\partial J_1$ are contained in $\overline{\pi}_G^{-1}(P_1)$, $\overline{\pi}_G^{-1}(P_1)\cap I_1$ is dense in $I_1,$  and $\overline{\pi}_G^{-1}(P_1)\cap J_1$ is dense in $J_1.$ So, we can apply Lemma \ref{lem: a triangulation of a square} to $CH(I_1 \cup J_1).$ Therefore, there is a geodesic lamination $\Lambda_1^0$ on $CH(I_1 \cup J_1).$ 
Then
$$\Lambda_1=\displaystyle \bigcup_{g\in G}g(\Lambda_1^0)$$
is also a geodesic lamination preserved under the $G$-action. Therefore, there is a $G$-invariant very full lamination system $\LL_1$, induced by $\Lambda_1$, such that $p(E(\LL_1))=\overline{\pi}_G^{-1}(P_1).$ Similarly, there are two $G$-invariant very full lamination systems $\LL_2$ and $\LL_3$ such that $p(E(\LL_2))=\overline{\pi}_G^{-1}(P_2)$ and $p(E(\LL_3))=\overline{\pi}_G^{-1}(P_3).$ By construction, $\{\LL_1, \LL_2, \LL_3\}$ is a pants-like $\COL_3$ collection. Thus, $G$ is a pants-like $\COL_3$ group.  
\end{proof} 

From now on, we consider the case where $G$ is a non-elementary Fuchsian group of the second kind. Let $R(G)$ be the set of all ramification points of $G$ in $\Omega(G).$ Note that $R(G)\cup L(G)$ is a closed subset of the Riemann sphere $\hat{\CC}.$ Hence, the open subset $\Omega(G)-R(G)=\hat{\CC}-(R\cap L(G))$ of $\hat{\CC}$ is a Riemann surface. Now, we write $U_G=\Omega(G)-R(G).$ Obviously, $\pi_1(U_G)$ is a free group of infinite rank and by Uniformization Theorem, the universal cover of $U_G$ is $\HH$, so $U_G$ is a complete hyperbolic surface. Observe that there is no isometrically embedded half plane and funnel in $U_G.$  Therefore, $G(U_G)$ is a Fuchsian group of the first kind. Now, we define 
$$F_G=\{h\in \Aut(\HH): \pi_{U_G}\circ h = g \circ \pi_{U_G} \ \text{for some} \ g \in G\}.$$ 
Note that $G(U_G)$ is the deck transformation group of $\pi_{U_G}$ and that $F_G$ is isomorphic to $\pi_1(U_G/G)$ since  $\pi_G^d \circ \pi_{U_G}$ is the universal covering map of $U_G/G.$

Now, we define a homomorphism $\pi_{U_G}^*$ from $F_G$ to $G$ so that for each $h\in F_G,$ $\pi_{U_G}\circ h= \pi_{U_G}^*(h)\circ \pi_{U_G}.$ Then we have an exact sequence
 \begin{center}
\begin{tikzcd}
\{1\} \arrow[r] & G(U_G) \arrow[r, "i"] & F_G\arrow[r, "\pi_{U_G}^*"]& G \arrow[r] &\{1\}
 \end{tikzcd}
\end{center}
where $i$ is the inclusion map. Hence, $F_G/G(U_G)$ is isomorphic to $G$ and since $\hat{\RR}=L(G(U_G))\subset L(F_G),$ $F_G$ is a Fuchsian group of the first kind. Therefore,  there is no isometrically embedded half plane and funnel in $U_G/G=\HH/F_G$ and  there is no cone point in $U_G/G$ by construction. Hence, $U_G/G$ is a complete hyperbolic surface of which the fundamental group is a Fuchsian group of the first kind. 

By Lemma \ref{lem : three geometric pants decompositions}, there are three pairwise transverse geometric pants decompositions $\Lambda_1,$ $\Lambda_2,$ and  $\Lambda_3$ in $U_G/G.$ Observe that $U_G/G$ is the doubling of $(\HH/G)^o$ about the ideal boundaries as the following diagram commutes. 
\begin{center}
\begin{tikzcd}
\HH-R(G)\arrow[r,hook] \arrow[d, "\pi_G"]&U_G \arrow[r,hook] \arrow[d, "\pi_G^d"] & \Omega(G)   \arrow[d, "\pi_G^d"] \\
(\HH/G)^o\arrow[r,hook]&U_G/G	\arrow[r,hook] & \Omega(G)/G   
\end{tikzcd}
\end{center}
where the horizontal arrows are the inclusion maps. Note that the inclusion maps are not isometries and that the map $J$ induced from the complex conjugation is an isometry on $U_G/G$ fixing every point of $(\hat{\RR}-L(G))/G.$ Therefore, we can see that each ideal boundary of $(\HH/G)^o$ is a simple geodesic in $U_G/G.$ Now, we write $B$ for $(\hat{\RR}-L(G))/G.$ Then, $B$ is a geodesic lamination in $U_G/G$.

Let $P$ be the completion of a connected component of the complement of $\Lambda_k$ for some $k\in \{1, 2, 3\}.$ Since there are no cone points in $U_G/G,$ there are three types of geometric pairs of pants which have only simple closed geodesics and punctures as geometric boundaries. See Figure \ref{fig:standard pants}. 

Now, let $b$ be a connected component of $B.$ If $b$ is a simple closed curve which corresponds to the ideal boundary of a funnel in $\HH/G,$ then for each $i\in \{1,2,3\},$ $b\cap \Lambda_i$ is $b$ itself or a finite subset of $b.$ If $b$ is a simple bi-infinite geodesic which corresponds to the ideal boundary of a half-plane, then for each $i\in \{1,2,3\},$ $b \cap \Lambda_i$ is a countable discrete closed subset of $b$  since any bi-infinite geodesic $\tilde{b}$ such that $(\pi_G^d\circ \pi_{U_G})(\tilde{b})=b$ has no cusp point as an end point.  In particular, when $b$ intersects with $P,$  each connected component of $b\cap P$ is a properly embedded arc in $P$ or a geodesic boundary of $P.$ On the other hands, if $c$ is a simple closed curve in $\Lambda_i$ for some $i\in \{1,2,3\},$ then $c \cap ((\HH/G)^o \cup B)$ is $c$ itself or a disjoint union of properly embedded arcs.

Now, we assume that $P\cap B\neq \emptyset.$ Let $P'$ be a connected component of $P\cap ((\HH/G)^o \cup B).$ Then, we can see that  each manifold boundary of $P'$ is a geodesic boundary of $P$ or consists of an alternative sequence of arcs in $B$ and arcs in $\partial P.$ Hence, each boundary of $P'$ is a piecewise geodesic circle. 

\begin{figure}
\centering
\begin{subfigure}[b]{0.25\textwidth}
  \includegraphics[width=\textwidth]{pants1.pdf}
  \caption{Three geodesic boundaries}
  \label{fig:standard pants1}
\end{subfigure}
\hfill
\begin{subfigure}[b]{0.22\textwidth}
  \includegraphics[width=\textwidth]{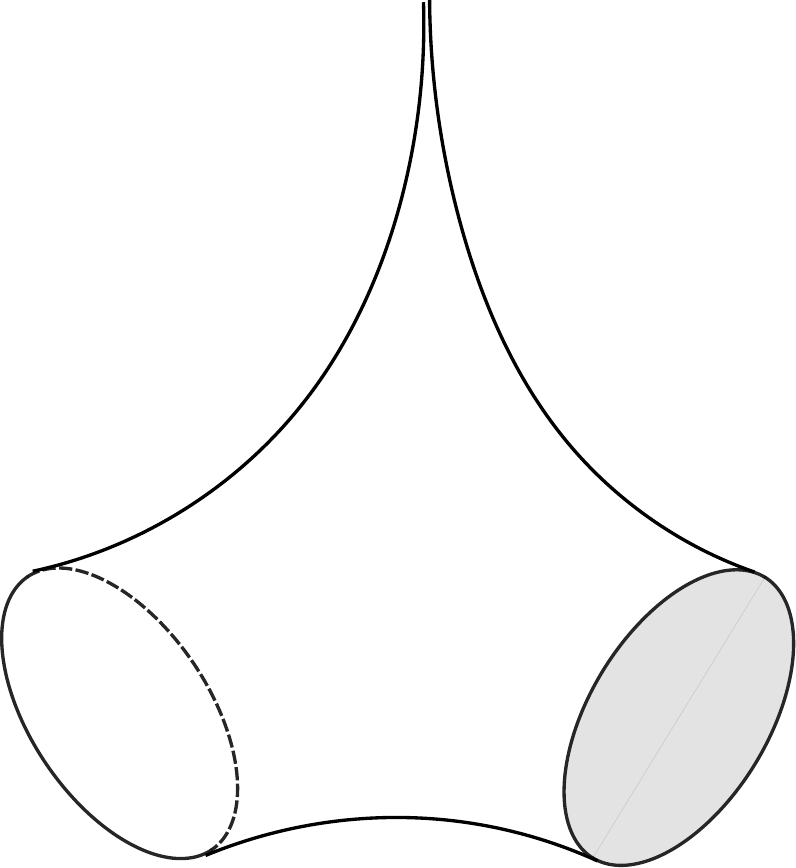}
  \caption{Two geodesic boundaries}
  \label{fig:standard pants2}
\end{subfigure}
\hfill\begin{subfigure}[b]{0.2\textwidth}
  \includegraphics[width=\textwidth]{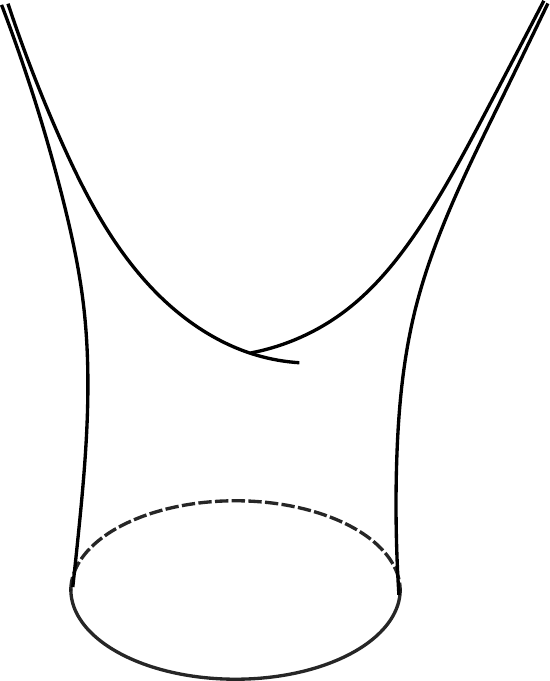}
  \caption{One geodesic boundary}
  \label{fig:standard pants3}
\end{subfigure}

\caption{In $\Lambda_i,$ there are three types of geometric pairs of pants according to the number of closed geodesic boundary components since $U_G/G$ has no cone point and has no half-plane and funnel.}
\label{fig:standard pants}
\end{figure}

From now on, we discuss the shape of $P'.$ First, we consider the case where there is an embedded arc in $P \cap B$ having the end points in one geodesic boundary of $P.$ See Figure \ref{fig:standard pants with return arc}. In each case, the red line represents the embedded arc and $P$ can have a finite number of parallel arcs in $P\cap B$ to the red line which are also geodesic. These parallel arcs divide $P$ into three types of pieces. See Figure \ref{fig:elementary pieces in type one}. The red lines are arcs in $P\cap B$ and the black lines are arcs in $\partial P.$ The small circle in Figure \ref{fig: monogon with puncture} is a puncture and the round circle in Figure \ref{fig: monogon with boundary} is a geodesic boundary of $P.$ We can see that a piece of the first two types can not intersect with $B$ in the complement of the red lines. But, a piece of the third type has three possibilities. The first is that there is no intersection with $B$ except the red lines. The second case is that the piece has an embedded arc in $P\cap B$ connecting two distinct black lines. See Figure \ref{fig: monogon with boundary_case1}. The additional red line represents the arc. Note that since the arc is not separating, there are even number of parallel arcs. Then, the parallel arcs divide the piece into pieces which are of the type of Figure \ref{fig:square} or of the type of Figure \ref{fig: haxagon}. Finally, in Figure \ref{fig: monogon with boundary}, the geodesic boundary is in $B$. See Figure \ref{fig: monogon with boundary_case2}. 
Thus, Figure \ref{fig:square},  Figure \ref{fig: monogon with puncture}, Figure \ref{fig: monogon with boundary}, Figure \ref{fig: monogon with boundary_case2}  and Figure \ref{fig: haxagon} are all possible cases of $P'$ when $P$ has an arc in $P \cap B$ having the end points in one geodesic boundary.

\begin{figure}
\centering
\begin{subfigure}[b]{0.25\textwidth}
  \includegraphics[width=\textwidth]{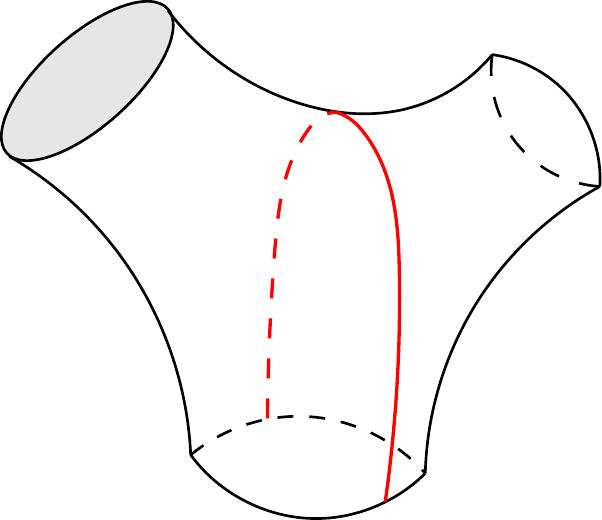}
  \caption{Three geodesic boundaries}
  \label{fig:standard pants1 with return arc}
\end{subfigure}
\hfill
\begin{subfigure}[b]{0.22\textwidth}
  \includegraphics[width=\textwidth]{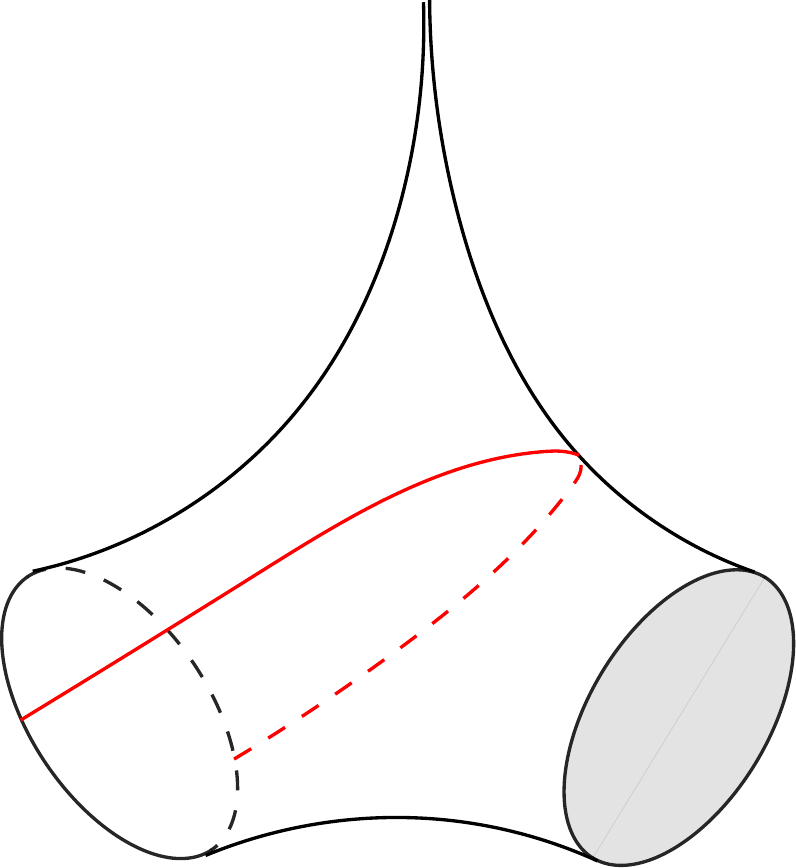}
  \caption{Two geodesic boundaries}
  \label{fig:standard pants2 with return arc}
\end{subfigure}
\hfill\begin{subfigure}[b]{0.2\textwidth}
  \includegraphics[width=\textwidth]{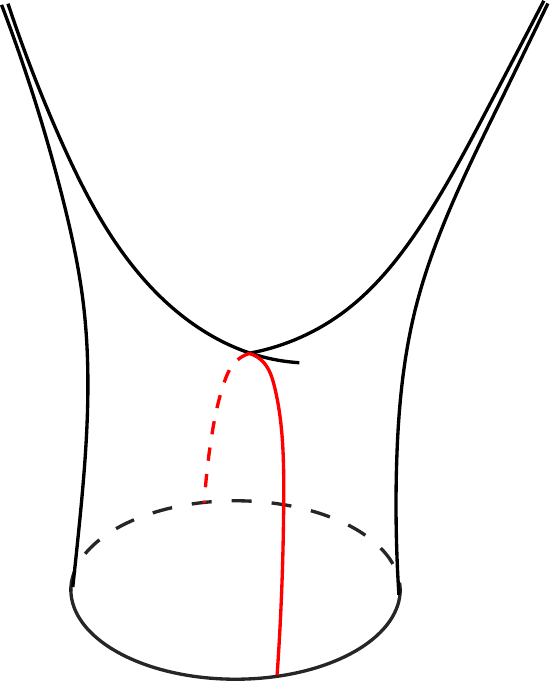}
  \caption{One geodesic boundary}
  \label{fig:standard pants3 with return arc}
\end{subfigure}

\caption{The red lines are curves in $P\cap B$.}
\label{fig:standard pants with return arc}
\end{figure}

\begin{figure}
\centering
\begin{subfigure}[b]{0.2\textwidth}
  \includegraphics[width=\textwidth]{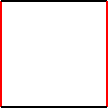}
  \caption{A square}
  \label{fig:square}
\end{subfigure}
\hfill
\begin{subfigure}[b]{0.25\textwidth}
  \includegraphics[width=\textwidth]{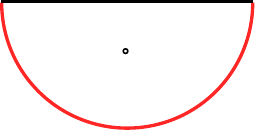}
  \caption{A bi-gon with one puncture}
  \label{fig: monogon with puncture}
\end{subfigure}
\hfill\begin{subfigure}[b]{0.25\textwidth}
  \includegraphics[width=\textwidth]{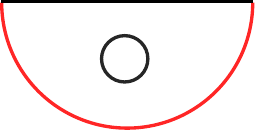}
  \caption{A bi-gon with one closed geodesic boundary}
  \label{fig: monogon with boundary}
\end{subfigure}

\caption{There are three types of pieces from Figure \ref{fig:standard pants with return arc}. The red lines are embedded arcs in $P\cap B$ and the black lines are arcs in $\partial P$. The small circle in Figure \ref{fig: monogon with puncture} represents a puncture. }
\label{fig:elementary pieces in type one}
\end{figure}

\begin{figure}
\centering
\begin{subfigure}[b]{0.25\textwidth}
  \includegraphics[width=\textwidth]{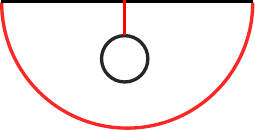}
  \caption{A bi-gon with one closed geodesic boundary and additional arcs}
  \label{fig: monogon with boundary_case1}
\end{subfigure}
\hspace*{2.5cm}
\begin{subfigure}[b]{0.25\textwidth}
  \includegraphics[width=\textwidth]{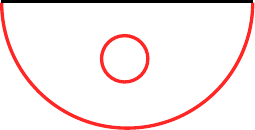}
  \caption{A bi-gon with one closed geodesic boundary which is in $B$}
  \label{fig: monogon with boundary_case2}
\end{subfigure}

\caption{The red lines are curves in $P\cap B$.}
\label{fig:elementary pieces in type one_two subtypes}
\end{figure}

Now, we assume that there is no properly embedded arc in $P\cap B$ having two end points in one geodesic boundary. Also, suppose that there is a properly embedded arc in $P \cap B$ connecting two distinct geodesic boundaries of $P.$ Then, there are four possible cases according to the number of homotopic types of arcs. See Figure \ref{fig:standard pants with non-return arc}. The red lines represent the homotopy classes of the embedded arcs in $P\cap B.$

In Figure \ref{fig:standard pants3 with arc}, the red line is not separating so there are even number of parallel arcs which are homotopic to the red line. Hence, these parallel arcs divide $P$ into pieces which are of the types of Figure \ref{fig:square} or Figure \ref{fig:square with puncture}. Similarly, in Figure \ref{fig:standard pants1 with arc}, the red line is not separating so there are even number of parallel arcs which are homotopic to the red line. These parallel arcs divide $P$ into pieces which are of the types of Figure \ref{fig:square} or Figure \ref{fig:square with boundary}. In particular, in a piece of the type of Figure \ref{fig:square with boundary}, the geodesic boundary may be in $P\cap B.$ Therefore,  Figure \ref{fig:square with red boundary} is also a possible case. Finally, from Figure \ref{fig:standard pants1 with two arcs} and Figure \ref{fig:standard pants1 with three arcs}, we can get pieces which are of the types of Figure \ref{fig:square}, Figure \ref{fig: haxagon}, or Figure \ref{fig: octagon}. Therefore, Figure \ref{fig:square}, Figure \ref{fig: haxagon}, Figure \ref{fig: octagon}, Figure \ref{fig:square with puncture}, Figure \ref{fig:square with boundary}, and Figure \ref{fig:square with red boundary} are all possible cases of $P'$ when $P\cap B$ has an arc connecting two distinct boundaries. 

\begin{figure}
\centering
\begin{subfigure}[b]{0.2\textwidth}
  \includegraphics[width=\textwidth]{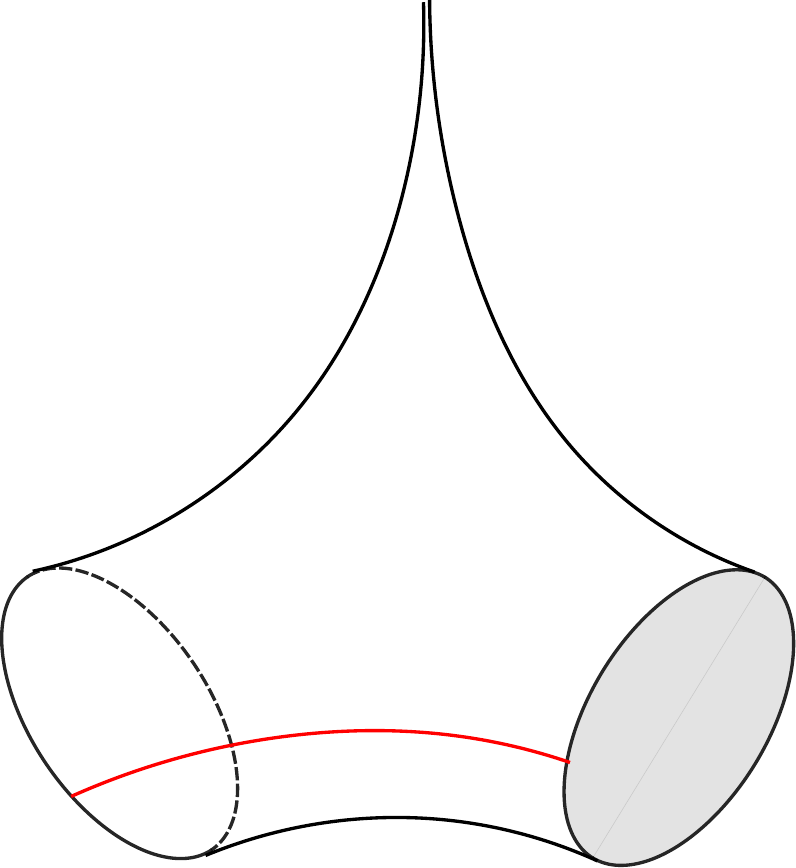}
  \caption{Only one homotopy class of arcs}
  \label{fig:standard pants3 with arc}
\end{subfigure}
\hfill
\begin{subfigure}[b]{0.2\textwidth}
  \includegraphics[width=\textwidth]{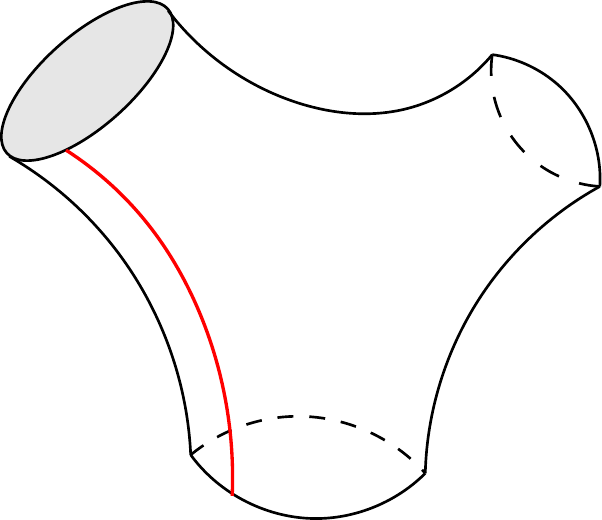}
  \caption{One homotopy class of arcs}
  \label{fig:standard pants1 with arc}
\end{subfigure}
\hfill
\begin{subfigure}[b]{0.2\textwidth}
  \includegraphics[width=\textwidth]{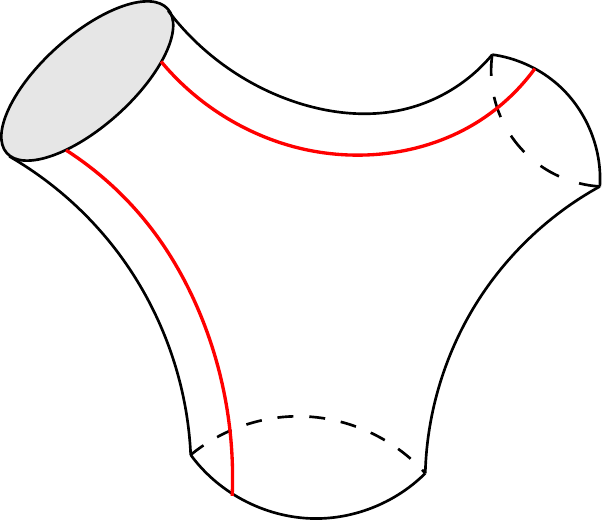}
  \caption{Two homotopy classes of arcs}
  \label{fig:standard pants1 with two arcs}
\end{subfigure}
\hfill
\begin{subfigure}[b]{0.2\textwidth}
  \includegraphics[width=\textwidth]{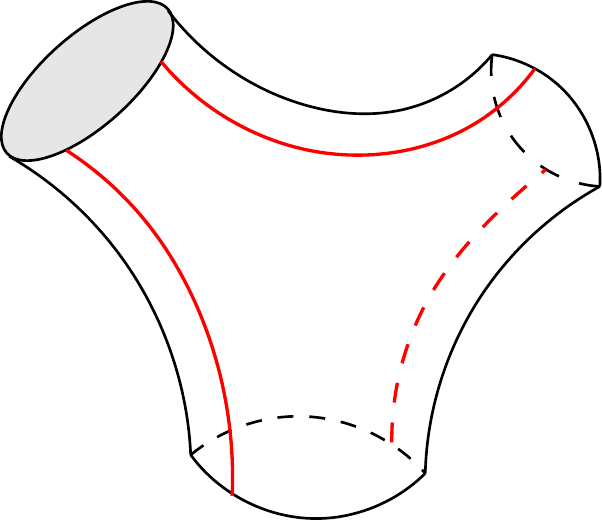}
  \caption{Three homotopy classes of arcs}
  \label{fig:standard pants1 with three arcs}
\end{subfigure}

\caption{The red lines are curves in $P\cap B$.}
\label{fig:standard pants with non-return arc}
\end{figure}

\begin{figure}
\centering
\begin{subfigure}[b]{0.2\textwidth}
  \includegraphics[width=\textwidth]{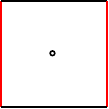}
  \caption{A square with one puncture}
  \label{fig:square with puncture}
\end{subfigure}
\hfill
\begin{subfigure}[b]{0.2\textwidth}
  \includegraphics[width=\textwidth]{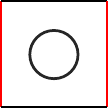}
  \caption{A square with one boundary}
  \label{fig:square with boundary}
\end{subfigure}
\hfill
\begin{subfigure}[b]{0.2\textwidth}
  \includegraphics[width=\textwidth]{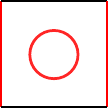}
  \caption{A square with one boundary which is in $P\cap B$}
  \label{fig:square with red boundary}
\end{subfigure}
\caption{The red lines are embedded arcs in $P\cap B$ and the black lines are arcs in $\partial P$. The small circle represents a puncture.}

\end{figure}

\begin{figure}
\centering
\begin{subfigure}[b]{0.28\textwidth}
  \includegraphics[width=\textwidth]{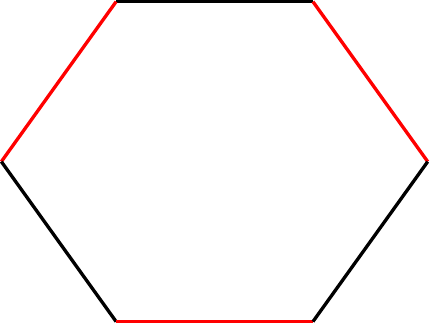}
  \caption{A hexagon}
  \label{fig: haxagon}
\end{subfigure}
\hspace*{2.5cm}
\begin{subfigure}[b]{0.22\textwidth}
  \includegraphics[width=\textwidth]{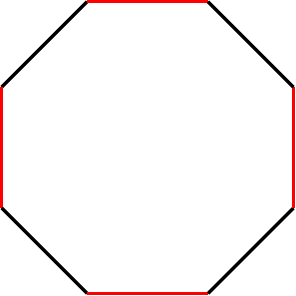}
  \caption{A octagon}
  \label{fig: octagon}
\end{subfigure}

\caption{The red lines are embedded arcs in $P\cap B$ and the black lines are arcs in $\partial P$.}
\label{fig:elementary pieces in type one_two subtypes}
\end{figure}

Then, we consider the case where $P\cap B$ has no arc. Then, $P\cap B$ consists of closed geodesics. Therefore, we can list up the possible types of $P'$ according to the number of boundaries contained in $P\cap B.$ See Figure \ref{fig: pants without arc}. Thus, we can summarize all possible cases of $P'$ in Figure \ref{fig: all possible types}.

\begin{figure}
\centering
\begin{subfigure}[b]{0.2\textwidth}
  \includegraphics[width=\textwidth]{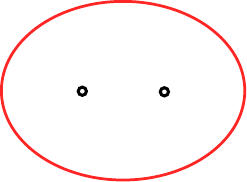}
  \caption{Two punctures and one boundary which is in $B$}
  \label{fig:pants3_one}
\end{subfigure}
\hspace*{2.5cm}
\begin{subfigure}[b]{0.2\textwidth}
  \includegraphics[width=\textwidth]{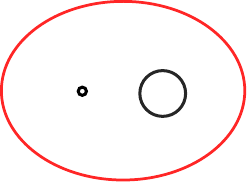}
  \caption{One puncture and two boundaries of which one is in $B$}
  \label{fig:pants2_one}
\end{subfigure}
\hspace*{2.5cm}
\begin{subfigure}[b]{0.2\textwidth}
  \includegraphics[width=\textwidth]{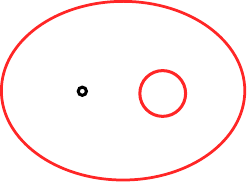}
  \caption{One puncture and two boundaries which are in $B$}
  \label{fig:pants2_two}
\end{subfigure}

\begin{subfigure}[b]{0.2\textwidth}
  \includegraphics[width=\textwidth]{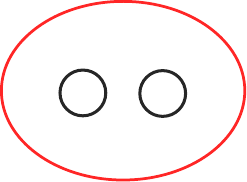}
  \caption{Only one boundary is in $B$ }
  \label{fig:pants1_one}
\end{subfigure}
\hspace*{2.5cm}
\begin{subfigure}[b]{0.2\textwidth}
  \includegraphics[width=\textwidth]{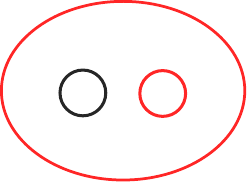}
  \caption{Two boundaries are in $B$}
  \label{fig:pants1_two}
\end{subfigure}
\hspace*{2.5cm}
\begin{subfigure}[b]{0.2\textwidth}
  \includegraphics[width=\textwidth]{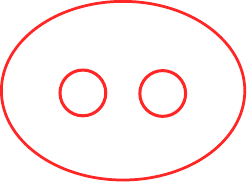}
  \caption{All boundaries are in $B$}
  \label{fig:pants1_three}
\end{subfigure}

\caption{The red circles are closed geodesics in $P\cap B$ and the black circles are closed geodesic in $\partial P -B.$ The small circles represent punctures.}
\label{fig: pants without arc}
\end{figure}

\begin{figure}
\centering
\begin{subfigure}[b]{0.2\textwidth}
  \includegraphics[width=\textwidth]{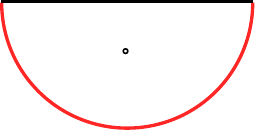}
  \caption{A bi-gon with one puncture}
  \label{fig:1-1}	
\end{subfigure}
\hspace*{2.5cm}
\begin{subfigure}[b]{0.2\textwidth}
  \includegraphics[width=\textwidth]{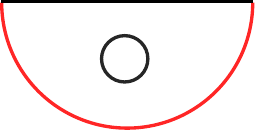}
  \caption{A bi-gon with one closed geodesic boundary}
  \label{fig:1-2}	
\end{subfigure}
\hspace*{2.5cm}
\begin{subfigure}[b]{0.2\textwidth}
  \includegraphics[width=\textwidth]{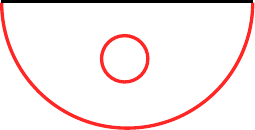}
  \caption{A bi-gon with one closed geodesic boundary which is in $B$}
  \label{fig:1-3}	
\end{subfigure}

\begin{subfigure}[b]{0.15\textwidth}
  \includegraphics[width=\textwidth]{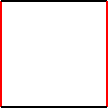}
  \caption{A square}  
  \label{fig:2-1}	
\end{subfigure}
\hspace*{2.6cm}
\begin{subfigure}[b]{0.22\textwidth}
  \includegraphics[width=\textwidth]{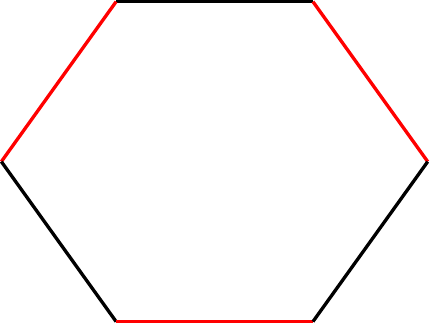}
  \caption{A hexagon}
  \label{fig:2-2}	
\end{subfigure}
\hspace*{2.6cm}
\begin{subfigure}[b]{0.17\textwidth}
  \includegraphics[width=\textwidth]{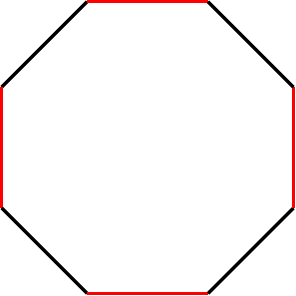}
  \caption{A octagon}
  \label{fig:2-3}	
\end{subfigure}

\begin{subfigure}[b]{0.15\textwidth}
  \includegraphics[width=\textwidth]{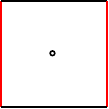}
  \caption{A square with one puncture}
  \label{fig:3-1}	
\end{subfigure}
\hspace*{3.3cm}
\begin{subfigure}[b]{0.15\textwidth}
  \includegraphics[width=\textwidth]{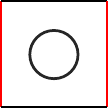}
  \caption{A square with one boundary}
  \label{fig:3-2}	
\end{subfigure}
\hspace*{3.3cm}
\begin{subfigure}[b]{0.15\textwidth}
  \includegraphics[width=\textwidth]{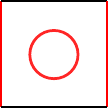}
  \caption{A square with one boundary which is in $P\cap B$}
  \label{fig:3-3}	
\end{subfigure}

\begin{subfigure}[b]{0.2\textwidth}
  \includegraphics[width=\textwidth]{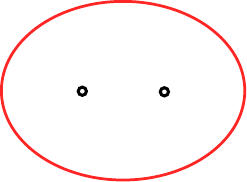}
  \caption{Two punctures and one boundary which is in $B$}
  \label{fig:4-1}	
\end{subfigure}
\hspace*{2.5cm}
\begin{subfigure}[b]{0.2\textwidth}
  \includegraphics[width=\textwidth]{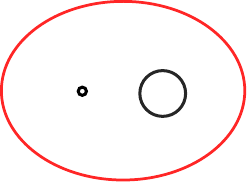}
  \caption{One puncture and two boundaries of which one is in $B$}
  \label{fig:4-2}	
\end{subfigure}
\hspace*{2.5cm}
\begin{subfigure}[b]{0.2\textwidth}
  \includegraphics[width=\textwidth]{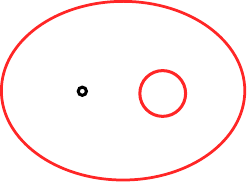}
  \caption{One puncture and two boundaries which are in $B$}
  \label{fig:4-3}	
\end{subfigure}

\begin{subfigure}[b]{0.2\textwidth}
  \includegraphics[width=\textwidth]{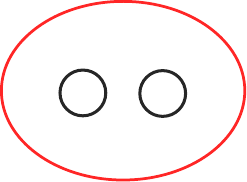}
  \caption{Only one boundary is in $B$ }
  \label{fig:5-1}	
\end{subfigure}
\hspace*{2.5cm}
\begin{subfigure}[b]{0.2\textwidth}
  \includegraphics[width=\textwidth]{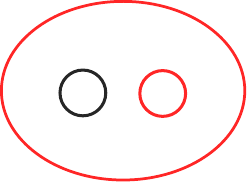}
  \caption{Two boundaries are in $B$}
  \label{fig:5-2}	
\end{subfigure}
\hspace*{2.5cm}
\begin{subfigure}[b]{0.2\textwidth}
  \includegraphics[width=\textwidth]{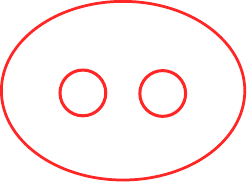}
  \caption{All boundaries are in $B$}
  \label{fig:5-3}	
\end{subfigure}

\caption{The red circles are closed geodesics in $P\cap B$ and the black lines are in $\partial P.$ The small circles represent punctures.}
\label{fig: all possible types}
\end{figure}

Now, for a simple closed geodesic $\gamma$ in $U_G/G,$ we set the collar width $c(\gamma)$ to be $\ln(\coth (l(\gamma)/4))$ where $l(\gamma)$ is the length of $\gamma$ with respect to the hyperbolic metric of $U_G/G.$ Then, we write $U(\gamma)$ for the $c(\gamma)$- neighborhood of $\gamma$ in $U_G/G$ which is homeomorphic to the annulus. Note that for two disjoint simple closed geodesics $\alpha$ and $\beta$ in $U_G/G,$ $U(\alpha)\cap U(\beta)=\emptyset$ and if $b$ is a connected component of $B$ and $b\cap \alpha=\emptyset,$ then $b \cap U(\alpha)=\emptyset.$

Then, we consider the restrictions $\mathcal{C}_i=\Lambda_i \cap ((\HH/G)^o \cup B)$ of the geodesic laminations.  $\mathcal{C}_i$ are a disjoint union of simple closed curves and properly embedded arcs in $(\HH/G)^o \cup B.$ First, we remove closed curves in $\mathcal{C}_i\cap B$ from $\mathcal{C}_i.$ Then, choose a connected component $\alpha$ of $\mathcal{C}_1.$ Now, we define a regular neighborhood $\overline{V}(\alpha)$ in $(\HH/G)^o \cup B$ as follows.  If $\alpha$ is a simple closed curve, then $\overline{V}(\alpha)=U(\alpha)$ and if $\alpha$ is a properly embedded arc, then there is the simple closed curve $\gamma$ in $\Lambda_1$ containing $\alpha$ and $\overline{V}(\alpha)$ is the connected component of $U(\gamma)\cap ((\HH/G)^o \cup B)$ containing $\alpha.$ Hence, each pair of two distinct connected components of $\mathcal{C}_1$ have disjoint regular neighborhoods. Similarly, we can define regular neighborhoods for connected components of $\mathcal{C}_2$ and  $\mathcal{C}_3.$ 

We can find three disjoint countable dense subsets $Q_1,$ $Q_2,$ and $Q_3$ of $B.$ First, we consider $\mathcal{C}_1.$ Suppose that $\gamma$ is a proper embedding of $[0,1]$ such that $\gamma([0,1])\subset \mathcal{C}_1.$ Then there is an isotopy $H$ through proper embeddings on $(\HH/G)^o\cup B$ satisfying the following:
\begin{enumerate}
\item $H(s,0)=\gamma(s),$
\item $H([0,1]\times [0,1])\subset \overline{V}(\gamma([0,1])),$
\item for all $t\in [0, 1]$, $\{H(0,t), H(1,t)\} \subset B,$ and
\item $\{H(0,1), H(1,1)\} \subset Q_1.$
\end{enumerate}
Then we write $\gamma'(s)=H(s,1).$ Then $\gamma'([0,1])$ is an properly embedded arc in $(\HH/G)^o\cup B$  of which the end points are in $Q_1$ and is contained in $\overline{V}(\gamma([0,1])).$ 
Now in $\mathcal{C}_1$, we replace $\gamma([0,1])$ with $\gamma'([0,1]).$ Since the regular neighborhoods are disjoint, we can properly replace arcs in $\mathcal{C}_1$ so that every arc in $\mathcal{C}_1$ has the end points in $Q_1.$ Similarly, we can assume that for each $i\in \{1,2,3\},$ every arc in $\mathcal{C}_i$ has the end points in $Q_i$ after properly replacing arcs. 

Now, we consider $\mathcal{C}_i$ in $\HH/G.$ Observe that for each arc $\alpha$ in $\mathcal{C}_i,$ there is a unique simple bi-infinite geodesic $\ell_\alpha$ in $\HH/G$ which has same homotopy type with $\alpha$ and connects two end points of $\alpha.$ Note that $\ell_\alpha$ does not pass through a cone point of $\HH/G.$ We say that $\ell_\alpha$ 
is the \textsf{geodesic realization} of $\alpha.$ We define $\Lambda_i^0$ to be the disjoint union of geodesic realizations of components of $\mathcal{C}_i.$ 

Fix $i\in \{1,2,3\}.$ Let $P'$ be the Cauchy completion of a connected component of $\HH/G-\Lambda_i^0.$ By construction, we can see that $P'$ is a geometric pair of pants or each possible shape of $P'$ corresponds to one of Figure \ref{fig: all possible types}. Now, the red arcs are a part of ideal boundaries of $\HH/G$, the black arcs are simple bi-infinite geodesics in $\HH/G$ connecting ideal points in $Q_i$ and the small circles are punctures or cone points. But, the black circles have two possibilities. One is that the geodesic is a regular simple closed curve in $\HH/G.$ The other is that the geodesic comes from a geodesic segment connecting two cone points of order two.

Finally, we construct a geodesic lamination in each case of $P'$ so that for any complement $A$ of the lamination in $P',$ the closure of each connected component of $\pi_G^{-1}(A)$ is an ideal polygon. The case where $P'$ is a geometric pair of pants is done by  Lemma \ref{lem : a lamination in a geometric pair of pants} so we consider the cases in Figure \ref{fig: all possible types}. The case of Figure \ref{fig:2-1} is already done by Lemma \ref{lem: a triangulation of a square}. Here, the countable dense subset $Q$ in Lemma \ref{lem: a triangulation of a square} comes from $Q_i.$ Now, we consider the cases of Figure \ref{fig:2-2}  and Figure \ref{fig:2-3}. See Figure \ref{fig:2_blue}. By adding the blue bi-infinite geodesics, the blue geodesics bound half-planes and the blue and black geodesics bound an ideal hexagon and an ideal octagon in each cases. In each half-plane, we can construct a geodesic lamination by Lemma \ref{lem: the half farey triangulation}. The countable dense subset $Q$ in Lemma \ref{lem: the half farey triangulation} also comes from $Q_i.$ Then we are done.  

Then, we consider the case where $P'$ is of the type of Figure \ref{fig:1-1} and the small circle corresponds to a cone point of order two. Choose a point $p$ of $Q_i$ in the interior of the red line. Then there is a unique regular simple geodesic ray emanating from the cone point with the ideal end point $p.$ See Figure \ref{fig:1-1_order two_blue}. The blue line is the geodesic ray. Then, since the completion of the complement of the geodesic ray is of the type of \ref{fig:2-1}, we can similarly construct a geodesic lamination. 

Now,  we consider the case where $P'$ is of the type of Figure \ref{fig:1-1} and the small circle is a cone point of which the order is greater than 2  or a puncture. Choose a point $p$ of $Q_i$ in the interior of the red line. Contrary to the case of oder two, there is a bi-infinite geodesic of which the end points are $p.$ See Figure \ref{fig:1_blue}. The blue line is the geodesic. This blue line divides the bi-gon into a piece of the type of Figure \ref{fig:2-1} and a piece which is an ideal monogon or a monogon with one puncture. Then, in the piece of the type of Figure \ref{fig:2-1}, we can construct a geodesic lamination as the previous cases. If the other piece is an ideal monogon, we are done. If the piece is a monogon with one puncture, we need to add one more bi-infinite geodesic emanating to the puncture as in Figure \ref{fig:monogoncusp}. 

Then, we consider the case of Figure \ref{fig:1-2}. Choose a point $p$ of $Q_i$ in the interior of the red line. Then, there is a blue bi-infinite geodesic of which the end points are $p.$ See Figure \ref{fig:1-2_blue}. The blue line is the geodesic. This blue line divides the bi-gon into a piece of the type of Figure \ref{fig:2-1} and a piece which is a monogon with one hole. In the piece of the type of Figure \ref{fig:2-1}, we can construct a geodesic lamination as the previous cases. In the monogon with one hole, we need to add one more geodesic as in Figure \ref{fig:monogonhole}. Then we are done. 

In the case of Figure \ref{fig:1-3}, there is a unique simple closed geodesic that bounds the funnel. See Figure \ref{fig:1-3_blue}. The blue circle is the geodesic. This line divides the bi-gon into a piece of the type of Figure \ref{fig:1-2} and the funnel. In the piece of the type of Figure \ref{fig:1-2}, we can construct a geodesic lamination by the previous case. We only need to consider the funnel. Choose a point $p$ of $Q_i$ in the closed ideal boundary. Then there is a unique bi-infinite geodesic of which end points are $p.$ See Figure \ref{fig:funnel_orange}. The orange line is the bi-infinite geodesic. Then, the orange line divides the funnel into a mongon with one hole and a half-plane. Therefore, we can construct a geodesic lamination as previous cases.  

For remaining cases, observe that by adding proper bi-infinite geodesic connecting two ideal points in $Q_i$ and dividing $P'$ along the geodesic, we can reduce the construction problem into the previous cases. See Figure \ref{fig:remaining cases_blue}. The blue line is the bi-infinite geodesic which we want. Observe that each geodesic divides $P'$ into pieces in which we can construct geodesic laminations. Therefore, in each case, we can construct a geodesic lamination in $P'.$

Now, by adding a geodesic lamination in each complementary regions of $\Lambda_i^0,$ we can get a geodesic lamination $\Lambda_i$ such that the corresponding laminations system $\LL_i$ to $\pi_G^{-1}(\Lambda_i)$ is very full. Now observe that by construction, for each pair of distinct indices $i$, $j\in \{1,2,3\},$ the set $E(\LL_i)\cap E(\LL_j)$ consists of cusp points of $G.$
Therefore, the collection $\{\LL_1,\LL_2,\LL_3\}$ of $G$-invariant very full lamination systems is a pants-like $\COL_3$ collection. Thus, $G$ is a pants-like $\COL_3$ group.

\begin{figure}
\centering

\begin{subfigure}[b]{0.24\textwidth}
  \includegraphics[width=\textwidth]{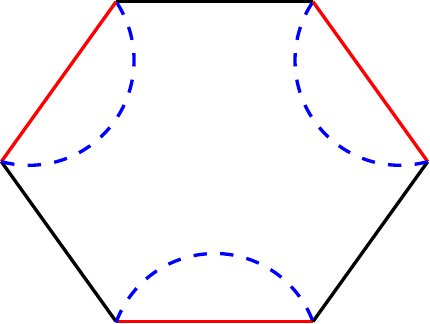}
  \caption{A heatagon }
  \label{fig:2-2_blue}
\end{subfigure}
\hspace*{2.5cm}
\begin{subfigure}[b]{0.2\textwidth}
  \includegraphics[width=\textwidth]{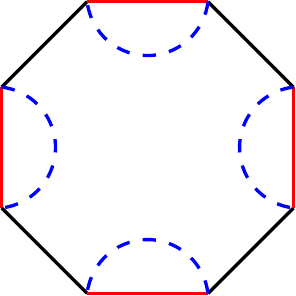}
  \caption{A octagon}
  \label{fig:2-3_blue}
\end{subfigure}
\caption{}
\label{fig:2_blue}
\end{figure}

\begin{figure}
\centering

\begin{subfigure}[b]{0.2\textwidth}
  \includegraphics[width=\textwidth]{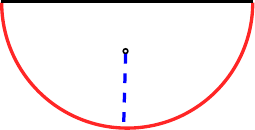}
  \caption{A bi-gon with a cone point of order two}
  \label{fig:1-1_order two_blue}
\end{subfigure}
\hspace*{2.5cm}
\begin{subfigure}[b]{0.2\textwidth}
  \includegraphics[width=\textwidth]{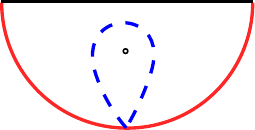}
  \caption{A bi-gon with a puncture or  a cone point which is not of order two}
  \label{fig:1-1_blue}
\end{subfigure}
\hspace*{2.5cm}
\begin{subfigure}[b]{0.2\textwidth}
  \includegraphics[width=\textwidth]{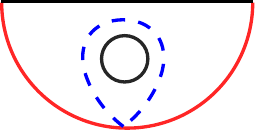}
  \caption{A bi-gone with one geodesic boundary}
  \label{fig:1-2_blue}
\end{subfigure}
\caption{}
\label{fig:1_blue}
\end{figure}

\begin{figure}
\centering

\begin{subfigure}[b]{0.2\textwidth}
  \includegraphics[width=\textwidth]{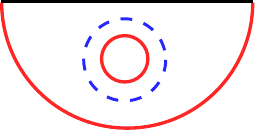}
  \caption{A bi-gon with one closed ideal boundary}
  \label{fig:1-3_blue}
\end{subfigure}
\hspace*{2.5cm}
\begin{subfigure}[b]{0.18\textwidth}
  \includegraphics[width=\textwidth]{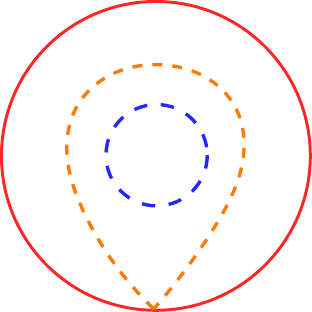}
  \caption{the funnel}
  \label{fig:funnel_orange}
\end{subfigure}

\caption{}

\end{figure}

\begin{figure}
\centering
\begin{subfigure}[b]{0.15\textwidth}
  \includegraphics[width=\textwidth]{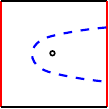}
  \caption{A square with one marked point}

\end{subfigure}
\hspace*{3.3cm}
\begin{subfigure}[b]{0.15\textwidth}
  \includegraphics[width=\textwidth]{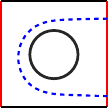}
  \caption{A square with one geodesic boundary}

\end{subfigure}
\hspace*{3.3cm}
\begin{subfigure}[b]{0.15\textwidth}
  \includegraphics[width=\textwidth]{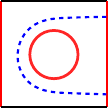}
  \caption{A square with one closed ideal boundary}
	
\end{subfigure}

\begin{subfigure}[b]{0.2\textwidth}
  \includegraphics[width=\textwidth]{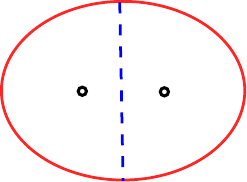}
  \caption{Two marked points and one closed ideal boundary}

\end{subfigure}
\hspace*{2.5cm}
\begin{subfigure}[b]{0.2\textwidth}
  \includegraphics[width=\textwidth]{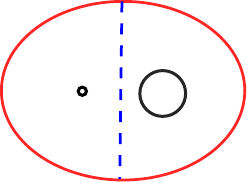}
  \caption{One marked point and one closed ideal boundaries}

\end{subfigure}
\hspace*{2.5cm}
\begin{subfigure}[b]{0.2\textwidth}
  \includegraphics[width=\textwidth]{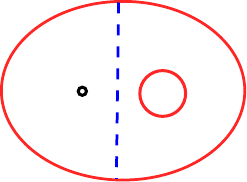}
  \caption{One marked points  and two closed ideal boundaries}

\end{subfigure}

\begin{subfigure}[b]{0.2\textwidth}
  \includegraphics[width=\textwidth]{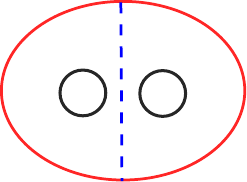}
  \caption{Only one ideal boundary }

\end{subfigure}
\hspace*{2.5cm}
\begin{subfigure}[b]{0.2\textwidth}
  \includegraphics[width=\textwidth]{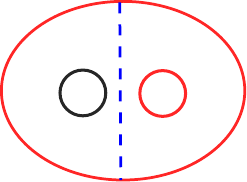}
  \caption{Two ideal boundaries}

\end{subfigure}
\hspace*{2.5cm}
\begin{subfigure}[b]{0.2\textwidth}
  \includegraphics[width=\textwidth]{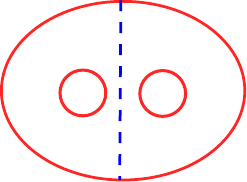}
  \caption{All boundaries are ideal boundaries}

\end{subfigure}

\caption{The blue lines are bi-infinite geodesics connecting two points in $Q_i.$ The marked points which are drawn as the small circles are puncture or cone points.}
\label{fig:remaining cases_blue}
\end{figure}

\begin{thm}\label{thm: non-elementary and of the second kind}
Let $G$ be a non-elementary Fuchsian group of the second kind. Then $G$ is a pants-like $\COL_3$ group.  
\end{thm}

 \begin{theorem}{\ref{B}}
Let $G$ be a Fuchsian group. Suppose that $\HH/G$ is not a geometric pair of pants. Then $G$ is a pants-like $\COL_3$ group.  
 \end{theorem}
 \begin{proof}
 By Theorem \ref{thm: of the first kind}, Theorem \ref{thm:elementary} and Theorem \ref{thm: non-elementary and of the second kind}, we are done.
 \end{proof}

\section{Appendix}
In this appendix, we provide a proof of Lemma  \ref{lem: having the convergence property}. In fact, this lemma is the refinement of a lemma proven in \cite{BaikFuchsian} under the assumption that  $G$ is discrete and torsion-free. Before proving the lemma, we need some lemma and definition.

\begin{lem}\label{lem: existence of angel wingss}
Let $G$ be a laminar group and $\LL$ a $G$-invariant very full lamination system. Assume that a point $p$ in $S^1$ is a cusp point of $G.$ Then there is a sequence $\{U_n\}_{n=1}^\infty$ of nondegenerate open intervals such that for each $n\in \NN,$ $\overline{U_{n+1}}\subset U_n$ and  $U_n=I_n\cup \{p\} \cup J_n$ for some  two disjoint elements $I_n$ and $J_n$ of $\LL$, and $\displaystyle \bigcap_{n=1}^\infty U_n=\{p\}.$ In particular, $\{I_n\}_{n=1}^\infty$ and $\{J_n\}_{n=1}^\infty$ are quasi-rainbows at $p.$
 \end{lem}
 \begin{proof} 
 We can take a parabolic element $g$ of $G$ fixing $p$  as $p$ is a cusp point of $G.$  Also there is a leaf $\ell$ of $\LL$ having $p$ as an end point and we denote the other end point by $q.$ Then since $g$ is a parabolic element, $g(q)\neq q$ and $g(q)\neq p$ so there is an element $I$ in $\ell$ containing the point $g(q).$ Because $g$ is orientation preserving, $g(I)\subsetneq I $ and since $g$ is parabolic, the sequence $\{g^n(I)\}_{n=1}^\infty$ is a quasi-rainbow at $p.$ Likewise $g^{-1}(q)\in I^*$ as $q\in g(I)^*.$ Since $g^{-1}$ is orientation preserving, $g^{-1}(I^*)\subsetneq I^*$ so since $g$ is parabolic, the sequence $\{g^{-n}(I^*)\}_{n=1}
^\infty$ is a quasi-rainbow at $p.$ Now for each $n\in \NN,$ we define $U_n$ as the union of the sets $g^n(I)$, $\{p\}$ and $g^{-n}(I^*).$ Since $I$ and $I^*$ are disjoint and $g^n(I) \subsetneq I$ and $g^{-n}(I^*)\subsetneq I^*$ for all $n\in \NN$, then $\{g^n(I),g^{-n}(I^*)\}$ is a distinct pair for all $n\in \NN$. Hence for each $n\in \NN$, $U_n$ is a nondegenerate open interval because $\partial g^n(I) \cap \partial g^{-n}(I^*)=\{p\},$ and $\overline{U_{n+1}}\subset U_n$ as $g^{n+1}(q)\in g^{n}(I)$ and $g^{-n-1}(q)\in g^{-n}(I^*).$ Finally $\displaystyle \bigcap_{n=1}^\infty U_n =\{p\}$ since the sequences  $\{g^n(I)\}_{n=1}^\infty$ and   $\{g^{-n}(I^*)\}_{n=1}
^\infty$ are quasi-rainbows at $p.$ Thus we are done. 
 \end{proof}
 
 For brevity we make a definition for the sequence $\{U_n\}_{n=1}^\infty$ of the lemma. 
 
 \begin{defn}
 Let $\LL$ be a  lamination system. Let $p$ be a point in $S^1$ and $\{U_n\}_{n=1}^\infty$ be a sequence of nondegenerate open intervals containing $p.$ We call the sequence $\{U_n\}_{n=1}^\infty$  \emph{angel wings} at $p$ if  for each $n\in \NN,$ $\overline{U_{n+1}}\subset U_n$ and  $U_n=I\cup \{p\} \cup J$ for some  two disjoint elements $I$ and $J$ of $\LL$, and $\displaystyle \bigcap_{n=1}^\infty U_n=\{p\}.$
 \end{defn}
 
 So Lemma \ref{lem: existence of angel wingss} says that every cusp point of a group has angel wings if the group has a very full invariant lamination system. Now we prove Lemma  \ref{lem: having the convergence property}.

\begin{lem} \label{lem: having the convergence property}
Let $G$ be a pants-like $\COL_3$ group and $\mathcal{C}=\{\LL_1, \LL_2, \LL_3\}$ be a pants-like collection for $G$.
Suppose that there are two sequences $\{x_n\}_{n=1}^\infty$ and $\{y_n\}_{n=1}^\infty$ in $S^1$ such that $\{x_n\}_{n=1}^\infty$  converges to a point $x$ in $S^1$, $\{y_n\}_{n=1}^\infty$ converges to a point $y$ in $S^1$ and $x\neq y,$ and that $\{g_n\}_{n=1}^\infty$ is an approximation sequence of $G$ such that $\{g_n(x_n)\}_{n=1}^\infty$ converges to a point $x'$ in $S^1$, $\{g_n(y_n)\}_{n=1}^\infty$ converges to a point $y'$ in $S^1$ and $x'\neq y'.$ Further assume that we can take a sequence $\{a_n\}_{n=1}^\infty$ in $S^1$ so that for each $n\in\NN,$ $a_n\in \Fix_{h_n}$ for some hyperbolic element $h_n$ of $G$, $\{a_n\}_{n=1}^\infty$ converges to $x$ and $\{g_n(a_n)\}_{n=1}^\infty$ converges to $y'.$ Then the sequence $\{g_n\}_{n=1}^\infty$ has the convergence property. 
\end{lem}
\begin{proof}
We will show that we can take a subsequence satisfying the condition of the convergence property with respect to $x$ and $y'.$ First we choose a number $\epsilon$ in $\RR$ so that $0<\epsilon<\frac{1}{3}\min\{d_\CC(x,y), d_\CC(x',y')\}.$ Then for each $p\in \{x,y,x',y'\},$ we denote the $\epsilon$-neighborhood of $p$ by $N_\epsilon(p)=\{z\in S^1 : d_\CC(p,z)<\epsilon \}.$  Note that $N_\epsilon(x)\cap N_\epsilon(y)=\emptyset$ and $N_\epsilon(x')\cap N_\epsilon(y')=\emptyset.$  We can take a subindex $\alpha$ of $\id_\NN$ so that for each $n\in \NN$, $\{ x_{\alpha(n)}, a_{\alpha(n)}\} \subset N_\epsilon(x)$, $y_{\alpha(n)}\in N_\epsilon(y)$, $g_{\alpha(n)}(x_{\alpha(n)})\in N_\epsilon(x')$, and $\{ g_{\alpha(n)}(y_{\alpha(n)}), g_{\alpha(n)}(a_{\alpha(n)})\} \subset N_\epsilon(y').$ Now we divide into four cases depending on whether $x$ and $y'$ are cusp points.  

First we consider the case where both $x$ and $y'$ are not cusp points. Then at least two of $\mathcal{C}$ have a rainbow at $x$ and also at least two of $\mathcal{C}$ have a rainbow at $y'.$ So at least one of $\mathcal{C}$ has rainbows at $x$ and $y'$ simultaneously. Hence we can take two rainbows $\{I_n\}_{n=1}^\infty$ at $x$ and $\{J_n\}_{n=1}^\infty$ at $y'$ in some $\LL$ of $\mathcal{C}$ so that for each $\overline{I_{n+1}}\subset I_n$ and $\overline{J_{n+1}}\subset J_n$ for all $n\in \NN$, and $\overline{I_n}\subset N_\epsilon(x)$ and $\overline{J_n}\subset N_\epsilon(y)$ for all $n\in \NN.$ Then for each $n\in \NN$, there is a number $m(n)$ in $\NN$ such that for all $m\in \NN$ with $m(n)\leq m$, $\{ x_{\alpha(m)}, a_{\alpha(m)}\} \subset I_n$ and $\{ g_{\alpha(m)}(y_{\alpha(m)}), g_{\alpha(m)}(a_{\alpha(m)})\} \subset J_n.$ A subindex $\beta$ of $\alpha$ can be given by $\displaystyle \beta(n)=\alpha \big( \sum_{k=1}^n m(k) \big)$ so that for each $n\in \NN,$ $\{ x_{\beta(n)}, a_{\beta(n)}\} \subset I_n$ and $\{ g_{\beta(n)}(y_{\beta(n)}), g_{\beta(n)}(a_{\beta(n)})\} \subset J_n.$ 

Now we show that the sequence $\{g_{\beta(n)}\}_{n=1}^\infty$ has the convergence property with respect to $x$ and $y'$. Let $K$ be a compact subset of $S^1-\{x\}$ and $U$ be an open neighborhood of $y'$ on $S^1.$ Then since $\{I_n\}_{n=1}^\infty$ and $\{J_n\}_{n=1}^\infty$ are  rainbows, there is a number $N$ in $\NN$ such that $K\subset I_N^* $ and $J_N \subseteq U.$ Note that for all $n\in \NN$ with $N\leq n,$ $\{ x_{\beta(n)}, a_{\beta(n)}\} \subset I_N$ and $\{ g_{\beta(n)}(y_{\beta(n)}), g_{\beta(n)}(a_{\beta(n)})\} \subset J_N$ as $I_n\subseteq I_N$ and $J_n\subseteq J_N.$ Choose a number $n_0\in \NN$ with $N\leq n_0.$ Because $N_\epsilon(x)$ does not intersect with $\overline{N_\epsilon(y)}$ and $y_{\beta(n_0)}\in I_N^*$,  $g_{\beta(n_0)}(y_{\beta(n_0)})\in g_{\beta(n_0)}(I_N)^*\cap J_N$ so by the unlinkedness of $g_{\beta(n_0)}(I_N)^*$ and $J_N$, there are three possible cases: $g_{\beta(n_0)}(I_N)^*\subset J_N$; $g_{\beta(n_0)}(I_N)\subset J_N$; $g_{\beta(n_0)}(I_N)\subset J_N^*.$ If $g_{\beta(n_0)}(I_N)\subset J_N$, then $g_{\beta(n_0)}(x_{\beta(n_0)})\in J_N$ and so $g_{\beta(n_0)}(x_{\beta(n_0)})\in N_\epsilon(y').$ However this is a contradiction since $N_\epsilon(x')$ and $N_\epsilon(y')$ are disjoint. If $g_{\beta(n_0)}(I_N)\subset J_N^*$, then $g_{\beta(n_0)}(I_N) \cap J_N=\emptyset$ but $g_{\beta(n_0)}(a_{\beta(n_0)}) \in g_{\beta(n_0)}(I_N) \cap J_N.$ So it is a contradiction. Therefore $g_{\beta(n_0)}(I_N)^*\subset J_N$ and this implies that $g_{\beta(n_0)}(K)\subset g_{\beta(n_0)}(I_N)^*\subset J_N \subset U.$ Thus we can conclude that $\{g_{\beta(n)}\}_{n=1}^\infty$ has the convergence property and so $\{g_{\alpha(n)}\}_{n=1}^\infty$ has the convergence property. 

Next we consider the case where $x$ is not a cusp point and $y'$ is a cusp point. Since $x$ is not a cusp point,  at most one of $\mathcal{C}$ has $x$ as an end point so there are at lest two of $\mathcal{C}$ have rainbows at $x.$ On the other hand, $y'$ has angel wings in each of $\mathcal{C}.$ So there is a lamination system $\LL$ in $\mathcal{C}$ which has a rainbow at $x$ and angel wings at $y'.$ Hence, in some $\LL$ in $\mathcal{C},$ we can take a rainbow $\{I_n\}_{n=1}^\infty$ at $x$ and angel wings $\{U_n\}_{n=1}^\infty$ at $y'$ so that for each $n\in \NN,$ $\overline{I_{n+1}} \subset I_n$, $\overline{I_n} \subset N_\epsilon(x)$ and $\overline{U_n} \subset N_\epsilon(y').$ Then for each $n\in \NN,$ there is a number $m(n)$ in $\NN$ such that for all $m\in \NN$ with $m(n)\leq m$, $\{ x_{\alpha(m)}, a_{\alpha(m)}\} \subset I_n$ and $\{ g_{\alpha(m)}(y_{\alpha(m)}), g_{\alpha(m)}(a_{\alpha(m)})\} \subset U_n.$ A subindex $\beta$ of $\alpha$ can be given by $\displaystyle \beta(n)=\alpha \big( \sum_{k=1}^n m(k) \big)$ so that for each $n\in \NN,$ $\{ x_{\beta(n)}, a_{\beta(n)}\} \subset I_n$ and $\{ g_{\beta(n)}(y_{\beta(n)}), g_{\beta(n)}(a_{\beta(n)})\} \subset U_n.$ 

Similarly we show that the sequence $\{g_{\beta(n)}\}_{n=1}^\infty$ has the convergence property with respect to $x$ and $y'$. Let $K$ be a compact subset of $S^1-\{x\}$ and $U$ be an open neighborhood of $y'$ on $S^1.$ Then since $\{I_n\}_{n=1}^\infty$ is a rainbow and $\{U_n\}_{n=1}^\infty$ is angel wings, there is a number $N$ in $\NN$ such that $K\subset I_N^* $ and $U_N \subseteq U.$ Note that for all $n\in \NN$ with $N\leq n,$ $\{ x_{\beta(n)}, a_{\beta(n)}\} \subset I_N$ and $\{ g_{\beta(n)}(y_{\beta(n)}), g_{\beta(n)}(a_{\beta(n)})\} \subset U_N$ as $I_n\subseteq I_N$ and $U_n\subseteq U_N.$  Now we write $I_n=(r_n,s_n)_{S^1}$ and $U_n=(u_n, v_n)_{S^1}$ for all $n\in \NN.$ Note that $(u_n,y')_{S^1}$ and $(y',v_n)_{S^1}$ are in $\LL$ for all $n\in \NN$. Now, we choose a number $n_0\in \NN$  with $N<n_0.$ Then $r_{n_0}\in (y_{\beta(n_0)},a_{\beta(n_0)})_{S^1}$ and $s_{n_0}\in (a_{\beta(n_0)},y_{\beta(n_0)})_{S^1}$ as $a_{\beta(n_0)}\in (r_{n_0}, s_{n_0})_{S^1}$ and $y_{\beta(n_0)}\in (s_{n_0},r_{n_0})_{S^1}$ by the choice of $\epsilon.$ So one of $ g_{\beta(n_0)}( [y_{\beta(n_0)},a_{\beta(n_0)}]_{S^1})$ and $g_{\beta(n_0)}([a_{\beta(n_0)},y_{\beta(n_0)}]_{S^1})$ must be contained in $U_{n_0}$ because $\{g_{\beta(n_0)}(y_{\beta(n_0)}),g_{\beta(n_0)}(a_{\beta(n_0)})\} \subset U_{n_0}.$ 

%Hence at least one point of $\partial I_{n_0}$ belongs to $U_n.$
First we consider the case where  $   [g_{\beta(n_0)}(y_{\beta(n_0)}),g_{\beta(n_0)}(a_{\beta(n_0)})]_{S^1}= g_{\beta(n_0)}([y_{\beta(n_0)},a_{\beta(n_0)}]_{S^1}) \subset U_{n_0}.$ If $g_{\beta(n_0)}(r_{n_0})\in (u_{n_0},y')_{S^1},$ then by the unlinkedness of $g_{\beta(n_0)}(I_{n_0})$ and $(u_{n_0},y')_{S^1},$ there are two possible cases: $g_{\beta(n_0)}(I_{n_0})\subset (u_{n_0},y')_{S^1}$; $g_{\beta(n_0)}(I_{n_0})^* \subset (u_{n_0},y')_{S^1}.$  However if  $g_{\beta(n_0)}(I_{n_0})\subset (u_{n_0},y')_{S^1},$ then $g_{\beta(n_0)}(x_{\beta(n_0)}) \in (u_{n_0},y')_{S^1}\subset U_{n_0} \subset N_\epsilon(y')$ so this is a contradiction. Therefore $g_{\beta(n_0)}(I_{n_0})^* \subset (u_{n_0},y')_{S^1}$ and so 
$$g_{\beta(n_0)}(K)\subset g_{\beta(n_0)}(I_{N})^* \subset g_{\beta(n_0)}(I_{n_0})^* \subset (u_{n_0},y')_{S^1}\subset  U_{n_0} \subset U_N \subset U.$$
Then if $g_{\beta(n_0)}(r_{n_0})\in (y',v_{n_0})_{S^1},$ then by the unlinkedness of $g_{\beta(n_0)}(I_{n_0})$ and $(y',v_{n_0})_{S^1},$ there are two possible cases: $g_{\beta(n_0)}(I_{n_0})\subset (y',v_{n_0})_{S^1}$; $g_{\beta(n_0)}(I_{n_0})^* \subset (y',v_{n_0})_{S^1}.$ However if  $g_{\beta(n_0)}(I_{n_0})\subset (y',v_{n_0})_{S^1},$ then $g_{\beta(n_0)}(x_{\beta(n_0)}) \in (y', v_{n_0})_{S^1}\subset U_{n_0} \subset N_\epsilon(y')$ so this is a contradiction. Therefore $g_{\beta(n_0)}(I_{n_0})^* \subset (y',v_{n_0})_{S^1}$ and so 
$$g_{\beta(n_0)}(K)\subset g_{\beta(n_0)}(I_{N})^* \subset g_{\beta(n_0)}(I_{n_0})^* \subset (y', v_{n_0})_{S^1}\subset  U_{n_0} \subset U_N \subset U.$$
Finally we assume that $g_{\beta(n_0)}(r_{n_0})=y'.$ Note that $g_{\beta(n_0)}(r_N)\in  (g_{\beta(n_0)}(y_{\beta(n_0)}),g_{\beta(n_0)}(a_{\beta(n_0)}))_{S^1}$ since  $a_{\beta(n_0)}\in I_N$ and $y_{\beta(n_0)}\in I_N^*$ by the choice of $\epsilon$, and so $g_{\beta(n_0)}(r_N)\in (g_{\beta(n_0)}(y_{\beta(n_0)}), y')_{S^1}\subset (u_{n_0}, y')_{S^1}$ since 
$$[y',g_{\beta(n_0)}( a_{\beta(n_0)}))_{S^1}=g_{\beta(n_0)}([r_{n_0}, a_{\beta(n_0)}))_{S^1}\subset g _{\beta(n_0)}(\overline{I_{n_0}})\subset  g _{\beta(n_0)}(I_N).$$ Hence by the unlinkedness of $g_{\beta(n_0)}(I_N)$ and $(u_{n_0}, y')_{S^1},$ there are two possible cases: $g_{\beta(n_0)}(I_N) \subset (u_{n_0}, y')_{S^1}$; $g_{\beta(n_0)}(I_N)^* \subset (u_{n_0}, y')_{S^1}$. However if $g_{\beta(n_0)}(I_N) \subset (u_{n_0}, y')_{S^1}$, then $g_{\beta(n_0)}(x_{\beta(n_0)}) \in (u_{n_0},y')_{S^1}\subset U_{n_0} \subset N_\epsilon(y')$ so this is a contradiction. Therefore $g_{\beta(n_0)}(I_N)^* \subset (u_{n_0}, y')_{S^1}$ and so 
$$g_{\beta(n_0)}(K)\subset g_{\beta(n_0)}(I_{N})^* \subset (u_{n_0},y')_{S^1}\subset  U_{n_0} \subset U_N \subset U.$$
Thus these implies that $\{g_{\beta(n)}\}_{n=1}^\infty$ has the convergence property and so $\{g_{\alpha(n)}\}_{n=1}^\infty$ has the convergence property. In the similar way we can prove that $\{g_{\beta(n)}\}_{n=1}^\infty$ has the convergence property in the case where $g_{\beta(n_0)}([a_{\beta(n_0)},y_{\beta(n_0)}]_{S^1})\subset U_{n_0}.$ Hence $\{g_{\alpha(n)}\}_{n=1}^\infty$ also has the convergence property.

Now we consider  the case where $x$ is a cusp point and $y'$ is not a cusp point. It follows at once from the previous case that $\{g_{\alpha(n)}^{-1}\}_{n=1}^\infty$ has the convergence property. Indeed it implies that $\{g_{\alpha(n)}\}_{n=1}^\infty$ has the convergence property  by Remark \ref{rmk: the convergence property for inverse}. Thus this case is also proved. 

Finally we consider the most complicated case when both $x$ and $y'$ are cusp points. Note that by Lemma \ref{lem: disjoint fixed point sets}, $a_n\neq x$ and $g_n(a_n)\neq y'$ for all $n\in\NN.$ Also we claim that there is at most one number $n$ in $\NN$ such that $g_n(x)=y'.$ If there are two numbers $n$ and $m$ in $\NN$ with $n \neq m$ such that $g_n(x)=y'$ and $g_m(x)=y',$ then 
$g_m\circ g_n^{-1}(y')=g_m(x)=y'$ and $y'\in \Fix_{g_m\circ g_n^{-1}}.$ However $g_m\circ g_n^{-1}$ is hyperbolic as $\{g_n\}_{n=1}^\infty$ is an approximation sequence and so it is a contradiction by  Lemma \ref{lem: disjoint fixed point sets}. Therefore there is a subindex $\beta$ of $\alpha$ such that $g_{\beta(n)}(x)\neq y'$ for all $n\in\NN.$

 Choose a lamination system $\LL$ in $\mathcal{C}.$ Then we can take two angel wings $\{I_n\}_{n=1}^\infty$ at $x$ and $\{J_n\}_{n=1}^\infty$ at $y'$ in $\LL$ so that for each $n\in\NN,$ $\overline{I_n} \subset N_\epsilon(x)$ and $\overline{J_n} \subset N_\epsilon(y').$ Then for each $n\in \NN,$ there is a number $m(n)$ in $\NN$ such that for all $m\in \NN$ with $m(n)\leq m$, $\{ x_{\beta(m)}, a_{\beta(m)}\} \subset I_n$ and $\{ g_{\beta(m)}(y_{\beta(m)}), g_{\beta(m)}(a_{\beta(m)})\} \subset J_n.$ A subindex $\gamma$ of $\beta$ can be given by $\displaystyle \gamma(n)=\beta \big( \sum_{k=1}^n m(k) \big)$ so that for each $n\in \NN,$ $\{ x_{\gamma(n)}, a_{\gamma(n)}\} \subset I_n$ and $\{ g_{\gamma(n)}(y_{\gamma(n)}), g_{\gamma(n)}(a_{\gamma(n)})\} \subset J_n.$ 

Now we write $I_n=(r_n,s_n)_{S^1}$ and $J_n=(u_n,v_n)_{S^1}$ for all $n\in \NN.$ Note that for each $n\in\NN,$ the nondegenerate open intervals $(r_n,x)_{S^1},(x,s_n)_{S^1}, (u_n,y')_{S^1},$ and $(y',v_n)_{S^1}$ are in $\LL.$ Then we show that the sequence $\{g_{\gamma(n)}\}_{n=1}^\infty$ has the convergence property with respect to $x$ and $y'$. Let $K$ be a compact subset of $S^1-\{x\}$ and $U$ be an open neighborhood of $y'$ on $S^1.$ Then since $\{I_n\}_{n=1}^\infty$ and $\{J_n\}_{n=1}^\infty$ are two angel wings, there is a number $N$ in $\NN$ such that $K\subset I_N^* $ and $J_N \subset U.$ Note that for all $n\in \NN$ with $N\leq n,$ $\{ x_{\gamma(n)}, a_{\gamma(n)}\} \subset I_N$ and $\{ g_{\gamma(n)}(y_{\gamma(n)}), g_{\gamma(n)}(a_{\gamma(n)})\} \subset J_N$ as $I_n\subset I_N$ and $J_n\subset J_N.$ 

We fix a number $n_0\in \NN$  with $N<n_0.$ We claim that $g_{\gamma(n_0)}(K)\subset U.$ Note that one of $g_{\gamma(n_0)}([a_{\gamma(n_0)}, y_{\gamma(n_0)}]_{S^1})$ and $g_{\gamma(n_0)}([y_{\gamma(n_0)}, a_{\gamma(n_0)}]_{S^1})$ must be contained in $J_{n_0}$ as $\{g_{\gamma(n_0)}(a_{\gamma(n_0)}), g_{\gamma(n_0)}(y_{\gamma(n_0)})\}\subset J_{n_0}.$  There are two cases depending on the location of $a_{\gamma(n_0)}:$ $a_{\gamma(n_0)}\in (r_{n_0},x)_{S^1};$ $a_{\gamma(n_0)} \in (x,s_{n_0})_{S^1}.$ 

First we consider the case where  $a_{\gamma(n_0)}\in (r_{n_0},x)_{S^1}.$ Then $y_{\gamma(n_0)}\in (x,r_{n_0})_{S^1}$ by the choice of $\epsilon.$ Hence one of $g_{\gamma(n_0)}(r_{n_0})$ and $g_{\gamma(n_0)}(x)$ belongs to $J_{n_0}$ since $x\in (a_{\gamma(n_0)}, y_{\gamma(n_0)})_{S^1}$ and $r_{n_0}\in (y_{\gamma(n_0)}, a_{\gamma(n_0)})_{S^1}.$ 

In the case where  $g_{\gamma(n_0)}(r_{n_0})\in J_{n_0},$ there are three possible cases: $g_{\gamma(n_0)}(r_{n_0})\in (u_{n_0}, y')_{S^1};$ $g_{\gamma(n_0)}(r_{n_0})=y';$ $g_{\gamma(n_0)}(r_{n_0})\in (y', v_{n_0})_{S^1}.$

 If  $g_{\gamma(n_0)}(r_{n_0})\in (u_{n_0}, y')_{S^1},$ then
$g_{\gamma(n_0)}((r_{n_0}, x)_{S^1})\subset (u_{n_0}, y')_{S^1}$ or $g_{\gamma(n_0)}((r_{n_0}, x)_{S^1})^*\subset (u_{n_0}, y')_{S^1}$ by the unlinkedness of $g_{\gamma(n_0)}((r_{n_0}, x)_{S^1})$ and $(u_{n_0}, y')_{S^1}.$ When $g_{\gamma(n_0)}((r_{n_0}, x)_{S^1})\subset (u_{n_0}, y')_{S^1},$ we get that  $g_{\gamma(n_0)}(x)\in (u_{n_0}, y')_{S^1} $ since $g_{\gamma(n_0)}(x)\neq y'$, and that $g_{\gamma(n_0)}(\ell((x, s_{n_0})_{S^1}))$ lies on $(u_{n_0}, y')_{S^1}.$  However if $g_{\gamma(n_0)}((x, s_{n_0})_{S^1})\subset (u_{n_0}, y')_{S^1},$ then $g_{\gamma(n_0)}(I_{n_0})=g_{\gamma(n_0)}((r_{n_0}, s_{n_0})_{S^1}) \subset (u_{n_0}, y')_{S^1}\subset J_{n_0} $ so this is a contradiction since  $g_{\gamma(n_0)}(x_{\gamma(n_0)})\notin J_{n_0}.$ Therefore $g_{\gamma(n_0)}(( s_{n_0}, x)_{S^1})\subset (u_{n_0}, y')_{S^1}.$ Then 
$$g_{\gamma(n_0)}(K)\subset g_{\gamma(n_0)}(I_N^*) \subset g_{\gamma(n_0)}(I_{n_0}^*)\subset g_{\gamma(n_0)}(( s_{n_0}, x)_{S^1})\subset (u_{n_0}, y')_{S^1} \subset J_{n_0} \subset J_N \subset U.$$  So this case is proved. Then when  $g_{\gamma(n_0)}((r_{n_0}, x)_{S^1})^*\subset (u_{n_0}, y')_{S^1}$, we can see that 
$$g_{\gamma(n_0)}(K)\subset g_{\gamma(n_0)}(I_N^*) \subset g_{\gamma(n_0)}(I_{n_0}^*)\subset g_{\gamma(n_0)}((r_{n_0}, x)_{S^1}^*)\subset (u_{n_0}, y')_{S^1} \subset J_{n_0} \subset J_N \subset U.$$ Therefore the case where $g_{\gamma(n_0)}(r_{n_0})\in (u_{n_0}, y')_{S^1}$ is done. 

Now if $g_{\gamma(n_0)}(r_{n_0})=y',$ then there are three cases by the unlinkeness of $g_{\gamma(n_0)}((r_{n_0}, x)_{S^1})$, $(u_{n_0}, y')_{S^1}$ and $(y', v_{n_0})_{S^1}:$ $g_{\gamma(n_0)}((r_{n_0}, x)_{S^1})\subsetneq (y', v_{n_0})_{S^1};$ $(y',v_{n_0})_{S^1}\subset g_{\gamma(n_0)}((r_{n_0}, x)_{S^1}) \subsetneq (u_{n_0}, y')_{S^1}^*;$ $ (u_{n_0}, y')_{S^1}^* \subset g_{\gamma(n_0)}((r_{n_0}, x)_{S^1}).$

We claim that the first case is not possible. If $g_{\gamma(n_0)}((r_{n_0}, x)_{S^1})\subsetneq (y', v_{n_0})_{S^1},$ then $x\in (y', v_{n_0})_{S^1}$ since $g_{\gamma(n_0)}(r_{n_0})=y'.$ Then $g_{\gamma(n_0)}(\ell((x, s_{n_0})_{S^1}))$ must lie on $ (y', v_{n_0})_{S^1}.$ However if $g_{\gamma(n_0)}((x, s_{n_0})_{S^1}) \subset  (y', v_{n_0})_{S^1},$ then 
$$g_{\gamma(n_0)}(x_{\gamma(n_0)})\in g_{\gamma(n_0)}(I_{n_0})=g_{\gamma(n_0)}((r_{n_0}, s_{n_0})_{S^1}) \subset  (y', v_{n_0})_{S^1} \subset J_{n_0}\subset N_\epsilon(y').$$ This is a contradiction by the choice of $\epsilon.$ Also if  $g_{\gamma(n_0)}((x, s_{n_0})_{S^1})^* \subset  (y', v_{n_0})_{S^1},$ then  $$(y',g_{\gamma(n_0)}(x))_{S^1}=g_{\gamma(n_0)}(( r_{n_0}, x)_{S^1})\subset g_{\gamma(n_0)}((x, s_{n_0})_{S^1}^*)=(g_{\gamma(n_0)}( s_{n_0}), g_{\gamma(n_0)}(x))_{S^1} \subset  (y', v_{n_0})_{S^1},$$
and this implies that $g_{\gamma(n_0)}( s_{n_0})=y'.$ However $g_{\gamma(n_0)}( s_{n_0})=y'=g_{\gamma(n_0)}(r_{n_0})$ so this is a contradiction. Therefore the claim is proved. 

Now we consider the case where $(y',v_{n_0})_{S^1}\subset g_{\gamma(n_0)}((r_{n_0}, x)_{S^1}) \subsetneq (u_{n_0}, y')_{S^1}^*=(y', u_{n_0})_{S^1}.$ This case is also impossible. Observe that $g_{\gamma(n_0)}(x)\in [v_{n_0}, u_{n_0})_{S^1}$ and $ g_{\gamma(n_0)}(y_{\gamma(n_0)})$ must belong to $(u_{n_0}, y')_{S^1}$ in this case. 
Since $y'=g_{\gamma(n_0)}(r_{n_0}) \subset g_{\gamma(n_0)}( (r_N,x)_{S^1}),$ the leaf $\ell((u_{n_0}, y')_{S^1})$ lies on  $g_{\gamma(n_0)}( (r_N,x)_{S^1}).$ If $(u_{n_0}, y')_{S^1}^* \subset g_{\gamma(n_0)}( (r_N,x)_{S^1}),$ then 
$(g_{\gamma(n_0)}(x), g_{\gamma(n_0)}(r_N))_{S^1}=g_{\gamma(n_0)}((r_N, x)_{S^1})^*\subset (u_{n_0}, y')_{S^1}$ and 
$g_{\gamma(n_0)}(x)\in [u_{n_0}, y')_{S^1}.$ This is a contradiction since $[v_{n_0}, u_{n_0})_{S^1}$ and $ [u_{n_0}, y')_{S^1}$ are disjoint. If $(u_{n_0}, y')_{S^1} \subset g_{\gamma(n_0)}( (r_N,x)_{S^1}),$ then $ g_{\gamma(n_0)}(y_{\gamma(n_0)})\in (u_{n_0}, y')_{S^1} \subset g_{\gamma(n_0)}( (r_N,x)_{S^1})$ and this is also a contradiction since $y_{\gamma(n_0)}\notin (r_N,x)_{S^1}.$

Therefore  $ (u_{n_0}, y')_{S^1}^* \subset g_{\gamma(n_0)}((r_{n_0}, x)_{S^1}).$ Then

$$g_{\gamma(n_0)}(K) \subset g_{\gamma(n_0)}(I_{n_0}^*) \subset  g_{\gamma(n_0)}((r_{n_0}, x)_{S^1}^*)
\subset (u_{n_0}, y')_{S^1}\subset J_{n_0}\subset J_N \subset U. $$
Thus the case where  $g_{\gamma(n_0)}(r_{n_0})=y'$ is also proved. 

Then we consider the case where $g_{\gamma(n_0)}(r_{n_0})\in (y', v_{n_0})_{S^1}.$ Then there are two cases by the unlinkedness of $g_{\gamma(n_0)}((r_{n_0}, x)_{S^1})$ and $(y', v_{n_0})_{S^1}:$ $g_{\gamma(n_0)}((r_{n_0}, x)_{S^1})\subset (y', v_{n_0})_{S^1};$ $g_{\gamma(n_0)}((r_{n_0}, x)_{S^1})^*\subset (y', v_{n_0})_{S^1}.$ If  $g_{\gamma(n_0)}((r_{n_0}, x)_{S^1})\subset (y', v_{n_0})_{S^1},$ then $g_{\gamma(n_0)}(x)\in (y', v_{n_0}]_{S^1}\subset (y', v_N)_{S^1}$ and so the leaf $g_{\gamma(n_0)}(\ell((x,s_{n_0})_{S^1}))$ must lie on $(y', v_N)_{S^1}.$ However if $g_{\gamma(n_0)}((x,s_{n_0})_{S^1})\subset (y', v_N)_{S^1},$ then
$g_{\gamma(n_0)}((r_{n_0},s_{n_0})_{S^1})\subset (y', v_N)_{S^1}$ and this is a contradiction since $g_{\gamma(n_0)}(x_{\gamma(n_0)})\notin J_N.$ Therefore  $g_{\gamma(n_0)}((x,s_{n_0})_{S^1})^*\subset (y', v_N)_{S^1}$ and so
$$g_{\gamma(n_0)}(K)\subset g_{\gamma(n_0)}(I_N^*) \subseteq  g_{\gamma(n_0)}((x,s_{n_0})_{S^1}^*)
\subset (y', v_N)_{S^1} \subset J_N \subset U.$$
Hence the case where $g_{\gamma(n_0)}(r_{n_0})\in (y', v_{n_0})_{S^1}$ is also proved. Thus the case where $g_{\gamma(n_0)}(r_{n_0})\in J_{n_0}$ is proved. 

Now we consider the case where $g_{\gamma(n_0)}(x)\in J_{n_0}.$ Then $g_{\gamma(n_0)}(x)\in (u_{n_0},y')_{S^1}$ or $g_{\gamma(n_0)}(x)\in (y', v_{n_0})_{S^1}$ as $g_{\gamma(n_0)}(x)\neq y'.$ For brevity, we write $L=g_{\gamma(n_0)}((r_{n_0}, x)_{S^1})$ and $R=g_{\gamma(n_0)}((x,s_{n_0})_{S^1}).$ If  $g_{\gamma(n_0)}(x)\in (u_{n_0},y')_{S^1},$ then the leaves $\ell(L)$ and $\ell(R)$ must lie on $(u_{n_0},y')_{S^1}.$ However if $L\subset (u_{n_0}, y')_{S^1}$ and $R\subset (u_{n_0}, y')_{S^1},$ then 
$g_{\gamma(n_0)}(I_{n_0})\subset (u_{n_0}, y')_{S^1}$ and this is a contradiction since $g_{\gamma(n_0)}(x_{\gamma(n_0)})\notin J_{n_0}.$ Hence $L^*\subset (u_{n_0},y')_{S^1}$ or $R^* \subset (u_{n_0}, y')_{S^1}.$ In both cases, $g_{\gamma(n_0)}(I_{n_0}^*)\subset (u_{n_0}, y')_{S^1}.$ Therefore 
$$g_{\gamma(n_0)}(K) \subset g_{\gamma(n_0)}(I_N^*) \subset g_{\gamma(n_0)}(I_{n_0}^*)\subset (u_{n_0}, y')_{S^1}\subset J_{n_0}\subset J_N \subset U.$$
Similarly if  $g_{\gamma(n_0)}(x)\in (y', v_{n_0})_{S^1},$ then the leaves $\ell(L)$ and $\ell(R)$ must lie on $(y', v_{n_0})_{S^1}.$ However if $L\subset (y', v_{n_0})_{S^1}$ and $R\subset (y', v_{n_0})_{S^1},$ then 
$g_{\gamma(n_0)}(I_{n_0})\subset (y', v_{n_0})_{S^1}$ and this is a contradiction since $g_{\gamma(n_0)}(x_{\gamma(n_0)})\notin J_{n_0}.$ Hence $L^*\subset(y', v_{n_0})_{S^1}$ or $R^* \subset(y', v_{n_0})_{S^1}.$ In both cases $g_{\gamma(n_0)}(I_{n_0}^*)\subset (y', v_{n_0})_{S^1}.$ Therefore 
$$g_{\gamma(n_0)}(K) \subset g_{\gamma(n_0)}(I_N^*) \subset g_{\gamma(n_0)}(I_{n_0}^*)\subset (y', v_{n_0})_{S^1}\subset J_{n_0}\subset J_N \subset U.$$
Thus  the case where $g_{\gamma(n_0)}(x)\in J_{n_0}$ also follows and so the case where $a_{n_0}\in (r_{n_0}, x)_{S^1}$ is done. Similarly $g_{\gamma(n_0)}(K) \subset U$ when  $a_{n_0}\in (x, s_{n_0})_{S^1}.$ This implies that for any $n\in \NN$ with $N<n$, 
$g_{\gamma(n)}(K)\subset U.$ Therefore the sequence $\{g_{\gamma(n)}\}_{n=1}^\infty$ has the convergence property and so  $\{g_{\alpha(n)}\}_{n=1}^\infty$ has the convergence property. Thus the original sequence $\{g_n\}_{n=1}^\infty$ has the convergence property.

\end{proof}

\newpage

\bibliography{biblio}
\bibliographystyle{alpha}

\end{document}